\documentclass[11pt]{amsart}

\usepackage{amsmath, amssymb}
\usepackage{amsthm}

\usepackage[all]{xy}
\usepackage{color}
 \makeatletter
    
    \@addtoreset{equation}{section}
  \makeatother

\def\@begintheorem#1#2{%
   \trivlist
   \item[\hskip \labelsep{\bfseries #1\ #2.}]%
   \itshape}
\def\@opargbegintheorem#1#2#3{%
   \trivlist
   \item[\hskip \labelsep{\bfseries #1\ #2.\ (#3)}]%
   \itshape}
\makeatother

 \makeatletter
    
    \@addtoreset{equation}{section}
  \makeatother

\theoremstyle{plain}
\newtheorem{thm}{Theorem}[subsection]

\newtheorem{lem}[thm]{Lemma}
\newtheorem{prop}[thm]{Proposition}
\newtheorem{cor}[thm]{Corollary}

\theoremstyle{remark}
\newtheorem{rem}[thm]{Remark}

\theoremstyle{definition}
\newtheorem{dfn}[thm]{Definition}

\theoremstyle{plain}
\newtheorem{theor}{Theorem}%[subsection]

\newtheorem*{claim}{Claim}
\newtheorem{nclaim}{Claim}

\newcommand{\QLS}{{\rm QLS}}
\newcommand{\B}{{\mathcal{B}}}
\newcommand{\QB}{{\mathcal{QB}}}

\newcommand{\QBG}{{\rm QBG}}
\newcommand{\qwt}{{\rm qwt}}
\newcommand{\ed}{{\rm end}}
\newcommand{\id}{{e}}

\newcommand{\gch}{{\rm gch}}
\newcommand{\wt}{{\rm wt}}
\newcommand{\dg}{{\rm deg}}
\newcommand{\Dg}{{\rm Deg}}
\newcommand{\dr}{{\rm dir}}

\newcommand{\OS}{{\rm OS}}

\newcommand{\aff}{{\rm aff}}
\newcommand{\cl}{{\rm cl}}

\newcommand{\half}{{\frac{1}{2}}}
\newcommand{\eqdef}{\overset{\rm def}{=}}
\newcommand{\ardef}{\overset{\rm def}{\Leftrightarrow}}
\newcommand{\si}{\frac{\infty}{2}}
\newcommand{\adj}{{\rm -adj}}
\newcommand{\SB}{{\rm SiB}}
\newcommand{\SILS}{\rm SiLS}
\newcommand{\bqed}{\quad \hbox{\rule[-0.5pt]{6pt}{6pt}}  \vspace{3mm}}
\newcommand{\lon}{w_\circ}
\newcommand{\lons}{w_\circ (S)}
\newcommand{\setzero}{ \{ {\bf 0 } \} }
\newcommand{\zero}{ {\bf 0 } }
\newcommand{\C}{{A_{2n}^{(2)}}}
\newcommand{\D}{{D_{n+1}^{(2)}}}
\newcommand{\tra}{t (\lambda_- ) }
\newcommand{\inv}{^{-1}}
\newcommand{\inte}{{\rm int}}
\newcommand{\appr}{\Delta^{+}_{S, \aff}}
\newcommand{\apr}{\Delta_{S, \aff}}
\newcommand{\pcr}{W^{S}_{\aff}}
\newcommand{\apsg}{W_{S, \aff}}
\newcommand{\pdag}{^{\dag}}

\newcommand{\bp}{\mathbf{p}}
\newcommand{\bq}{\mathbf{q}}

\newcommand{\pair}[2]{\langle #1,\,#2 \rangle}

\newenvironment{enu}{%
 \begin{enumerate}%
}{\end{enumerate}}

\title[The specialization of Macdonald polynomials at $t=0$
in type $\C$]{Quantum Lakshmibai-Seshadri paths and the specialization of Macdonald polynomials at $t=0$
in type $\C$.}

\author[F.~Nomoto]{Fumihiko Nomoto}
\address[Fumihiko Nomoto]
 {Department of Mathematics, Tokyo Institute of Technology,
  2-12-1 Oh-Okayama, Meguro-ku, Tokyo 152-8551, Japan}
\email{nomoto.f.aa@m.titech.ac.jp}

\begin{document}

%\title{An explicit formula for the specialization of 
%nonsymmetric Macdonald polynomials at $t=\infty$ }
%\date{Algebraic Lie Theory and Representation Theory 2015, 2015年6月6日}
%\author{野本 文彦\footnote{東京工業大学理工学研究科数学専攻博士課程1年}}

\maketitle
\begin{abstract}
In this paper, we give 
a combinatorial realization of the crystal basis of a quantum Weyl module over a quantum affine algebra of type $\C$,
and
a representation-theoretic interpretation
of the specialization $P_{\lambda}^\C (q,0)$ of the symmetric Macdonald polynomial $P_{\lambda}^\C(q,t)$ at $t=0$, 
where $\lambda$ is a dominant weight
and $P_{\lambda}^\C(q,t)$ denotes the specific specialization of the symmetric Macdonald-Koornwinder polynomial
$P_{\lambda}(q,t_1, t_2, t_3, t_4, t_5)$.
More precisely, 
as some results for untwisted affine types,
the set of all $\C$-type quantum Lakshmibai-Seshadri paths of shape $\lambda$, which is described in terms of 
the finite Weyl group $W$, 
realizes the crystal basis of a quantum Weyl module over a quantum affine algebra of type $\C$
and its graded character is equal to the specialization $P_{\lambda}^\C (q,0)$ of the symmetric Macdonald-Koornwinder polynomial.
\end{abstract}

\section{Introduction}
Symmetric Macdonald polynomials %(in affine type A) 
with two parameters $q$ and $t$ were introduced 
by Macdonald
 \cite{M1} as a family of orthogonal symmetric polynomials,
which include as special or limiting cases almost all the classical families of orthogonal symmetric polynomials.
This family of polynomials are characterized in terms of the double affine Hecke algebra (DAHA) introduced by Cherednik (\cite{C1}, \cite{C3}).
In fact, there exists another family of orthogonal polynomials, called nonsymmetric Macdonald polynomials, which are simultaneous eigenfunctions of $Y$-operators 
acting on the polynomial representation of the DAHA; 
by 
``symmetrizing''
nonsymmetric Macdonald polynomials, we obtain symmetric Macdonald polynomials (see \cite{M}).

%%%%%%%%%%%%1
Based on the characterization above of nonsymmetric Macdonald polynomials,
Ram-Yip \cite{RY} obtained a combinatorial formula expressing symmetric or nonsymmetric Macdonald polynomials associated to an arbitrary untwisted affine
root system; this formula is described in terms of alcove walks, which are certain strictly combinatorial objects.
%%%%%%%%%%%%1
In addition, Orr-Shimozono \cite{OS} refined the Ram-Yip formula above,
and generalized it to an arbitrary affine root system
(including the twisted case);
also, they specialized their formula at $t=0$, $t=\infty$, $q=0$, and $q=\infty$.

As for representation-theoretic interpretations of the specialization of symmetric or nonsymmetric Macdonald polynomials at $t=0$,
%
%
%
%
%
%
%
%Concerning the second question, 
we know the following.
Ion \cite{I} proved that for a dominant integral weight $\lambda$ and an element $x$ of the finite Weyl group $W$, 
%%%%%%%%%%%%2
the specialization $E_{x \lambda} (q, 0)$ of the nonsymmetric Macdonald polynomial $E_{x\lambda} (q,t)$ at $t=0$
is equal to
%%%%%%%%%%%%2
the graded character of a certain Demazure submodule of an irreducible highest weight module over
an affine Lie algebra of a dual untwisted type. %equals the specialization $E_{x\lambda}(q, 0)$ of the nonsymmetric Macdonald polynomial $E_{x\lambda}(q, t)$ at $t=0$.
Afterward,
Lenart-Naito-Sagaki-Schilling-Shimozono \cite{LNSSS2} proved that for a dominant integral weight $\lambda$, the set 
$\QLS(\lambda)$
of all quantum Lakshmibai-Seshadri (QLS) paths of shape $\lambda$ provides a realization of the crystal basis of a special quantum Weyl module %$W_q (\lambda)$ 
over a quantum affine algebra $U_{q}'(\mathfrak{g}_\aff)$ (without degree operator) of an arbitrary untwisted type, 
and that its graded character equals the specialization $E_{\lon \lambda}(q,0)$ at $t=0$, 
which equals the specialization $P_\lambda (q,0)$ of the symmetric Macdonald polynomial $P_\lambda (q, t)$ at $t=0$,
where $\lon$ denotes the longest element of $W$.
Here a QLS path is described in terms of (the parabolic version of) the quantum Bruhat graph, introduced by Brenti-Fomin-Postnikov \cite{BFP};
the set of QLS paths 
is endowed with an affine crystal structure in a way
similar to
the one for the set of ordinary LS paths introduced by Littelmann \cite{L1}.
Moreover,
Lenart-Naito-Sagaki-Schilling-Shimozono \cite{LNSSS3}
 obtained a formula for the specialization $E_{x\lambda}(q, 0)$,
$x \in W$, at $t=0$ in an arbitrary untwisted affine type, which is described in terms of $\QLS$ paths of shape $\lambda$, and proved that the specialization $E_{x\lambda}(q, 0)$ is just the 
graded character of a certain Demazure-type submodule of the special quantum Weyl module.

Also,
Ishii-Naito-Sagaki proved that
for a level-zero dominant weight $\Lambda$,
the set of semi-infinite Lakshmibai-Seshadri (SiLS) paths of shape $\Lambda$ provides a realization of the crystal basis of the 
 level-zero extremal weight module of extremal weight $\Lambda$ over a quantum affine algebra $U_{q} (\mathfrak{g}_\aff)$ of a untwisted affine type,
and by factoring the null root $\delta$ of the affine Lie algebra $\mathfrak{g}_\aff$,
we obtain a surjective strict morphism of crystals from the crystal of all SiLS paths of shape $\Lambda$ onto the set $\QLS(\lambda)$;
here $\lambda$ denotes the image of $\Lambda$ via the projection $ P_\aff \rightarrow  P$
 (see \S 2).

This paper is mainly a $\C$-type-analog of the Lenart-Naito-Sagaki-Schilling-Shimozono's work \cite{LNSSS2}.
%%%%%%%%%%%%4
More precisely, the purpose of this paper is %$\C$-analog of 
%the first purpose is
to define a QLS path of type $\C$,
and to give a representation-theoretic interpretation of the specialization $P^{\C}_\lambda (q, 0)= E^\C_{\lon \lambda} (q, 0)$ at $t=0$
of the symmetric Macdonald polynomial $P^{\C}_\lambda (q, t)$ of type $\C$
in terms of the set $\QLS^\C(\lambda)$ of all $\C$-type QLS paths of shape $\lambda$; 
here $P^{\C}_\lambda (q, t)$ (resp., $E^\C_{\mu} (q, t)$) is a specific
specialization of the symmetric Macdonald-Koornwinder polynomial $P_\lambda (q, t_1, t_2, t_3, t_4, t_5)$
 (resp., nonsymmetric Macdonald-Koornwinder polynomial $E_\mu (q, t_1, t_2, t_3, t_4, t_5)$) (see \S 5.1 and \cite{OS}).
We prove the following 2 theorems:
\begin{theor}[{$=$ Theorem~\ref{realization_theorem}}]\label{thm:A}
Let $\lambda$ be a dominant weight.
Then, the set $\QLS^\C(\lambda)$ of all QLS path of shape $\lambda$ provides a realization of the crystal basis of a quantum Weyl module 
$W_q (\overline{\xi}(\lambda))$ over a quantum affine algebra $U_{q}'(\mathfrak{g} (\C) )$ of type $\C$,
where $\overline{\xi}$ denotes the bijection $\overline{\xi}:P \rightarrow P^0_\aff / (P^0_\aff \cap \mathbb{Q}\delta)$,
and $P^0_\aff$ the set of all level-zero weight {\rm(}see \S $2${\rm)}.
\end{theor}
\vspace{4mm}
\begin{theor}[{$=$ Theorem \ref{theorem_graded_character}}]\label{thm:B}
Let $\lambda$ be a dominant weight.
Then,
%%%%%%%%%%%%
\begin{equation*}
P^\C_\lambda (q,0) = \sum_{\eta \in \QLS^\C(\lambda)} q^{\Dg (\eta)} e^{\wt (\eta)},
\end{equation*}
where
for $\eta \in \QLS(\lambda)$, $\Dg(\eta)$ is a certain nonnegative integer,
which is explicitly described in terms of the quantum Bruhat graph;
see \S $5.1$ for details.
\end{theor}
Theorem~\ref{thm:A} implies that the definition of a $\QLS$ path of type $\C$ is reasonable.
By Theorem~\ref{thm:B},
$P^\C_\lambda (q,0)$ is equal to the graded character of the set $\QLS^\C (\lambda)$.
Moreover, by Theorem~\ref{thm:A},
$P^\C_\lambda (q,0)$ is equal to the graded character of the crystal basis of the quantum Weyl module $W_q (\xi(\lambda))$ over a quantum affine algebra $U_{q}'(\mathfrak{g} (\C) )$.
Also, we state an explicit formula for the specialization $E^\C_\mu (q,0)$ of nonsymmetric Macdonald polynomial $E^\C_\mu (q,t)$ at $t=0$
in terms of the specific subset of $\QLS^\C(\lambda)$.
Furthermore, we establish the relation between QLS paths and SiLS paths of type $\C$, which is 
$\C$-type analog of the Ishii-Naito-Sagaki's work \cite{INS}.

This paper is organized as follows.
In Section 2, we fix our notation.
In Section 3, we define a QLS path
of dual untwisted types and type $\C$.
Also, we prove Theorem~\ref{thm:A} by using root operators acting on the set $\QLS^\C(\lambda)$.
In Section 4, 
we prove Theorem \ref{thm:B};
this theorem gives the description of the specialization $P^\C_\lambda (q, 0)$ at $t=0$ in terms of QLS paths of shape $\lambda$.
In Appendix, we define a SiLS path of type $\C$.
Also, we prove that 
the set of QLS paths of shape $\lambda$
is obtained from  the set of SiLS paths of shape $\Lambda$
through the projection $\mathbb{R}\otimes_{\mathbb{Z}} P_\aff \rightarrow \mathbb{R}\otimes_{\mathbb{Z}} P$.

\section{Notation}

Let $A = (a_{ij})_{i,j \in I}$ be a Cartan matrix, and $A_\aff = (a_{ij})_{i,j \in I_\aff}$ be a generalized Cartan matrix of (twisted) affine type;
in this paper, we consider the following pairs  $(A, A_\aff)$.
\begin{center}
\begin{tabular}{ccc} \hline
$A$ & $A_\aff$ & \\ \hline 
$C_n$ & $A_{2n-1}^{(2)}$ & $(n \geq 3)$ \\ \hline 
$B_n$ & $\D$ & $(n \geq 2)$ \\ \hline 
$F_4$ & $E_6^{(2)}$ &  \\ \hline 
$G_2$ & $D_4^{(3)}$ &  \\ \hline 
$C_n$ & $\C$ & $(n \geq 2)$ \\ \hline 
\end{tabular}
\end{center}
Let
$\mathfrak{g}(A)$ be a finite dimensional simple Lie algebra associated with the Cartan matrix $A $,
and
$\mathfrak{g} (A_\aff)$ be a twisted affine Lie algebra associated with the generalized Cartan matrix $A_\aff$.
We denote by $I$ and $I_\aff = I \sqcup \setzero$ the vertex sets for the Dynkin diagrams of $A$ and $A_\aff$, respectively.

First, we consider the cases when $\mathfrak{g} (A_\aff)$ is  a
dual untwisted affine Lie algebra; that is,
$A_\aff =
A_{2n-1}^{(2)}$, 
 $\D$,
 $E_6^{(2)}$, or
	$D_4^{(3)}$.
Let $\{ \alpha_i \}_{i \in I }$
(resp., $\{ {\alpha}^{\lor}_i \}_{i \in I }$) be
 the set of all simple roots (resp., coroots) of  $\mathfrak{g}(A)$,
$\mathfrak{h} = \bigoplus_{i \in I}\mathbb{C}\alpha^\lor_i$ a Cartan subalgebra of  $\mathfrak{g}(A)$,
$\mathfrak{h}^* = \bigoplus_{i \in I}\mathbb{C}\alpha_i$ the dual space of $\mathfrak{h}$,
and 
$\mathfrak{h}_\mathbb{R} = \bigoplus_{i \in I}\mathbb{R}\alpha_i^\lor$
(resp., $\mathfrak{h}_\mathbb{R}^* = \bigoplus_{i \in I}\mathbb{R}\alpha_i$)
 the real form of $\mathfrak{h}$ (resp., $\mathfrak{h}^*$);
the duality pairing between  $\mathfrak{h}$ and  $\mathfrak{h}^*$ is denoted by
$\langle \cdot, \cdot \rangle : \mathfrak{h}^* \times \mathfrak{h} \rightarrow \mathbb{C}$.
Let $Q = \sum_{i \in I}\mathbb{Z}\alpha_i  
\subset \mathfrak{h}_\mathbb{R}^*$
(resp.,
$Q^\lor = \sum_{i \in I} \mathbb{Z}\alpha^\lor_i  
\subset \mathfrak{h}_\mathbb{R}$
)
 denote 
the root (resp., coroot) lattice of $\mathfrak{g}(A)$,
and 
$P = \sum_{i \in I}\mathbb{Z}\varpi_i \subset \mathfrak{h}_\mathbb{R}^*$  the weight lattice of $\mathfrak{g}(A)$;
here $\varpi_i$, $i \in I$, are the fundamental weights for $\mathfrak{g}(A)$,
that is, 
$\langle \varpi_i ,\alpha_j^\lor \rangle = \delta_{i j}$
for $j \in I$.
Let us denote by $\Delta$ the set of all roots,
and by $\Delta^{+}$ (resp., $\Delta^{-}$) the set of all positive (resp., negative) roots.
For a subset $\Gamma \subset \Delta$, we set $\Gamma^\lor \eqdef \{ \alpha^\lor \ | \ \alpha \in \Gamma \}$.
Also, let $W \eqdef \langle s_i \ | \ i \in I \rangle$
be the Weyl group of $\mathfrak{g}(A)$,
where
$s_i $, $i \in I$, are the simple reflections acting on $\mathfrak{h}^*$ and on $\mathfrak{h}$:
\begin{align*}
s_i x=s_{\alpha_i}x = x - \langle  x , \alpha^\lor_i  \rangle \alpha_i, & \ \ \
x \in \mathfrak{h}^*,\\
s_i y =s_{\alpha^\lor_i} y= y - \langle \alpha_i , y \rangle \alpha^\lor_i,
& \ \ \
y \in \mathfrak{h};
\end{align*}
we denote 
the identity element and
the longest element of $W$ by $e$ and $\lon$, respectively.
If $\alpha \in \Delta$ is written as $\alpha = w \alpha_i$ for 
$w\in W$ and $i \in I$, we define $\alpha^\lor$ to be $w \alpha^\lor_i$;
we often identify $s_\alpha$ with $s_{\alpha^\lor}$.
For  $u \in W$,
the length of $u$ is denoted by $\ell(u)$,
 which equals
$\# (\Delta^+ \cap u^{-1}\Delta^-)$.

For a subset $S \subset I$,
we set
$W_S \eqdef \langle s_i \ | \ i \in S \rangle$.
We denote the longest element of $W_S$ by $\lons$.
Also, we set 
$Q_S \eqdef \sum_{i \in S} \mathbb{Z}\alpha_i$,
$Q_S^\lor \eqdef \sum_{i \in S} \mathbb{Z}\alpha_i^\lor$,
$\Delta_S \eqdef Q_S \cap \Delta^+$,
$\Delta_S^+ \eqdef \Delta_S \cap \Delta^+ $, and 
$\Delta_S^- \eqdef \Delta_S \cap \Delta^- $.
Let $W^S$ denote the set of all minimal-length coset representatives for the cosets in $W / W_S$.
For $w\in W$, we denote  the minimal-length coset representative of the coset $w W_S$ by 
$\lfloor w \rfloor$, and
for a subset $T \subset W$, we set $\lfloor T \rfloor \eqdef \{ \lfloor w \rfloor \ | \ w \in T \} \subset W^S$.
For a dominant weight $\lambda \in P$, i.e., $\pair{\lambda}{\alpha^\lor_i}\geq 0$, $i\in I$,
 we set 
$
S_\lambda \eqdef \{ i \in I \ | \  \langle \lambda , \alpha^{\lor}_i \rangle = 0 \} \subset I.
$

Let $\{ \alpha_i \}_{i \in I_\aff}$ (resp., $\{ \alpha_i^\lor  \}_{\in I_\aff }$) be the set of all simple roots (resp., coroots), and $d$ the degree operator;
here,  
$\alpha_0 = \delta - \theta$ and
$\alpha^\lor_0 = c - \theta^\lor$, 
where 
$\delta$ denotes the null root of $\mathfrak{g}(A_\aff)$,
$c$ denotes the canonical central element of $\mathfrak{g} (A_\aff)$,
and $\theta$ denotes the highest short root of $\mathfrak{g}(A)$.
Then the Cartan subalgebra of $\mathfrak{g}(A_\aff)$
is $\mathfrak{h}_\aff = \left( \bigoplus_{i \in I_\aff}\mathbb{C}\alpha_i^\lor  \right) \oplus \mathbb{C}d$.
We denote by $\Delta_\aff$  the set of all real roots of $\mathfrak{g}(A_\aff)$,
and by $\Delta_\aff^+$ (resp., $\Delta_\aff^-$) the set of all positive (resp., negative) roots.

Also, let 
$\Lambda_i \in \mathfrak{h}_\aff^*$, $i \in I_\aff$,  be the (affine) fundamental weights.
We have
\begin{align*}
\langle \alpha_j , \alpha^\lor_i \rangle_\aff &= a_{i j},  \\
\langle \alpha_j , d \rangle_\aff &= \delta_{j,0},  \\
\langle \Lambda_j , \alpha_i^\lor \rangle_\aff &= \delta_{i,j},  \\
\langle \Lambda_j , d \rangle_\aff &= 0 ,
\end{align*}
for $i, j \in I_\aff$,
where
$\langle \cdot , \cdot \rangle_\aff : \mathfrak{h}_\aff^* \times \mathfrak{h}_\aff \rightarrow \mathbb{C}$ 
denotes the duality pairing.

The weight lattice and the coweight lattice of
$\mathfrak{g} (A_\aff)$ are
$P_\aff \eqdef \left( \oplus_{i \in I_\aff } \mathbb{Z}\Lambda_i  \right)  \oplus \mathbb{Z}\delta  \subset \mathfrak{h}_\aff^*$
and
$P^\lor_\aff \eqdef {\rm Hom}_\mathbb{Z}(P_\aff , \mathbb{Z}) =
\left( \oplus_{i \in I_\aff } \mathbb{Z}\alpha_i^\lor  \right)  \oplus \mathbb{Z}d
\subset \mathfrak{h}_\aff$,
respectively.
We also consider the following lattices:
\begin{align*}
\overline{P_\aff} &= P_\aff / (P_\aff \cap {\mathbb{Q}\delta}) \cong  \bigoplus_{i \in I_\aff } \mathbb{Z}\Lambda_i, \\
\overline{P_\aff}^{\lor} &= \bigoplus_{i \in I_\aff } \mathbb{Z}\alpha_i^\lor 
\subset P^\lor_\aff.
\end{align*}
We denote the canonical projection $P_\aff \rightarrow \overline{P_\aff}$ by $\cl$, and
take an injective $\mathbb{Z}$-linear map $\xi : P \rightarrow  P_\aff$%\footnote{$\overline{P_\aff}$ から変更?} 
given by
$\xi (\varpi_i ) = \Lambda_i - \langle \Lambda_i , c \rangle_\aff \Lambda_0 $,
$i \in I$.
%Here $c$ is the canonical central element of $\mathfrak{g}_\aff$.
%Note that 
%$\langle \xi (\mu) , \alpha_i^\lor \rangle_\aff = \langle \mu , \alpha^\lor_i \rangle $ for $\mu \in P$, $i \in I$.

\begin{rem}
For $x \in P_\aff$, $x$ is a level-zero weight if $\langle x, c \rangle_\aff =0$.
We denote the set of all level-zero weights by
\begin{equation*}
P^0_\aff \eqdef \{ x \in P_\aff \ | \ \langle x, c \rangle_\aff =0 \},
\end{equation*}
and set
\begin{align*}
\overline{P_\aff^0} \eqdef \{ \cl (x) \in \overline{P_\aff} \ | \ x \in P_\aff^0 \} = \bigoplus_{i \in I} \mathbb{Z} (\Lambda_i - \langle \Lambda_i , c \rangle_\aff \Lambda_0 )
\left(  = \bigoplus_{i \in I} \mathbb{Z} \xi (\varpi_i)
\right)
 .
\end{align*}
Then $\xi (P) \subset P_\aff^0$, and $\overline{\xi}\eqdef \cl \circ \xi : P \xrightarrow{\xi}P_\aff^0 \xrightarrow{\cl} \overline{P_\aff^0} $ is a linear isomorphism of $\mathbb{Z}$-modules.
%We can naturally extend the isomorphism $\overline{\xi}$ to $\overline{\xi}: \mathfrak{h}^* = \mathbb{R} \otimes_{\mathbb{Z}} P \rightarrow \mathbb{R} \otimes_{\mathbb{Z}} \overline{P_\aff^0} = \bigoplus_{i \in I} \mathbb{R} \xi(\varpi_i) $.
Also, $x \in P^0_\aff$ is a level-zero dominant weight if $\overline{\xi}\inv \circ \cl (x) \in P$ is a dominant weight.
\end{rem}

We consider the case $(A, A_\aff) = (C_n , \C)$.
In this case, we use all the notations which we set above,
putting the dagger for $\mathfrak{g}(\C)$ and $\mathfrak{g}(C_n)$
except $\theta\pdag$, $\alpha\pdag_0$ and $P\pdag_\aff$;
here 
we denote by $\theta\pdag$ the highest root,
and set
$\alpha\pdag_0 \eqdef \half( -\theta\pdag + \delta\pdag )$,
and 
$P\pdag_\aff \eqdef \left( \oplus_{i \in I\pdag_\aff } \mathbb{Z}\Lambda\pdag_i  \right)  \oplus \half\mathbb{Z}\delta\pdag  \subset (\mathfrak{h}\pdag_\aff)^*$.

\begin{rem}\label{BC}
%Let $(A, A_\aff) = (C_n , \C)$.
Let $I_\aff = \{ 0, \ldots , n\}$
and $I = \{ 1, \ldots , n\}$ denote
  the vertex sets for the Dynkin diagrams of $\mathfrak{g}(\D)$ and $\mathfrak{g} (B_n)$, respectively, indexed as follows
\begin{equation*}
\begin{xy}
\ar@{<=}
(10,0) *++!D{1} *\cir<4pt>{}="B";
 (0,0)  *++!D{0} *\cir<4pt>{};
\ar@{-} "B";(20,0) *++!D{2} *\cir<4pt>{}="C"
\ar@{-} "C";(25,0) \ar@{.} (25,0);(30,0)^*!U{\cdots}
\ar@{-} (30,0);(35,0) *++!D{n-1} *\cir<4pt>{}= "D"
\ar@{=>}  "D"; (45,0) *++!D{n} *\cir<4pt>{}
\end{xy}.
\end{equation*}
Also, let $I\pdag_\aff = \{ 0, \ldots , n\}=I_\aff$
and $I\pdag = \{ 1, \ldots , n\}=I$ denote
  the vertex sets for the Dynkin diagrams of $\mathfrak{g}(\C)$ and $\mathfrak{g} (C_n)$, respectively, indexed as follows
\begin{equation*}
\begin{xy}
\ar@{<=}
(10,0) *++!D{1} *\cir<4pt>{}="B";
 (0,0)  *++!D{0} *\cir<4pt>{};
\ar@{-} "B";(20,0) *++!D{2} *\cir<4pt>{}="C"
\ar@{-} "C";(25,0) \ar@{.} (25,0);(30,0)^*!U{\cdots}
\ar@{-} (30,0);(35,0) *++!D{n-1} *\cir<4pt>{}= "D"
\ar@{<=}  "D"; (45,0) *++!D{n} *\cir<4pt>{}
\end{xy}.
\end{equation*}
Let $\mathfrak{h}_\mathbb{R}$ and $\mathfrak{h}_\mathbb{R}^\dag$ be a Cartan subalgebra of $\mathfrak{g}(B_n)$ and $\mathfrak{g}(C_n)$, 
respectively.
We define linear isomorphisms 
$\iota: \mathfrak{h}\pdag_\mathbb{R} \rightarrow \mathfrak{h}_\mathbb{R}$ and $\iota^* : (\mathfrak{h}_\mathbb{R}\pdag)^* \rightarrow \mathfrak{h}^*_\mathbb{R}$ by
\begin{align*}
\iota^*(\alpha\pdag_i) \eqdef
\left\{
\begin{array}{ll}
\alpha_i & \mbox{ if }i\neq n,\\
2\alpha_i & \mbox{ if }i = n,
\end{array}\right.
\mbox{ and }
\iota((\alpha\pdag_i)^\lor) =
\left\{
\begin{array}{ll}
\alpha_i^\lor & \mbox{ if }i\neq n,\\
\half\alpha_i^\lor & \mbox{ if }i = n;
\end{array}\right.
\end{align*}
for $i \in I$;
here we notice that 
\begin{equation*}
\pair{x}{y}^\dag
=
\pair{\iota^*(x)}{\iota(y)},
\mbox{ for }
x\in (\mathfrak{h}_\mathbb{R}^\dag)^*
\mbox{ and }
y\in \mathfrak{h}_\mathbb{R}^\dag,
\end{equation*}
where 
the pairing $\pair{\cdot}{\cdot}^\dag$ in the left-hand side is the duality pairing $(\mathfrak{h}^\dag_\mathbb{R})^* \times \mathfrak{h}^\dag_\mathbb{R} \rightarrow \mathbb{R}$
and the pairing $\pair{\cdot}{\cdot}$ in the right-hand side is the duality pairing $\mathfrak{h}^*_\mathbb{R} \times \mathfrak{h} _\mathbb{R} \rightarrow \mathbb{R}$.
Then, for $\alpha^\dag \in \Delta\pdag$,
$\iota^* (\alpha^\dag)$ is a long root in $\Delta$ if $\alpha^\dag$ is a short root in $\Delta^\dag$,
or 
twice as a short root in $\Delta$ if $\alpha^\dag$ is a long root in $\Delta^\dag$.
Also,
for $\alpha^\dag \in \Delta\pdag$,
$\iota ((\alpha^\dag)^\lor)$ is a short coroot in $\Delta^\lor$ if $\alpha^\dag$ is a short root in $\Delta^\dag$,
or 
half as a long coroot in $\Delta$ if $\alpha^\dag$ is a long root in $\Delta^\dag$.
%
%
%here the vertex sets $I_\aff = \{ 0, \ldots , n\} = I\pdag_\aff$ and $I = \{ 1, \ldots , n\} = I\pdag$ for the Dynkin diagrams of $\mathfrak{g}(\D)$ and $\mathfrak{g}(B_n)$,
%respectively, is indexed as follows:
%\begin{equation*}
%\begin{xy}
%\ar@{<=}
%(10,0) *++!D{1} *\cir<4pt>{}="B";
% (0,0)  *++!D{0} *\cir<4pt>{};
%\ar@{-} "B";(20,0) *++!D{2} *\cir<4pt>{}="C"
%\ar@{-} "C";(25,0) \ar@{.} (25,0);(30,0)^*!U{\cdots}
%\ar@{-} (30,0);(35,0) *++!D{n-1} *\cir<4pt>{}= "D"
%\ar@{=>}  "D"; (45,0) *++!D{n} *\cir<4pt>{}
%\end{xy}.
%\end{equation*}
If we identify $\mathfrak{h}\pdag_\mathbb{R}$ and ${\mathfrak{h}\pdag_\mathbb{R}}^*$
with
$\mathfrak{h}_\mathbb{R}$ and $\mathfrak{h}_\mathbb{R}^*$, respectively,
then, 
the set of all roots $\Delta\pdag$,
the set of all coroots $(\Delta\pdag)^\lor$,
the highest root $\theta\pdag$,
the weight lattice $P\pdag$, and 
the Weyl group $W\pdag$
of $\mathfrak{g}(C_n)$
can be described 
in terms of 
those of $\mathfrak{g}(B_n)$.
More precisely,
$\iota^* (\Delta\pdag) = \{ \alpha \ | \ \alpha \in \Delta \mbox{ is a long root}\}
\sqcup
\{ 2\alpha \ | \ \alpha \in \Delta \mbox{ is a short root}\}$
and
$\iota^* ( (\Delta\pdag)^\lor ) = \{ \alpha^\lor \ | \ \alpha \in \Delta \mbox{ is a long root}\}
\sqcup
\{ \half \alpha^\lor \ | \ \alpha \in \Delta \mbox{ is a short root}\}$;
in particular,
the highest root $\theta\pdag$ of $\mathfrak{g}(C_n)$ is identified with $2 \theta$,
where $\theta$ is the highest root of $\mathfrak{g}(B_n)$.
The weight lattice $P\pdag$ of $\mathfrak{g}(C_n)$ is identified with the root lattice $Q$ of $\mathfrak{g}(B_n)$;
indeed,
\begin{align*}
\iota^*(P\pdag) &= \left\{ x \in \mathfrak{h}^*_\mathbb{R} \ \left| \ \langle x, \alpha^\lor_i  \rangle \in \mathbb{Z},  1 \leq i \leq n-1 ,\mbox{and }\langle x, \half\alpha^\lor_n \rangle \in \mathbb{Z}\right. \right\} \\
&=\left( \bigoplus_{i \neq n} \mathbb{Z} \varpi_i \right)
\oplus 2 \mathbb{Z} \varpi_i = Q;
\end{align*}
notice that $\iota^*(P\pdag) =Q \subset P$.
There exists a group isomorphism $\iota^* :W\pdag\rightarrow W$ which satisfies $\iota^*(s_i)=s_i$, $i \in I$.
Furthermore, if we identify the null root $\delta\pdag$ of $\mathfrak{g}(\C)$ with $2 \delta$,
where $\delta$ is the null root of $\mathfrak{g}(\D)$, 
then $\alpha\pdag_0$ is identified with $\alpha_0$.
\end{rem}

We fix a pair $(A, A_\aff)$ with $\mathfrak{g}(A_\aff)$ a dual untwisted affine Lie algebra.
For $\alpha \in \Delta$ of $\mathfrak{g}(A)$,
we set
\begin{equation*}
c_\alpha \eqdef \left\{
\begin{array}{ll}
1 & \mbox{if }\alpha \mbox{ is a short root,}\\
2 & \mbox{if }\alpha \mbox{ is a long root and }A \neq G_2,\\
3 & \mbox{if }\alpha \mbox{ is a long root and }A = G_2.
\end{array}\right.
\end{equation*}
Then $\Delta_\aff = \{ \alpha + c_\alpha a \delta \ | \ \alpha \in \Delta , a \in \mathbb{Z} \}$.

\begin{rem}\label{rem:root_coroot}
Keep the notation and setting in Remark~\ref{BC}.
Then, for every $\alpha^\dag \in \Delta^\dag$, there exists
$\alpha \in \Delta$ such that $\iota^* (\alpha^\dag) = \frac{2}{c_\alpha}\alpha$ and 
$\iota ((\alpha^\dag)^\lor) = \frac{c_\alpha}{2}\alpha^\lor$.
\end{rem}

\begin{rem}\label{BC'}
For $A=B_n$, 
we define the standard bilinear form $(\cdot, \cdot) : \mathfrak{h}_\mathbb{R}^* \times \mathfrak{h}_\mathbb{R}^* \rightarrow \mathbb{R}$
by
\begin{equation*}
( x , \alpha_i ) \eqdef \frac{c_{\alpha_i}}{2} \langle x, \alpha_i^\lor \rangle,
\ \
x \in \mathfrak{h}_\mathbb{R}^*.
\end{equation*}
We can identify $\mathfrak{h}_\mathbb{R}$ with $\mathfrak{h}^*_\mathbb{R}$ by $(\cdot, \cdot) : \mathfrak{h}_\mathbb{R}^* \times \mathfrak{h}_\mathbb{R}^* \rightarrow \mathbb{R}$;
$\alpha^\lor$ is identified with $\frac{2}{c_\alpha}\alpha$.
Then, by Remark~\ref{rem:root_coroot}, we have $\iota^* (\Delta^\dag) = \Delta^\lor$ and $\iota ((\Delta^\dag)^\lor) = \Delta$.
We claim that 
%%%%%%%%%%
%%%%%%%%%%
%equ:pair
%%%%%%%%%%
%%%%%%%%%%
\begin{equation}\label{equ:pair}
\langle \alpha^\dag, (\beta^\dag)^\lor \rangle^\dag = \langle \iota((\beta^\dag)^\lor), \iota^*(\alpha^\dag) \rangle,
\mbox{ for }\alpha^\dag, \beta^\dag \in \Delta^\dag;
\end{equation}
here, the pairing $\pair{\cdot}{\cdot}^\dag$ in the left-hand side is the duality pairing $(\mathfrak{h}^\dag)^* \times \mathfrak{h}^\dag \rightarrow \mathbb{R}$
and the pairing $\pair{\cdot}{\cdot}$ in the right-hand side is the duality pairing $\mathfrak{h}^* \times \mathfrak{h} \rightarrow \mathbb{R}$.
Indeed, 
Since there exist $\alpha , \beta \in \Delta$ such that
 $\iota^* (\alpha^\dag) = \frac{2}{c_\alpha}\alpha$ and $\iota ((\beta^\dag)^\lor) = \frac{c_\beta}{2}\beta^\lor$,
then 
\begin{equation*}
\langle \alpha^\dag, (\beta^\dag)^\lor \rangle^\dag
=
\langle \iota^*(\alpha^\dag), \iota((\beta^\dag)^\lor) \rangle
=
\langle \frac{2}{c_\alpha}\alpha, \frac{c_\beta}{2}\beta^\lor \rangle
=
\frac{2}{c_\alpha}(\alpha, \beta)
=
\frac{2}{c_\alpha}(\beta, \alpha)
=
\langle \beta, \alpha^\lor \rangle.
\end{equation*}
%
%By Remark \ref{BC}, $\{ \alpha^\lor \ | \ \alpha \in \Delta \}$ is the set of all roots of $\mathfrak{g}(C_n)$.
%We use this identification in \S 2.2.
\end{rem}

\section{The (parabolic) quantum Bruhat graphs of type $\C$}
\subsection{Definition of the (parabolic) quantum Bruhat graphs}
First, we fix $(A, A_\aff)$ with $\mathfrak{g}(A_\aff)$ a dual untwisted affine Lie algebra.
\begin{dfn}[\normalfont{\cite[Definition 6.1]{BFP}}]\label{QBG}
The quantum Bruhat graph, denoted by $\QBG$, is the directed graph with vertex set $W$, and  directed edges 
 labeled by positive coroots;
for $u,v \in W$, and $\beta \in \Delta^+$, 
an arrow $u \xrightarrow{\beta^\lor} v $ is an edge of $\QBG$
if the following conditions  hold:

\begin{enu}
\item
$v=u s_\beta$, and

\item
either
(2a): 
$\ell(v)=\ell(u)+1$, or
(2b):
$\ell(v)=\ell(u) - 2\langle  \beta, \rho^\lor \rangle +1$,
\end{enu}
where $\rho^\lor \eqdef \half \sum_{\alpha \in \Delta^+}\alpha^\lor$.
An edge satisfying condition (2a)
(resp., (2b))
is called a Bruhat (resp., quantum) edge.
\end{dfn}

\begin{dfn}[\normalfont{{see  \cite[\S 4.3]{LNSSS1} for untwisted types}}]\label{QBGS}
The parabolic quantum Bruhat graph, denoted by $\QBG^S$, is the directed graph with vertex set $W^S$, and directed edges
labeled by 
positive coroots in
$( \Delta^+ \setminus \Delta^+_S )^\lor$;
for $u,v \in W^S$, 
and $\beta \in \Delta^+ \setminus \Delta^+_S$,
an arrow $u \xrightarrow{\beta^\lor} v $ is an edge of $\QBG^S$
if the following conditions hold{\rm :}

\begin{enu}
\item
$v=\lfloor u s_\beta \rfloor$, and

\item
 either
{\rm (2a):} 
$\ell(v)=\ell(u)+1$, or
{\rm (2b):} $\ell(v)=\ell(u) - 2\langle   \beta, \rho^\lor - \rho_S^\lor \rangle +1$,
\end{enu}
where $\rho_S^\lor =\half \sum_{\alpha \in \Delta^+_S}\alpha^\lor$.
An edge satisfying condition {\rm (2a)} (resp., {\rm (2b)}) is called a Bruhat (resp., quantum) edge.
\end{dfn}

For a dominant weight 
$\lambda \in P$,
we set $S = S_\lambda$.

\begin{dfn}[{see \cite[\S 3.2]{LNSSS2} for untwisted types}]
Let $\lambda \in P$ be a dominant weight and 
$b \in \mathbb{Q}\cap [0,1]$.
We denote by
$\QBG_{b\lambda}$ (resp., $\QBG^S_{b\lambda}$ ) 
 the subgraph of $\QBG$ (resp., $\QBG^S$)
with the same vertex set but having only the edges:
$u \xrightarrow{\beta^\lor} v$ with $b\langle \lambda, \beta^{\lor}  \rangle \in \mathbb{Z}$.
\end{dfn}

Let us fix  $(A, A_\aff) = (B_n , \D)$.
Let us define $\QBG_{b\lambda}^{A^{(2)}_{2n}}$ and $\left( \QBG_{b\lambda}^{A^{(2)}_{2n}} \right)^S$
in terms of $\QBG_{b\lambda}$ and $\QBG_{b\lambda}^S$.

\begin{dfn}
Let $\lambda \in Q $ be a dominant weight and 
$b \in \mathbb{Q}\cap [0,1]$.
We denote by
$\QBG_{b\lambda}^{A^{(2)}_{2n}}$ (resp., $\left( \QBG_{b\lambda}^{A^{(2)}_{2n}} \right)^S$ )
 the subgraph of $\QBG_{b\lambda}$ (resp., $\QBG_{b \lambda}^S$)
with the same vertex set but having only the edges:
\begin{equation*}
u \xrightarrow{\beta^\lor} v \mbox{ with } 
\left\{
			\begin{array}{ll}
				b\langle \lambda, \beta^\lor \rangle \in \mathbb{Z} & \ \ \mbox{ if }\beta \mbox{ is a long root of }\Delta  ,  \\
				b\langle \lambda, \beta^\lor \rangle \in \mathbb{Z} & \ \ \mbox{ if the edge is a quantum edge with }\beta \mbox{ a short root of }\Delta  ,  \\
				b\langle \lambda, \beta^\lor \rangle \in 2\mathbb{Z} & \ \ \mbox{ if the edge is a Bruhat edge with }\beta \mbox{ a short root of }\Delta.
			\end{array}
		\right.
\end{equation*}
\end{dfn}

\subsection{$\QBG^\C$ in terms of the root system of type $C_n$}
Let $(A, A_\aff) = (C_n , \C)$.
In this subsection, we use the notation in Remark \ref{BC}.
\begin{dfn}[\normalfont{\cite[Definition 6.1]{BFP}}]\label{QBG_A}
The quantum Bruhat graph (of type $C_n$), denoted by $\QBG\pdag$, is the directed graph with vertex set $W\pdag$, and  directed edges 
 labeled by positive roots;
for $u,v \in W\pdag$, and $\beta \in (\Delta\pdag)^+$, 
an arrow $u \xrightarrow{\beta} v $ is an edge of $\QBG\pdag$
if the following conditions hold:

\begin{enu}
\item
$v=u s_\beta$, and

\item
either
(2a): 
$\ell(v)=\ell(u)+1$, or
(2b):
$\ell(v)=\ell(u) - 2\langle \rho\pdag, \beta^\lor \rangle^\dag +1$,
\end{enu}
where $\rho\pdag \eqdef \half \sum_{\alpha \in (\Delta\pdag)^+}\alpha$.
An edge satisfying condition (2a) 
(resp., (2b))
is called a Bruhat (resp., quantum) edge.
\end{dfn}

We remark that $\QBG\pdag$ is identical to the quantum Bruhat graph for the untwisted affine Lie algebra of type $C_n^{(1)}$ 
(see \cite{LNSSS1}).

\begin{dfn}[\normalfont{{\cite[\S 4.3]{LNSSS1}}}]\label{QBGS_A}
The parabolic quantum Bruhat graph, denoted by $(\QBG\pdag)^S$, is the directed graph with vertex set $(W\pdag)^S$, and directed edges
labeled by 
positive coroots in
$ (\Delta\pdag)^+ \setminus  (\Delta\pdag)^+_S$;
for $u,v \in (W\pdag)^S$, 
and $\beta \in (\Delta\pdag)^+ \setminus  (\Delta\pdag)^+_S$,
$u \xrightarrow{\beta} v $ is an edge of $(\QBG\pdag)^S$
if the following conditions hold{\rm:}

\begin{enu}
\item
$v=\lfloor u s_\beta \rfloor$, and

\item 
either
{\rm (2a):} 
$\ell(v)=\ell(u)+1$, or
{\rm (2b):} $\ell(v)=\ell(u) - 2\langle \rho\pdag - \rho_S\pdag , \beta^\lor \rangle^\dag +1$,
\end{enu}
where $\rho_S\pdag =\half \sum_{\alpha \in (\Delta\pdag_S)^+}\alpha$.
An edge satisfying condition {\rm (2a)} (resp., {\rm (2b)}) is called a Bruhat (resp., quantum) edge.
\end{dfn}

We take and fix an arbitrary dominant weight
$\lambda \in P\pdag$, i.e., 
$\langle \lambda , (\alpha^\dag_i)^{\lor} \rangle^\dag \geq 0$
for all $i \in I$.
We set 
\begin{equation*}
S = S_\lambda \eqdef \{ i \in I \ | \  \langle \lambda , (\alpha^\dag_i )^{\lor} \rangle^\dag = 0 \} \subset I.
\end{equation*}

\begin{dfn}\label{QBG_B}
Let $\lambda \in P\pdag $ be a dominant weight and 
$b \in \mathbb{Q}\cap [0,1]$.
We denote by
$(\QBG_{b\lambda}\pdag)^{\C}$ (resp., $\left( (\QBG\pdag_{b\lambda})^{\C} \right)^S$ )
 the subgraph of $\QBG\pdag$ (resp., $(\QBG\pdag)^S$)
with the same vertex set but having only the edges:
\begin{equation*}
u \xrightarrow{\beta} v \mbox{ with } 
\left\{
			\begin{array}{ll}
				b\langle \lambda, \beta^\lor \rangle^\dag \in \mathbb{Z} & \ \ \mbox{ if }\beta \mbox{ is a short root of }\Delta\pdag  ,  \\
				b\langle \lambda, \beta^\lor \rangle^\dag \in \half \mathbb{Z} & \ \ \mbox{ if the edge is a quantum edge with }\beta \mbox{ a long root of }\Delta\pdag  ,  \\
				b\langle \lambda, \beta^\lor \rangle^\dag \in \mathbb{Z} & \ \ \mbox{ if the edge is a Bruhat edge with }\beta \mbox{ a long root of }\Delta\pdag.
			\end{array}
		\right.
\end{equation*}
\end{dfn}

\begin{rem}\label{identification_QBG}
It follows from Lemmas~\ref{BC},~\ref{rem:root_coroot}, and~\ref{BC'} that 

\begin{enu}
\item
$W\pdag \overset{\iota^*}{\cong} W$,

\item
$ \iota^*(\rho\pdag) =\rho^\lor$
and
$\iota^*((\rho\pdag)_S) = \rho^\lor_S $ since $\iota^*(\Delta^\dag)=\Delta^\lor$, and $\iota^*((\Delta^\dag)^\lor)=\Delta$,

\item
$\pair{\alpha}{\beta^\lor}^\dag = \pair{\iota(\beta^\lor)}{\iota^*(\alpha)}$
for $\alpha, \beta \in \Delta^\dag$,

\item
we can rewrite integrality condition by using
$\iota^*(\Delta\pdag) = \{ \alpha \ | \ \alpha \in \Delta \mbox{ is a long root}\}
\sqcup
\{ 2\alpha \ | \ \alpha \in \Delta \mbox{ is a short root}\}$.
\end{enu}
Therefore, by \eqref{equ:pair}, for $\lambda \in P\pdag$ a dominant weight, and $b \in \mathbb{Q}\cap [0,1]$,
there exist graph isomorphisms $ (\QBG\pdag_{b \lambda})^\C \overset{\iota^*}{\cong} \QBG_{b \iota^*(\lambda)}^\C $ and 
$((\QBG\pdag_{b \lambda})^\C)^S \overset{\iota^*}{\cong} (\QBG_{b \iota^*(\lambda)}^\C)^S $
induced by the group isomorphism $\iota^*:W\pdag \rightarrow W$.
\end{rem}

\subsection{Some properties on QBG}
Let $\prec$ be a total order on $(\Delta^+)^\lor$ satisfying the following condition:
if $\alpha$, $\beta$, $\gamma \in \Delta^+$ with $\gamma = \alpha + \beta$, then 
$\alpha^\lor \prec \gamma^\lor \prec \beta^\lor$ or $\beta^\lor \prec \gamma^\lor \prec \alpha^\lor$ (see \S 2.2).

\begin{prop}[{\cite[Theorem 6.4]{BFP}}] \label{shellability}
	Let $u$ and $v$ be elements in $W$.

\begin{enu}
\item
There exists a unique directed path from $u$ to $v$ in $\QBG$ 
for which the edge labels are strictly increasing {\rm(}resp., strictly decreasing{\rm)} in the total order $\prec$ above.

\item
The unique label-increasing {\rm(}resp., label-decreasing{\rm)} path
\begin{equation*}
u = u_0
\xrightarrow{\gamma_1^\lor}
u_1
\xrightarrow{\gamma_2^\lor}
\cdots
\xrightarrow{\gamma_r^\lor}
u_r=v
\end{equation*}
 from $u$ to $v$ in $\QBG$
is a shortest directed path from $u$ to $v$.
Moreover, it is lexicographically minimal {\rm(}resp., lexicographically maximal{\rm)} among all shortest directed paths from $u$ to $v${\rm;}
that is, for an arbitrary shortest directed path
\begin{equation*}
u = u'_0
\xrightarrow{{\gamma'_1}^\lor}
u'_1
\xrightarrow{{\gamma'_2}^\lor}
\cdots
\xrightarrow{{\gamma'_r}^\lor}
u'_r=v
\end{equation*}
from $u$ to $v$ in $\QBG$, there exists $1 \leq j \leq r$
such that $\gamma_j^\lor \prec {\gamma'_j}^\lor$ {\rm(}resp., $\gamma_j^\lor \succ {\gamma'_j}^\lor${\rm)},
and $\gamma_k^\lor = {\gamma'_k}^\lor$ for $1 \leq k \leq j-1$.
\end{enu}
	\end{prop}

\begin{lem}\label{involution}
If $x \xrightarrow{\beta^\lor} y$ is a Bruhat $($resp., quantum$)$ edge of $\QBG$, then $y\lon \xrightarrow{- \lon \beta^\lor} x \lon$ is also a Bruhat $($resp., quantum$)$ edge of $\QBG$.
\end{lem}

\begin{proof}
This follows easily from equalities
$\ell(y)-\ell(x)=\ell(x \lon )-\ell(y \lon )$ and $\langle  - \lon \beta,\rho^\lor \rangle =\langle  \beta, \rho^\lor\rangle $.
\end{proof}

\begin{lem}[{see \cite[Proposition 5.10 and Proposition 5.11]{LNSSS1} for untwisted types}]\label{leftaction}
Let $\lambda \in P$ be a dominant weight. Let $w \in W^S$, $i \in I$, and $\beta \in \Delta^+$.

\begin{enu}
\item
If $\langle w \lambda , \alpha_i^\lor \rangle > 0$ {\rm(}resp., $<0${\rm)}, then
$w \xrightarrow{w\inv \alpha_i^\lor} \lfloor s_i w \rfloor$
{\rm(}resp., $w \xleftarrow{-w\inv \alpha_i^\lor} \lfloor s_i w \rfloor${\rm)}
is a Bruhat edge.

\item
If $\langle w \lambda , -\theta^\lor \rangle < 0$ {\rm(}resp., $>0${\rm)}, then
$w \xleftarrow{z w\inv \theta^\lor} \lfloor s_\theta w \rfloor$
{\rm(}resp., $w \xrightarrow{-w\inv \theta^\lor} \lfloor s_\theta w \rfloor${\rm)}
is a quantum edge where $z \in W_S$ is defined by $ s_\theta w = \lfloor s_\theta w \rfloor z$.
\end{enu}
\end{lem}

\begin{proof}
(1) is well-known.
The proof of (2) is similar to that of \cite[Proposition 5.11]{LNSSS1}.
\end{proof}

The proof of the following lemma is similar to those of \cite[Proposition 4.1.4 (3), (4) and Proposition 4.1.5 (3), (4)]{LNSSS1}.
\begin{lem}[{see \cite[Proposition 4.1.4 (3), (4) and Proposition 4.1.5 (3), (4)]{LNSSS1} for untwisted types}]\label{2.3.3.5}
Let $\lambda \in P$ be a dominant weight. Let $w \in W^S$, and $i \in I$.

\begin{enu}
\item
If $\langle w \lambda, \alpha_i^\lor \rangle \geq 0$ and $\langle w s_\beta \lambda, \alpha_i^\lor \rangle < 0$,
then $w \beta = \pm \alpha_i$.

\item
If $\langle w \lambda, \alpha_i^\lor \rangle > 0$ and $\langle w s_\beta \lambda, \alpha_i^\lor \rangle \leq 0$,
then $w \beta = \pm \alpha_i$.

\item
If $\langle w \lambda, -\theta^\lor \rangle \geq 0$ and $\langle w s_\beta \lambda, -\theta^\lor \rangle < 0$,
then $w \beta = \pm \theta$.

\item
If $\langle w \lambda, -\theta^\lor \rangle > 0$ and $\langle w s_\beta \lambda, -\theta^\lor \rangle \leq 0$,
then $w \beta = \pm \theta$.
\end{enu}
\end{lem}

In what follows,
we fix a dominant weight $\lambda \in P$ for a dual untwisted type, or a dominant weight $\lambda \in Q$ for type $\C$;
that is, 
if we consider $\QBG_{b \lambda}$ and $\QBG_{b \lambda}^S$, then $\lambda \in P$,
and if we consider $\QBG_{b \lambda}^\C$ and $(\QBG_{b \lambda}^\C)^S$, then $\lambda \in Q$.

In diagrams of the following lemma, a plain (resp., dotted) edge represents a Bruhat (resp., quantum) edge.
\begin{lem}[{see \cite[Lemma 5.14]{LNSSS2} for untwisted types}]\label{diamond}
Let $G$ be $\QBG_{b \lambda}^S$ or $(\QBG_{b \lambda}^\C)^S$.
Let $i \in I$, $\gamma \in \Delta^+ \setminus \Delta^+_S$, and $w \in W^S$.
Then we have following cases 
in each of which 
the bottom two edges in $G$ imply the top two edges in $G$ in the left diagram,
and
the top two edges in $G$ imply the bottom two edges in $G$ in the right diagram.

\begin{enu}
\item
Here we assume that $\gamma \neq w^{-1}\alpha_i$ and have $s_i \lfloor w s_\gamma \rfloor = s_i w s_\gamma = \lfloor s_i w s_\gamma \rfloor$  in both cases.
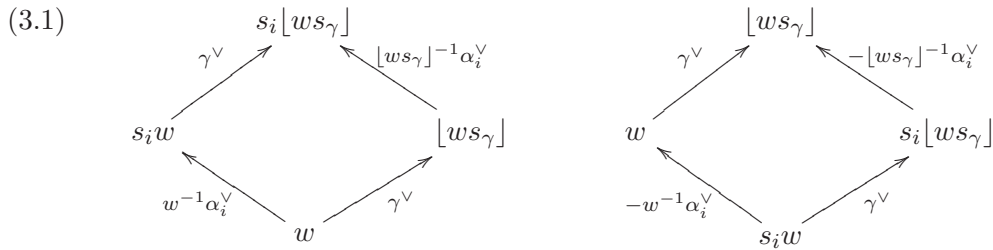
\begin{align}\label{eq_5.3}
\xymatrix{
  & s_i \lfloor w s_\gamma \rfloor  & \\
s_i w\ar[ur]^{\gamma^\lor}
& & \lfloor w s_\gamma \rfloor\ar[ul]_{\lfloor w s_\gamma \rfloor^{-1}\alpha_i^\lor} \\
& w  \ar[ul]^{w^{-1}\alpha_i^\lor}\ar[ur]_{\gamma^\lor}&
}
\ \ \ \ \ \ \ \
\xymatrix{
  &  \lfloor w s_\gamma \rfloor  & \\
\; \;w \; \; \ar[ur]^{\gamma^\lor}
& &s_i \lfloor w s_\gamma \rfloor\ar[ul]_{-\lfloor w s_\gamma \rfloor^{-1}\alpha_i^\lor} \\
& s_i w  \ar[ul]^{-w^{-1}\alpha_i^\lor}\ar[ur]_{\gamma^\lor}&
}
\end{align}

\item
Here we have $s_i \lfloor w s_\gamma \rfloor = \lfloor s_i w s_\gamma \rfloor$ in both cases.
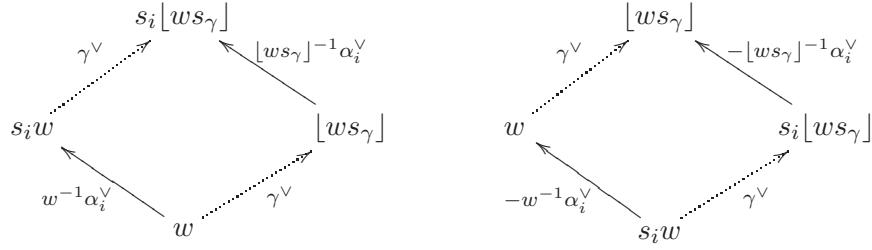
\begin{align}\label{eq_5.4}
\xymatrix{
  & s_i \lfloor w s_\gamma \rfloor  & \\
s_i w\ar@{.>}[ur]^{\gamma^\lor}
& & \lfloor w s_\gamma \rfloor\ar[ul]_{\lfloor w s_\gamma \rfloor^{-1}\alpha_i^\lor} \\
& w  \ar[ul]^{w^{-1}\alpha_i^\lor}\ar@{.>}[ur]_{\gamma^\lor}&
}
\ \ \ \ \ \ \ \
\xymatrix{
  &  \lfloor w s_\gamma \rfloor  & \\
\; \;w \; \; \ar@{.>}[ur]^{\gamma^\lor}
& &s_i \lfloor w s_\gamma \rfloor\ar[ul]_{-\lfloor w s_\gamma \rfloor^{-1}\alpha_i^\lor} \\
& s_i w  \ar[ul]^{-w^{-1}\alpha_i^\lor}\ar@{.>}[ur]_{\gamma^\lor}&
}
\end{align}

\item
Here $z$, $z' \in W_S$ are defined by 
$s_\theta w = \lfloor s_\theta w \rfloor z$, 
$s_\theta \lfloor  w s_\gamma \rfloor = \lfloor s_\theta \lfloor  w s_\gamma \rfloor \rfloor z' = \lfloor s_\theta  w s_\gamma \rfloor z' $.
In subcase \eqref{eq_5.5} {\rm(}resp., \eqref{eq_5.6}{\rm)}
we assume that $\langle w^{-1}\theta , \gamma^\lor \rangle$ is nonzero {\rm(}resp., zero{\rm)}.
In both cases, we have $w s_\gamma = \lfloor  w s_\gamma \rfloor$.
\begin{align}\label{eq_5.5}
\xymatrix{
  & \lfloor s_\theta  w s_\gamma \rfloor  & \\
\lfloor s_\theta w \rfloor \ar@{.>}[ur]^{z\gamma^\lor}
& & \lfloor w s_\gamma \rfloor\ar@{.>}[ul]_{-\lfloor w s_\gamma \rfloor^{-1}\theta^\lor} \\
& w  \ar@{.>}[ul]^{-w^{-1}\theta^\lor}\ar[ur]_{\gamma^\lor}&
}
\ \ \ \ \ \ \ \
\xymatrix{
  &  \lfloor w s_\gamma \rfloor  & \\
\; \;w \; \; \ar[ur]^{\gamma^\lor}
& &\lfloor s_\theta  w s_\gamma \rfloor \ar@{.>}[ul]_{z'\lfloor w s_\gamma \rfloor^{-1}\theta^\lor} \\
&\lfloor s_\theta w\rfloor \ar@{.>}[ul]^{zw^{-1}\theta^\lor}\ar@{.>}[ur]_{z\gamma^\lor}&
}
\end{align}
\begin{align}\label{eq_5.6}
\xymatrix{
  & \lfloor s_\theta  w s_\gamma \rfloor  & \\
\lfloor s_\theta w \rfloor \ar[ur]^{z\gamma^\lor}
& & \lfloor w s_\gamma \rfloor\ar@{.>}[ul]_{-\lfloor w s_\gamma \rfloor^{-1}\theta^\lor} \\
& w  \ar@{.>}[ul]^{-w^{-1}\theta^\lor}\ar[ur]_{\gamma^\lor}&
}
\ \ \ \ \ \ \ \
\xymatrix{
  &  \lfloor w s_\gamma \rfloor  & \\
\; \;w \; \; \ar[ur]^{\gamma^\lor}
& &\lfloor s_\theta  w s_\gamma \rfloor \ar@{.>}[ul]_{z'\lfloor w s_\gamma \rfloor^{-1}\theta^\lor} \\
&\lfloor s_\theta w\rfloor \ar@{.>}[ul]^{zw^{-1}\theta^\lor}\ar[ur]_{z\gamma^\lor}&
}
\end{align}

\item
Here we assume $\gamma \neq - w^{-1}\theta$ in all cases,
and $z$, $z' \in W_S$ are defined as in (3).
In subcase \eqref{eq_5.7} {\rm(}resp., \eqref{eq_5.8}{\rm)} we assume that $\langle w^{-1}\theta, \gamma^\lor \rangle$ is nonzero {\rm(}resp., zero{\rm)}.
\begin{align}\label{eq_5.7}
\xymatrix{
  & \lfloor s_\theta  w s_\gamma \rfloor  & \\
\lfloor s_\theta w \rfloor \ar[ur]^{z\gamma^\lor}
& & \lfloor w s_\gamma \rfloor\ar@{.>}[ul]_{-\lfloor w s_\gamma \rfloor^{-1}\theta^\lor} \\
& w  \ar@{.>}[ul]^{-w^{-1}\theta^\lor}\ar@{.>}[ur]_{\gamma^\lor}&
}
\ \ \ \ \ \ \ \
\xymatrix{
  &  \lfloor w s_\gamma \rfloor  & \\
\; \;w \; \; \ar@{.>}[ur]^{\gamma^\lor}
& &\lfloor s_\theta  w s_\gamma \rfloor \ar@{.>}[ul]_{z'\lfloor w s_\gamma \rfloor^{-1}\theta^\lor} \\
&\lfloor s_\theta w\rfloor \ar@{.>}[ul]^{zw^{-1}\theta^\lor}\ar[ur]_{z\gamma^\lor}&
}
\end{align}

\begin{align}\label{eq_5.8}
\xymatrix{
  & \lfloor s_\theta  w s_\gamma \rfloor  & \\
\lfloor s_\theta w \rfloor \ar@{.>}[ur]^{z\gamma^\lor}
& & \lfloor w s_\gamma \rfloor\ar@{.>}[ul]_{-\lfloor w s_\gamma \rfloor^{-1}\theta^\lor} \\
& w  \ar@{.>}[ul]^{-w^{-1}\theta^\lor}\ar@{.>}[ur]_{\gamma^\lor}&
}
\ \ \ \ \ \ \ \
\xymatrix{
  &  \lfloor w s_\gamma \rfloor  & \\
\; \;w \; \; \ar@{.>}[ur]^{\gamma^\lor}
& &\lfloor s_\theta  w s_\gamma \rfloor \ar@{.>}[ul]_{z'\lfloor w s_\gamma \rfloor^{-1}\theta^\lor} \\
&\lfloor s_\theta w\rfloor \ar@{.>}[ul]^{zw^{-1}\theta^\lor}\ar@{.>}[ur]_{z\gamma^\lor}&
}
\end{align}
\end{enu}
\end{lem}

\begin{proof}
If $G= \QBG_{b \lambda}^S$, then the proof is similar to the proof of \cite[Lemma 5.14]{LNSSS1}.
Hence we assume that $G= (\QBG_{b \lambda}^\C)^S$.
Then in all the cases except (\ref{eq_5.5}) and (\ref{eq_5.7}) the assertion is obvious by Lemma \ref{diamond} for $G= (\QBG_{b \lambda})^S$ of
type $\D$.

In cases (\ref{eq_5.5}) and (\ref{eq_5.7}), the assertion for $b=1$ is obvious by Lemma \ref{diamond} for $G= (\QBG_{b \lambda})^S$ of type $\D$.
Suppose that $w \xrightarrow{\gamma^\lor} \lfloor w s_\gamma \rfloor$ is a quantum edge of $ (\QBG_{b \lambda}^\C)^S$.
It suffices to show that $\lfloor s_\theta w \rfloor \xrightarrow{z \gamma^\lor} \lfloor  s_\theta w s_\gamma \rfloor$ is a Bruhat edge of $ (\QBG_{b \lambda}^\C)^S$.
By Lemma \ref{diamond} for $G= (\QBG_{b \lambda})^S$ of type $\D$,
we know that $\lfloor s_\theta w \rfloor \xrightarrow{z \gamma^\lor} \lfloor  s_\theta w s_\gamma \rfloor$ is a Bruhat edge of $\QBG_{b\lambda}^S$.
Recall that $\theta$ is the highest short root of $\mathfrak{g} (\D)$.
Hence if $\langle w^{-1}\theta, \gamma^\lor \rangle \neq 0$, then $\gamma$ must be a long root.
Therefore, $b \langle \lambda , z\gamma^\lor \rangle =b \langle \lambda , \gamma^\lor \rangle  \in \mathbb{Z}$, and hence
$\lfloor s_\theta w \rfloor \xrightarrow{z \gamma^\lor} \lfloor  s_\theta w s_\gamma \rfloor$ is a Bruhat edge of $ (\QBG_{b \lambda}^\C)^S$.
\end{proof}

%%%%%%%%%%
For an edge $u \xrightarrow{\beta} v $ of $\QBG$,
we set
\begin{equation*}
\wt (u \rightarrow v)
\eqdef
\left\{
\begin{array}{ll}
      0 & \mbox{if} \ u \xrightarrow{\beta} v \mbox{ is a Bruhat edge}, \\
      \beta^{\lor} &  \mbox{if} \ u \xrightarrow{\beta} v \mbox{ is a quantum edge}.
\end{array}
\right.
\end{equation*}
Also, for $u,v \in W$, we take a shortest directed path
$u=x_0 \xrightarrow{\gamma_1} x_1 \xrightarrow{\gamma_2}\cdots \xrightarrow{\gamma_r} x_r=v$ in $\QBG$,
and set
\begin{equation*}
\wt (u\Rightarrow v) \eqdef \wt (x_0 \rightarrow x_1)+\cdots +\wt (x_{r-1} \rightarrow x_r) \in Q^\lor;
\end{equation*}
%\begin{equation*}
%\wt_\lambda (u\Rightarrow v) \eqdef \wt_\lambda (x_0 \rightarrow x_1)+\cdots +\wt_\lambda (x_{r-1} \rightarrow x_r);
%\end{equation*}
we know from \cite[Lemma 1 (2), (3)]{Po} that
this definition does not depend on the choice of a shortest directed path from $u$ to $v$ in $\QBG$.
For a dominant weight $\lambda \in P$, we set $\wt_\lambda (u \Rightarrow v ) \eqdef  \pair{\lambda}{\wt (u\Rightarrow v)}$,
and call it the $\lambda$-weight of a directed path from $u$ to $v$ in $\QBG$.

%%%%%%%%%%%%%%
%%%%%%%%%%%%%%
%begin 1
%%%%%%%%%%%%%%
%%%%%%%%%%%%%%
For an edge $u \xrightarrow{\beta} v $ in $\QBG^S$, we set
\begin{equation*}
\wt^S (u \rightarrow v)
\eqdef
\left\{
\begin{array}{ll}
      0 & \mbox{if} \ u \xrightarrow{\beta} v \mbox{ is a Bruhat edge}, \\
      \beta^{\lor} &  \mbox{if} \ u \xrightarrow{\beta} v \mbox{ is a quantum edge}.
\end{array}
\right.
\end{equation*}
Also, for $u, v \in W^S$, we take a shortest directed path
$\bp:u=x_0 \xrightarrow{\gamma_1} x_1 \xrightarrow{\gamma_2}\cdots \xrightarrow{\gamma_r} x_r=v$ in $\QBG^S$
(such a path always exists by \cite[Lemma 6.12]{LNSSS1}),
and set
\begin{equation*}
\wt^S (\bp) \eqdef \wt^S (x_0 \rightarrow x_1)+\cdots +\wt^S (x_{r-1} \rightarrow x_r) \in Q^\lor;
\end{equation*}
we know from \cite[Proposition 8.1]{LNSSS1} that if $\bq$ is another shortest directed path from $u$ to $v$ in $\QBG^S$,
then $\wt^S(\bp) - \wt^S(\bq) \in Q_S^\lor \eqdef \sum_{i \in S} \mathbb{Z}_{\geq 0} \alpha^\lor_i$. 
%%%%%%%%%%%%%%
%%%%%%%%%%%%%%
%end 1
%%%%%%%%%%%%%%
%%%%%%%%%%%%%%

%\subsection{Relation between $\QBG$ and  $\QBG^S$}\footnote{タイトル不適切?}
Now, we take and fix an arbitrary dominant weight
$\lambda \in P^+$,
%$\lambda \in P$, i.e., 
%$\langle \lambda , \alpha^{\lor}_i \rangle \geq 0$
%for all $i \in I$.
and set 
\begin{equation*}
S = S_\lambda \eqdef \{ i \in I \ | \  \langle \lambda , \alpha^{\lor}_i \rangle =0 \}
% \subset I
.
\end{equation*}
%%%%%%%%%%%%%%
%%%%%%%%%%%%%%
%begin 2
%%%%%%%%%%%%%%
%%%%%%%%%%%%%%
By the remark just above,
for $u,v \in W^S$,
the value $\pair{\lambda}{\wt^S(\bp)}$ does not depend on the choice of a shortest directed path $\bp$
from $u$ to $v$ in $\QBG^S$;
this value is called the $\lambda$-weight of a directed path from $u$ to $v$ in $\QBG^S$.
Moreover,
we know from \cite[Lemma 7.2]{LNSSS2}
that the value $\pair{\lambda}{\wt^S(\bp)}$ is equal to the value 
$\wt_\lambda (x \Rightarrow y) =  \pair{\lambda}{\wt (x \Rightarrow y) }$ for all $x \in uW_S$ and $y \in vW_S$.
In view of this, for $u, v \in W^S$, we write $\wt_\lambda (u \Rightarrow v)$ also for the value $\pair{\lambda}{\wt^S(\bp)}$
by abuse of notation.

The proof of the following lemma is similar to that of \cite[Lemma 6.1]{LNSSS2}.
\begin{lem}[see {\cite[Lemma 6.1]{LNSSS2}} for untwisted types]\label{lem:6.1}
Let $b \in \mathbb{Q}\cap [0,1]${\rm;} notice that $b$ may be 1.
If $u \rightarrow v$ is an edge of $\QBG^\C_{b\lambda}$,
then there exists a directed path from $\lfloor u \rfloor$ to $\lfloor v \rfloor$ in $(\QBG^\C_{b\lambda})^S$.
\end{lem}

For
$u, v \in W$ (resp., $\in W^S$),
let $\ell (u ,v)$ denote the length of a shortest path in 
$\QBG$ (resp., $\QBG^S$) from $u$ to $v$.
For $w\in W$,
we define the $w$-tilted (dual) Bruhat order $<_{w}$ on $W$ as follows:
for $u,v \in W$,
	\begin{equation*}
		u<_{w} v
		\ardef
		\ell(v , w) =\ell(v , u)+\ell(u, w).
	\end{equation*}

The proof of the following lemma is similar to those of {\cite[Theorem 7.1]{LNSSS1}} and  {\cite[Lemma 6.5]{LNSSS2}}.
\begin{lem}[see {\cite[Theorem 7.1]{LNSSS1}}, {\cite[Lemma 6.5]{LNSSS2}} for untwisted types]\label{8.5}
Let $u, v\in W^S$, and $w \in W_S$.

\begin{enu}

\item
There exists a unique minimal element in the coset $v W_S$
in the $u w$-tilted Bruhat order $<_{u w}$.
We denote it by $\min(v W_S, <_{u w})$.

\item
There exists a unique  directed path from some $x \in v W_S$ to some $uw$ 
in $\QBG$ whose edge labels are increasing 
in the 
total order $\prec$ on $(\Delta^+)^\lor$,
defined at the beginning of \S {\rm2.3},
and lie in $(\Delta^+ \setminus \Delta^+_S)^\lor$.
This path begins with $\min(v W_S, <_{u w})$.

\item
Let
$b\in\mathbb{Q}\cap [0,1]$.
If there exists a directed path from $v$ to $u$ in $\QBG_{b \lambda}^S$ {\rm(}resp., $(\QBG_{b \lambda}^\C)^S${\rm)},
then the directed path in {\rm(2)} is in  $\QBG_{b \lambda}$ {\rm(}resp., $\QBG_{b \lambda}^\C${\rm)}.
\end{enu}
\end{lem}

\section{Quantum Lakshmibai-Seshadri paths and root operators}
\subsection{Quantum Lakshmibai-Seshadri paths}
First of all, 
we define quantum Lakshmibai-Seshadri paths for dual untwisted type.
\begin{dfn}[{see \cite[Definition 3.1]{LNSSS2} for untwisted types}]\label{def_qls}
We fix $(A, A_\aff)$ with $\mathfrak{g} (A_\aff)$ a dual untwisted affine Lie algebra.
Let $\lambda \in P$ be a dominant  weight
and set $S = S_\lambda = \{ i \in I \ | \ \langle \lambda , \alpha^\lor_i \rangle =0 \}$.
A pair $\eta = (w_1, w_2 ,\ldots ,w_s ; \tau_0, \tau_1 , \ldots , \tau_s ) $
of a sequence
$w_1,\ldots , w_s$ of  elements in $W^S$ such that $w_k \neq w_{k+1}$ for $1 \leq k \leq s-1$
 and a increasing sequence  
$0=\tau_0, \ldots , \tau_s=1$ of rational numbers
is called a quantum Lakshmibai-Seshadri ($\QLS$) path of shape $\lambda$
if 

(C)
for every $1\leq i \leq s-1$, there exists a directed  path from $w_{i+1}$ to $w_{i}$ in $\QBG^S_{\tau_i \lambda}$.

\noindent
Let $\QLS(\lambda)$ denote the set of all $\QLS$ paths of shape $\lambda$.
\end{dfn}

\begin{rem}\label{def_qls_rem}
As in \cite[Definition 3.2.2 and Theorem 4.1.1]{LNSSS4},
condition (C) can be replaced by

(C')
for every $1\leq i \leq s-1$, there exists a shortest directed  path in $\QBG^S$ from $w_{i+1}$ to $w_{i}$
which is also a directed path in $\QBG^S_{\tau_i \lambda}$.
\end{rem}

Let $(A, A_\aff) = (B_n, \D)$.
We define a $\C$-type quantum Lakshmibai-Seshadri paths
as a quantum Lakshmibai-Seshadri paths for type $\D$ with a specific condition.
\begin{dfn}\label{def_qls_c}
Let $\lambda \in Q$ be a dominant  weight
and set $S = S_\lambda = \{ i \in I \ | \ \langle \lambda , \alpha^\lor_i \rangle =0 \}$.
A pair $\eta = (w_1, w_2 ,\ldots ,w_s ; \tau_0, \tau_1 , \ldots , \tau_s ) $
of a sequence
$w_1,\ldots , w_s$ of  elements in $W^S$ such that $w_k \neq w_{k+1}$ for $1 \leq k \leq s-1$  and a increasing sequence  
$0=\tau_0, \ldots , \tau_s=1$ of rational numbers
is called a $\C$-type quantum Lakshmibai-Seshadri ($\QLS$) path (or a $\QLS$ path of type $\C$) of shape $\lambda$
if 

(C)
for every $1\leq i \leq s-1$, there exists a directed  path from $w_{i+1}$ to $w_{i}$ in $(\QBG^\C_{\tau_i \lambda})^S$.

\noindent
Let $\QLS^\C(\lambda) \subset \QLS(\lambda)$ denote the set of all $\C$-type $\QLS$ paths of shape $\lambda$.
\end{dfn}

\begin{dfn}
We define a subset $\widetilde{\QLS}^\C(\lambda) \subset \QLS^\C(\lambda)$
by
\begin{align*}
\widetilde{\QLS}^\C(\lambda) \eqdef \{ \eta=(w_1, w_2 ,\ldots ,w_s ; \tau_0, \tau_1 , \ldots , \tau_s )\in \QLS^\C(\lambda) \ | \ \eta \mbox{ satisfies condition (C')}  \},
\end{align*}
where

(C')
for every $1\leq i \leq s-1$, there exists a shortest directed  path in $\QBG^S$ from $w_{i+1}$ to $w_{i}$
which is also a directed path in $(\QBG^\C_{\tau_i \lambda})^S$.
\end{dfn}
Indeed,
the set $\widetilde{\QLS}^\C(\lambda)$ is identical to $\QLS^\C(\lambda)$ (see \S 3.4).

Let $(A, A_\aff) = (C_n, \C)$.
\begin{dfn}\label{QLS_C}
Let $\lambda \in P\pdag$ be a dominant  weight
and set $S = S_\lambda = \{ i \in I\pdag \ | \ \langle \lambda , \alpha^\lor_i \rangle =0 \}$.
A pair $\eta = (w_1, w_2 ,\ldots ,w_s ; \tau_0, \tau_1 , \ldots , \tau_s ) $
of a sequence
$w_1,\ldots , w_s$ of  elements in $(W\pdag)^S$  and a increasing sequence  
$0=\tau_0, \ldots , \tau_s=1$ of rational numbers
is called a $\C$-type quantum Lakshmibai-Seshadri ($\QLS\pdag$) path of shape $\lambda$
if for every $1\leq i \leq s-1$, there exists a directed  path from $w_{i+1}$ to $w_{i}$ in $(\QBG\pdag_{\tau_i \lambda})^S$.

\noindent
Let $\QLS\pdag(\lambda)$ denote the set of all $\C$-type $\QLS\pdag$ paths of shape $\lambda$.
\end{dfn}

\begin{rem}\label{remark_bijection_qls}
It follows from Remark \ref{identification_QBG}
that 
for $\lambda \in P\pdag$,
there exists a bijection 
$\iota^* : \QLS\pdag(\lambda) \rightarrow \QLS^\C(\iota^* (\lambda) )$
which satisfies
$\iota^* ((w_1, \ldots, w_s; \tau_0, \ldots , \tau_s))
=
(\iota^*(w_1), \ldots, \iota^*(w_s); \tau_0, \ldots , \tau_s)$,
where
$\iota^*: W\pdag \rightarrow W$ denotes the group isomorphism defined in Remark \ref{BC}.
Since
in dual untwisted types, a QLS path of type $A_\aff$ is defined in terms of the root system of type $A$,
and in untwisted types, a QLS path of type $A_\aff^{(1)}$ is defined in terms of the root system of type $A$ (see \cite{LNSSS2}),
QLS path of type $\C$ should be defined in terms of the root system of type $C_n$.
However, in this paper, we mainly study $\QLS^\C(\lambda)$, because 
the proofs of all properties on QLS paths of dual untwisted types are similar to that of untwisted types,
and
Orr-Shimozono formula (see \S 4.2) of type $\C$ is described in terms of the root system of types $\D$ and $C_n$.
\end{rem}

\subsection{Root operators}

We fix a pair $(A, A_\aff)$ with $\mathfrak{g} (A_\aff)$ a dual untwisted affine Lie algebra.
We mainly consider the case $(A, A_\aff)=(B_n , \D)$.
In what follows,
we use the following notation:
for $i \in I_\aff$, 
\begin{equation*}
\overline{\alpha}_i \eqdef
\left\{
			\begin{array}{ll}
				-\theta  & \ \ \mbox{for}\ i=0  ,  \\
				\alpha_i & \ \ \mbox{for}\ i \in I  ,
			\end{array}
		\right.
\mbox{ and }
r_i \eqdef
\left\{
			\begin{array}{ll}
				s_\theta  & \ \ \mbox{for}\ i=0  ,  \\
				s_i & \ \ \mbox{for}\ i \in I .
			\end{array}
		\right. 
\end{equation*}
%Here, $\theta$ is the highest root of $\Delta$.
%We remark that $\xi (\overline {\alpha}_i) =\overline{ \alpha }_i $ for
%$i \in I_\aff$.
%note that 
%$\xi (\overline {\alpha}_i) =\overline{ \alpha }_i $ for $i \in I_\aff$.

For a piecewise-linear, continuous (PLC for short) map $\pi:[0,1]\rightarrow \mathfrak{h}^*_\mathbb{R} $,
we define a function $H(t)= H_i^\pi (t)$ on [0,1] by $H(t)=H^\eta_i (t) \eqdef \langle \eta(t) , \left( \overline{\alpha}_i \right)^\lor \rangle $,
$t \in [0,1]$, and set $m = m^\eta_i \eqdef \min \{ H^\eta_i (t) \ | \ t \in [0,1]  \}$.
Let $\mathbb{B}_\inte$ be the set of all PLC maps $[0,1]\rightarrow \mathfrak{h}^*_\mathbb{R}$ whose all local minima of the function $H^\pi_i(t)$,
$t\in [0,1]$ are integers.

We fix a dominant weight $\lambda \in P$.
\begin{rem}\label{piecewiseliner}
We identify an element $\eta = (v_1 , \ldots , v_s ; \sigma_0 , \sigma_1 , \ldots ,\sigma_s ) \in \QLS(\lambda)$
with the following PLC map $\eta : [0,1]\rightarrow \mathfrak{h}^*_\mathbb{R} = \bigoplus_{i \in I}\mathbb{R}\alpha_i$:
\begin{align*}
\eta(t) \eqdef \sum_{k=1}^{p-1} (\sigma_k -\sigma_{k-1})v_k \lambda +(t - \sigma_{p-1})v_p \lambda
\ \ {\rm for} \ \sigma_{p-1} \leq t \leq \sigma_p , 1\leq p\leq s.
\end{align*}
Then, $\eta \in \mathbb{B}_\inte$ (see Lemma \ref{4.1.12}).
We set $\wt (\eta) \eqdef \eta(1)$.
In fact, $\wt (\eta) \in P$ (see Lemma \ref{4.1.7}).
\end{rem}

Let $i\in I_\aff$.
We define the root operators
$e_i ,f_i : \mathbb{B}_\inte \rightarrow \mathbb{B}_\inte \sqcup \setzero$ as follows.

If $m=m^\eta_i=0$,  then $e_i  \eta \eqdef {\bf 0} $.
If $m\leq -1$, then we set
\begin{align*}
t_1 &\eqdef \min \{ t \in [0,1] \ | \ H(t) = m \}  , \\
t_0 &\eqdef \max \{ t \in [0,t_1] \ | \ H(t) = m+1 \} ,
\end{align*}
and define $e_i \eta $ by
\begin{equation*}
e_i \eta (t) \eqdef
\left\{
			\begin{array}{ll}
				\eta(t)  & \ \ \mbox{for}\ t\in [0, t_0]  ,  \\
				\eta(t_0)+r_i (\eta(t)-\eta(t_0) ) & \ \ \mbox{for}\  t\in [t_0, t_1] , \\
				\eta(t) + \overline{\alpha}_i & \ \ \mbox{for}\  t \in [t_1 ,1].
			\end{array}
		\right. 
\end{equation*}

Similarly, we define $f_i$ as follows.
If $m-H(1)  =0$, 
then $f_i  \eta \eqdef \zero $. 
Otherwise, we set
\begin{align*}
t_2 & \eqdef \max \{ t \in [0,1] \ | \ H(t) = m \}, \\
t_3 & \eqdef \min \{ t \in [t_2,1] \ | \ H(t) = m+1 \} ,
\end{align*}
and define $f_i \eta $ by 
\begin{equation*}
f_i \eta (t) \eqdef
\left\{
			\begin{array}{ll}
				\eta(t)  & \ \ \mbox{for}\ t\in [0, t_2]  ,  \\
				\eta(t_2)+r_i (\eta(t)-\eta(t_2) ) & \ \ \mbox{for}\ t\in [t_2, t_3] , \\
				\eta(t) - \overline{\alpha}_i & \ \ \mbox{for}\ t \in [t_3 ,1].
			\end{array}
		\right. 
\end{equation*}
Then, it is a well-known fact that
$e_i \eta, f_i \eta \in \mathbb{B}_\inte \sqcup \setzero$ for $\eta \in \QLS(\lambda)$.
Moreover,
\begin{prop}[{see \cite[Theorem 4.1.12]{LNSSS4} for untwisted types}]
The set $\QLS(\lambda) \sqcup \setzero$ is stable under the action of the root operators $e_i$ and $f_i$ for all $i \in I_\aff$.
\end{prop}

Also, for  $i \in I_\aff$,
we define $\varepsilon_i , \varphi_i : \QLS(\lambda) \rightarrow \mathbb{Z}$ by
$\varepsilon_i (\eta) \eqdef \max\{ k \in \mathbb{Z}_{\geq 0} \ | \ e^k_i \eta \neq {\bf 0} \}$, 
$\varphi_i (\eta) \eqdef \max\{ k \in \mathbb{Z}_{\geq 0} \ | \ f^k_i \eta \neq {\bf 0} \}$.
Then
we see that 
\begin{equation}\label{epsilon_phi}
\varepsilon_i (\eta) = m^\eta_i \ \mbox{ and } \ \varphi_i (\eta)=H_i^\eta (1) - m^\eta_i.
\end{equation}

The proof of the following proposition is similar to that of \cite[Theorem 4.1.1]{LNSSS4}.
%\begin{prop}[{see \cite[Theorem 4.1.1]{LNSSS4}, \cite[Theorem 3.2]{NS}, \cite[Remark 2.15]{Nakajima}}]\label{crystal}
\begin{prop}[{see \cite[Theorem 4.1.1]{LNSSS4} for untwisted types}]\label{crystal}
The tuple $( \QLS(\lambda), \overline{\xi} \circ \wt , e_i ,f_i , \varepsilon_i , \varphi_i \ (i\in I_\aff ) )$
is a $U'_q(\mathfrak{g}(\D))$-crystal.
Moreover,
it is a realization of the crystal basis of a particular quantum Weyl module $W_q (\overline{\xi}(\lambda))$
over a quantum affine algebra $U'_q(\mathfrak{g}(\D))$.
\end{prop}
\begin{rem}
Proposition~\ref{crystal} also holds for every dual untwisted types.
\end{rem}

\begin{rem}\label{simple_crystal_dual_untwisted}
We set $\eta_\lambda \eqdef (e; 0,1) \in \QLS(\lambda)$.
By \cite[Remark 2.3.6 (2)]{LNSSS4},
for $\eta \in \QLS(\lambda)$,
there exist $i_1 , \ldots , i_k \in I_\aff$ such that
$e_{i_1}^{\max} \cdots e_{i_k}^{\max} \eta = \eta_\lambda$.
Also, 
there exist $i_1 , \ldots , i_k \in I_\aff$ such that
$f_{i_1}^{\max} \cdots f_{i_k}^{\max} \eta = \eta_\lambda$.
\end{rem}

We consider the case $(A, A_\aff) = (C_n , \C)$.
In this case, we use all the notation which we set above,
putting the dagger for $(C_n , \C)$
except $\overline{\alpha}_0\pdag$;
here we set $\overline{\alpha}_0\pdag \eqdef - \half \theta\pdag (= - \theta)$.
Thus we obtain the root operators $e\pdag_i$ and $f\pdag_i$, $i \in I_\aff$, acting on the set $\mathbb{B}\pdag_\inte$.

We fix a dominant weight $\lambda^\dag \in P\pdag$.
We set
\begin{equation}\label{chi}
\chi_i \eqdef \left\{
\begin{array}{ll}
1 &\mbox{ if }i \in I_\aff \setminus \{n\},\\
2 &\mbox{ if }i= n.
\end{array}\right.
\end{equation}
We remark that %$\chi_i = c_{\alpha_i}$ for $i = 1, \ldots , n$, where $c_{\alpha_i}$ is defined in \S 2.
$\iota^*(\overline{\alpha}^\dag_i) = \chi_i \overline{\alpha}_i$
and
$\iota ((\overline{\alpha}^\dag_i)^\lor) = \frac{1}{\chi_i} \overline{\alpha}_i^\lor$ for $i \in I_\aff$
by Remark \ref{BC}.

\begin{lem}
When we see an element $\eta \in \QLS\pdag(\lambda^\dag)$ as a PLC map $\eta(t): [0,1] \rightarrow (\mathfrak{h}\pdag_\mathbb{R})^*$,
all local minima of $H_i^\eta (t) = \langle \eta(t),  (\overline{\alpha}\pdag_i)^\lor \rangle$, $t \in [0,1]$, are integers{\rm;}
that is, $\QLS\pdag(\lambda^\dag) \in \mathbb{B}^\dag_\inte$.
\end{lem}

\begin{proof}
In the proof, we assume that Lemma \ref{4.1.12} holds; we will prove the Lemma in \S 3.3.
Let $\eta = (w_1, \ldots, w_s; \sigma_0, \ldots, \sigma_s) \in \QLS\pdag(\lambda^\dag)$.
Then we have
\begin{equation*}
\iota^* (\eta) = (\iota^*(w_1), \ldots, \iota^*(w_s);\sigma_0, \ldots, \sigma_s) \in \QLS^\C (\iota^*(\lambda)),
\end{equation*}
where  $\iota^* (\eta)$ is defined in Remark~\ref{remark_bijection_qls}.
By Lemma \ref{4.1.12}, all local minima of $H_i^{\iota^*(\eta)} (t) = \langle \iota^*( \eta) (t), \overline{\alpha}_i^\lor \rangle = \langle \iota^*( \eta (t) ), \overline{\alpha}_i^\lor \rangle $
lie in $\chi_i \mathbb{Z}$.
Also, it follows from Remark \ref{BC} that
$\iota ((\overline{\alpha}_i^\dag)^\lor) = \frac{1}{\chi_i}\overline{\alpha}^\lor_i$ for $i \in I_\aff$.
Therefore, all local minima of  $H_i^\eta (t) = \langle \eta(t), (\overline{\alpha}\pdag_i)^\lor \rangle =\frac{1}{\chi_i} \pair{\iota^*(\eta)}{\overline{\alpha}^\lor_i}$, $t \in [0,1]$, are integers.
\end{proof}

We define an injective map $\iota^*: \mathbb{B}\pdag_\inte \rightarrow \mathbb{B}_\inte$
by $\eta(t) \mapsto \iota^*(\eta(t))$, $t \in [0,1]$.
The bijection $\iota^* : \QLS\pdag(\lambda^\dag) \rightarrow \QLS^\C (\iota^*(\lambda^\dag))$
defined in Remark \ref{remark_bijection_qls}
is identical to the restriction of the injective map $\iota^*: \mathbb{B}\pdag_\inte \rightarrow \mathbb{B}_\inte$ to $\QLS\pdag(\lambda)$.

Since the root operators $e\pdag_i$ and $f\pdag_i$, $i \in I_\aff$, act on $\mathbb{B}\pdag_\inte$,
let us define $\widetilde{e}\pdag_i$ and $\widetilde{f}\pdag_i$, $i \in I_\aff$, acting on $\mathbb{B}_\inte$
which are compatible with $\iota^* : \mathbb{B}\pdag_\inte \rightarrow \mathbb{B}_\inte$.

If $m=m^\eta_i > -\chi_i $,  then $\widetilde{e}\pdag_i  \eta \eqdef {\bf 0} $.
If $m\leq -\chi_i$, then we set
\begin{align*}
t_1 				&\eqdef \min \{ t \in [0,1] \ | \ H(t) = m \}  , \\
t_0^\chi&\eqdef \max \{ t \in [0,t_1] \ | \ H(t) = m + \chi_i \} ,
\end{align*}
and define $\widetilde{e}\pdag_i \eta $ by
\begin{equation*}
\widetilde{e}\pdag_i \eta (t) \eqdef
\left\{
			\begin{array}{ll}
				\eta(t)  & \ \ \mbox{for}\ t\in [0, t_0]  ,  \\
				\eta(t_0)+r_i (\eta(t)-\eta(t_0) ) & \ \ \mbox{for}\  t\in [t_0, t_1^\chi] , \\
				\eta(t) + \chi_i \overline{\alpha}_i & \ \ \mbox{for}\  t \in [t_1^\chi ,1].
			\end{array}
		\right. 
\end{equation*}

Similarly, we define $\widetilde{f}\pdag_i$ as follows.
If $m - H(1) > - \chi_i $, 
then $\widetilde{f}\pdag_i  \eta \eqdef {\bf 0} $. 
Otherwise, we set
\begin{align*}
t_2 & \eqdef \max \{ t \in [0,1] \ | \ H(t) = m \}, \\
t_3^\chi & \eqdef \min \{ t \in [t_2,1] \ | \ H(t) = m+\chi_i \} ,
\end{align*}
and define $\widetilde{f}\pdag_i \eta $ by 
\begin{equation*}
\widetilde{f}\pdag_i \eta (t) \eqdef
\left\{
			\begin{array}{ll}
				\eta(t)  & \ \ \mbox{for}\ t\in [0, t_2]  ,  \\
				\eta(t_2)+r_i (\eta(t)-\eta(t_2) ) & \ \ \mbox{for}\ t\in [t_2, t_3^\chi] , \\
				\eta(t) - \chi_i \overline{\alpha}_i & \ \ \mbox{for}\ t \in [t_3^\chi  ,1].
			\end{array}
		\right. 
\end{equation*}
Then we see that 
\begin{equation}\label{compatible}
\iota^* \circ e\pdag_i \eta =\widetilde{e}\pdag_i \circ \iota^* \eta, \mbox{ and } \iota^* \circ f\pdag_i \eta =\widetilde{f}\pdag_i \circ \iota^*\eta
\end{equation}
for $i \in I_\aff$, $\eta \in \mathbb{B}\pdag_\inte$.

\begin{rem}\label{root_operator_identification}
Let $\eta \in \mathbb{B}_\inte$ 
and $i \in I_\aff$ be such that $\widetilde{e}\pdag_i \eta \neq \zero$ (resp., $\widetilde{f}\pdag_i \eta \neq \zero$).
By the definition of $\widetilde{e}\pdag_i$ (resp., $\widetilde{f}\pdag_i$), we have
$\widetilde{e}\pdag_i \eta = e_i^{\chi_i} \eta$ (resp., $\widetilde{f}\pdag_i \eta = f_i^{\chi_i} \eta$);
here $e_i$ and $f_i$ are root operators on $\mathbb{B}_\inte$.
\end{rem}

For  $i \in I_\aff$,
we define $\widetilde{\varepsilon}\pdag_i , \widetilde{\varphi}\pdag_i : \QLS(\iota^*(\lambda^\dag)) \rightarrow \mathbb{Z}$ by
$\widetilde{\varepsilon}\pdag_i (\eta) \eqdef \max\{ k \in \mathbb{Z}_{\geq 0} \ | \ (\widetilde{e}\pdag_i)^k \eta \neq {\bf 0} \}$, 
$\widetilde{\varphi}\pdag_i (\eta) \eqdef \max\{ k \in \mathbb{Z}_{\geq 0} \ | \ (\widetilde{f}\pdag_i)^k \eta \neq {\bf 0} \}$,
and set $(\widetilde{e}\pdag_i)^{\max}\eta \eqdef (\widetilde{e}\pdag_i)^{\widetilde{\varepsilon}\pdag_i (\eta)}\eta$ and 
$(\widetilde{f}\pdag_i)^{\max}\eta \eqdef (\widetilde{f}\pdag_i)^{\widetilde{\varphi}\pdag_i (\eta)}\eta$.
Then
we see that 
\begin{align}\label{epsilon_phi_C}
\widetilde{\varepsilon}\pdag_i (\eta) = \left\{
\begin{array}{ll}
m^\eta_i & \mbox{ if }i\neq n ,\\
\lfloor \frac{m^\eta_i}{2}\rfloor & \mbox{ if }i = n,
\end{array} \right.
\widetilde{\varphi}\pdag_i (\eta) = \left\{
\begin{array}{ll}
H_i^\eta (1) - m^\eta_i & \mbox{ if }i\neq n ,\\
\lfloor \frac{H_i^\eta (1) - m^\eta_i}{2}\rfloor & \mbox{ if }i = n.
\end{array} \right.
\end{align}

The following Theorem will be proved in \S 3.4.
\begin{thm}\label{stable}
For a dominant weight $\lambda^\dag \in P^\dag$,
the sets $\QLS^\C(\iota^*(\lambda^\dag)) \sqcup \setzero$ and $\widetilde{\QLS}^\C(\iota^*(\lambda^\dag)) \sqcup \setzero$  
are stable under the action of the root operators $\widetilde{e}\pdag_i$ and $\widetilde{f}\pdag_i$ for all $i \in I_\aff$.
\end{thm}

%Let $\eta_\lambda \eqdef (e; 0,1)$.
We define a subset $\mathbb{B}\pdag(\lambda^\dag)_\cl$ of $\mathbb{B}\pdag_\inte$ by
\begin{equation*}
\eta \in \mathbb{B}\pdag(\lambda^\dag)_\cl \ardef \mbox{there exists a sequence }(i_1, \ldots, i_k)
\mbox{ such that }(f_{i_1}\pdag)^{\max} \dots (f_{i_k}\pdag)^{\max}\eta =  \eta_{\lambda^\dag} .
\end{equation*}
By \cite[Theorem 2.4.1]{LNSSS4},
the set $\mathbb{B}\pdag( \lambda^\dag )_\cl$
is a realization of the crystal basis $\mathcal{B}(\overline{\xi}^\dag(\lambda^\dag))$ of a particular quantum Weyl module
$W_q (\overline{\xi}^\dag(\lambda^\dag))$
over a quantum affine algebra $U_q'(\mathfrak{g}(\C))$.
It follows from \eqref{compatible} that
$\iota^*(\mathbb{B}\pdag( \lambda^\dag )_\cl)=\widetilde{\mathbb{B}}( \iota^*(\lambda^\dag) )_\cl$ is also a realization of the $U_q'(\mathfrak{g}(\C))$-crystal $\mathcal{B}(\overline{\xi}^\dag (\lambda^\dag))$,
where 
for a dominant weight $\lambda \in Q$,
$\widetilde{\mathbb{B}}(\lambda)_\cl \subset \mathbb{B}_\inte$ is defined by
\begin{equation*}
\eta \in \widetilde{\mathbb{B}}( \lambda )_\cl \ardef \mbox{there exists a sequence }(i_1, \ldots, i_k)
\mbox{ such that } (\widetilde{f}\pdag_{i_1})^{\max} \dots (\widetilde{f}\pdag_{i_k})^{\max} \eta =\eta_{\lambda } .
\end{equation*}

We will prove the following theorem in \S 3.4.
\begin{thm}\label{realization_theorem}
Let $\lambda^\dag \in P^\dag$ be a dominant weight.
$\QLS^\C( \iota^*(\lambda^\dag) ) = \widetilde{\QLS}^\C( \iota^*(\lambda^\dag) )= \widetilde{\mathbb{B}}( \iota^*(\lambda^\dag) )_\cl${\rm;}
that is, the set $\QLS^\C ( \iota^*(\lambda^\dag) )= \widetilde{\QLS}^\C( \iota^*(\lambda^\dag) )$ is a realization of the $U_q'(\mathfrak{g}(\C))$-crystal $\mathcal{B}(\overline{\xi}^\dag(\lambda^\dag))$.
\end{thm}

\subsection{Some technical lemmas}

In this subsection, we state some lemmas for dual untwisted types and type $\C$, but prove them only for $\C$;
the proofs of them for dual untwisted types are similar to those of lemmas in \cite[\S 4.1]{LNSSS4}.
We fix a dominant weight $\lambda \in P$ for a dual untwisted type
(especially for type $\D$), or a dominant weight $\lambda \in Q$ for type $\C$;
that is, 
if we consider $\QBG_{b \lambda}$, $\QBG_{b \lambda}^S$, and $\QLS(\lambda)$, then $\lambda \in P$,
and if we consider $\QBG_{b \lambda}^\C$, $(\QBG_{b \lambda}^\C)^S$, and $\QLS^\C(\lambda)$, then $\lambda \in Q$,
where $P$ denote the weight lattice of $\mathfrak{g}(B_n)$, and $Q$ the root lattice of  $\mathfrak{g}(B_n)$.
Also, we set $S=S_\lambda = \{ i \in I \ | \ \pair{\lambda}{\alpha_i^\lor}=0 \}$.

\begin{lem}[\normalfont{see \cite[Lemma 4.1.6]{LNSSS4}} for untwisted types]\label{4.1.6}
Let $u, v \in W^S$, $b \in \mathbb{Q}\cap [0,1]$, 
and $G = \QBG_{b\lambda}^S$ or $(\QBG_{b\lambda}^\C)^S$. 
Let
\begin{equation*}
u = x_0 \xleftarrow{\gamma_1^\lor} x_1 \xleftarrow{\gamma_2^\lor} \cdots \xleftarrow{\gamma_r^\lor} x_r =v
\end{equation*}
be a directed path from $v$ to $u$ in $G$.
Then $b (x \lambda -y \lambda) \in Q$.
\end{lem}
\begin{proof}
This is obvious even for $G = (\QBG_{b\lambda}^\C)^S$ since $(\QBG_{b\lambda}^\C)^S \subset \QBG_{b\lambda}^S$.
\end{proof}

\begin{lem}[\normalfont{see \cite[Lemma 4.1.7]{LNSSS4}} for untwisted types]\label{4.1.7}
If $\eta \in \QLS(\lambda)$ {\rm(}resp., $\in \QLS^\C(\lambda)${\rm)}, then $\eta(1)$ is contained in $\lambda + Q$.
\end{lem}

\begin{proof}
Even for $\eta \in \QLS^\C(\lambda)$, this is obvious since $\QLS^\C(\lambda) \subset \QLS(\lambda)$.
\end{proof}

\begin{lem}[\normalfont{see \cite[Lemma 4.1.8]{LNSSS4}} for untwisted types]\label{4.1.8}
Let $u, v \in W^S$, %$b \in \mathbb{Q}\cap [0,1]$, 
and 
\begin{equation}\label{eq_4.1.1}
u = x_0 \xleftarrow{\gamma_1^\lor} x_1 \xleftarrow{\gamma_2^\lor} \cdots \xleftarrow{\gamma_r^\lor} x_r =v
\end{equation}
be a directed path from $v$ to $u$ in $\QBG^S$.
Let $i \in I_\aff$.
\begin{enu}
\item
If there exists $1 \leq p \leq r$ such that $\langle x_k\lambda, \overline{\alpha}_i^\lor \rangle < 0$ for all $0 \leq k \leq p-1$
and $\langle x_p\lambda,  \overline{\alpha}_i^\lor\rangle \geq 0$,
then
$\lfloor r_i x_{p-1} \rfloor = x_p$,
and there exists a directed path from $v$ to $r_i u$ of the form{\rm:}
\begin{equation}\label{eq_4.1.2}
\lfloor r_i u \rfloor = \lfloor r_i x_0 \rfloor
\xleftarrow{z_1 \gamma_1^\lor} \cdots \xleftarrow{z_{p-1} \gamma_{p-1}^\lor}
\lfloor r_i x_{p-1} \rfloor = x_p
\xleftarrow{\gamma_{p+1}^\lor} \cdots \xleftarrow{\gamma_r^\lor}
x_r =v.
\end{equation}
Here,
if $ i \in I$,
then we define $z_k = e$ for all $1 \leq k \leq p-1${\rm;}
if $ i = 0$,
then we define $z_k \in W_S$ by 
$s_\theta x_k = \lfloor s_\theta x_k \rfloor z_k$
 for all $1 \leq k \leq p-1$.

\item
If the directed path \eqref{eq_4.1.1} from $v$ to $u$ is shortest,
i.e., $\ell(v, u)= r$,
then the directed path \eqref{eq_4.1.2} from $v$ to $\lfloor r_i u \rfloor$ is also shortest,
i.e., $\ell(v, \lfloor r_i u \rfloor) = r-1$.

\item
Let
$b \in \mathbb{Q}\cap [0,1]$,
and suppose that the directed path \eqref{eq_4.1.1} is a path in $\QBG_{b \lambda}^S$.
Then the directed path \eqref{eq_4.1.2} is also a path in $\QBG_{b \lambda}^S$.

\item
Let
$b \in \mathbb{Q}\cap [0,1]$, and suppose that
 the directed path \eqref{eq_4.1.1} is a path in $(\QBG_{b \lambda}^\C)^S$.
Then the directed path \eqref{eq_4.1.2} is also a path in $(\QBG_{b \lambda}^\C)^S$,
and $\lfloor r_i u \rfloor \rightarrow u$ for $i \in I$ {\rm(}resp., $i=0${\rm)}
is a Bruhat {\rm(}resp., quantum{\rm)} edge of $(\QBG_{b \lambda}^\C)^S$.
\end{enu}
\end{lem}

\begin{proof}
(1)
The proof of (1), (2), and (3) is similar to that of \cite[Lemma 4.1.8]{LNSSS4};
here we introduce their proof of (1)
in order to prove (4).
Assume that $i \in I$.
The proof for the case $i =0$ is similar;
replace $\alpha_i$ and $\alpha_i^\lor$ by $- \theta$ and $- \theta^\lor$, respectively,
and use Lemmas \ref{diamond} (3), (4) and \ref{2.3.3.5} (3) instead of Lemmas \ref{diamond} (1), (2) and \ref{2.3.3.5} (1).
First, we show that $x_k \gamma_k \neq \pm \alpha_i$ for all $1 \leq k \leq p-1$.
Suppose that $x_k \gamma_k = \pm \alpha_i$ for some $1 \leq k \leq p-1$.
Then, $x_{k-1}\lambda = x_k s_{\gamma_k}\lambda = s_{x_k \gamma_k}x_k \lambda = r_i x_k \lambda$,
and hence $\langle x_{k-1}\lambda, \alpha_i^\lor \rangle = \langle r_i x_k \lambda, \alpha_i^\lor \rangle = -\langle x_k \lambda, \alpha_i^\lor \rangle >0$, which contradicts our assumption.
Thus $x_k \gamma_k \neq \pm \alpha_i$ for $1 \leq k \leq p-1$.
By Lemma \ref{diamond} (1), (2) for $b=1$, we see that $\lfloor r_i x_{k-1} \rfloor \xrightarrow{\gamma_k^\lor} \lfloor r_i x_k \rfloor$ 
is an edge of $\QBG^S$ for all $1 \leq k \leq p-1$.
Also, since $\langle x_{p-1}\lambda, \alpha_i^\lor \rangle <0$, and  $\langle x_p \lambda, \alpha_i^\lor \rangle \geq 0$, 
it follows from Lemma \ref{2.3.3.5} (1) that $x_p \gamma_p = \pm \alpha_i$,
and hence 
$r_i x_p \lambda = r_i x_p r_{\gamma_p} \lambda = r_i s_{x_p \gamma_p}x_p \lambda = r_i^2 x_p \lambda = x_p \lambda$,
which implies $\lfloor r_i x_{p-1} \rfloor = x_p$.
Therefore, we obtain a directed path of the form (\ref{eq_4.1.2}) from $v$ to $\lfloor r_i u \rfloor $.

Let us show (4).
Since $\lfloor r_i x_{p-1} \rfloor = x_p$ and $x_{p-1}\xleftarrow{\gamma_p^\lor} x_p$ is an edge of $(\QBG_{b \lambda}^\C)^S$,
we see that $x_{p-1} \xleftarrow{\gamma_p^\lor} \lfloor s_i x_{p-1} \rfloor$ is an edge of $(\QBG_{b \lambda}^\C)^S$;
note that $x_{p-1} \xleftarrow{\gamma_p^\lor} \lfloor r_i x_{p-1} \rfloor$ is a Bruhat (resp., quantum) edge for $i \in I$ (resp., $i=0$)
by Lemma \ref{leftaction} (1) (resp., (2)).
Hence by Lemma \ref{diamond}, 
for $1 \leq k \leq p-1$,
$\lfloor r_i x_{k-1} \rfloor \xrightarrow{\gamma_k^\lor} \lfloor r_i x_k \rfloor$ is an edge of $(\QBG_{b \lambda}^\C)^S$,
and $x_{k-1} \leftarrow \lfloor r_i x_{k-1} \rfloor$ is a Bruhat  (resp., quantum) edge of $(\QBG_{b \lambda}^\C)^S$ for $i \in I$ (resp., $i=0$).
Thus the directed path (\ref{eq_4.1.2}) is a path in $(\QBG_{b \lambda}^\C)^S$, and 
$\lfloor r_i u \rfloor \rightarrow u$ is a Bruhat  (resp., quantum) edge of $(\QBG_{b \lambda}^\C)^S$ for $i \in I$ (resp., $i=0$).
\end{proof}

The following lemma can be shown in the same way as Lemma \ref{4.1.8}.
\begin{lem}[\normalfont{see \cite[Lemma 4.1.9]{LNSSS4}} for untwisted types]\label{4.1.9}
Keep the notation and setting in Lemma $\ref{4.1.8}$.
\begin{enu}
\item
If there exists $1 \leq p \leq r$ such that $\langle x_k\lambda, \overline{\alpha}_i^\lor \rangle > 0$ for all $p \leq k \leq r$
and $\langle x_{p-1}\lambda,  \overline{\alpha}_i^\lor\rangle \leq 0$,
then
$x_{p-1} = \lfloor r_i x_{p} \rfloor$,
and there exists a directed path from $\lfloor r_i v \rfloor$ to $u$ of the form{\rm:}
\begin{equation}\label{eq_4.1.3}
u =  x_0 
\xleftarrow{ \gamma_1^\lor} \cdots \xleftarrow{\gamma_{p-1}^\lor}
x_{p-1} = \lfloor r_i x_{p} \rfloor
\xleftarrow{z_{p+1} \gamma_{p+1}^\lor} \cdots \xleftarrow{z_r \gamma_r^\lor}
\lfloor r_i x_r \rfloor =\lfloor r_i v \rfloor.
\end{equation}
Here,
if $ i \in I$,
then we define $z_k = e$ for all $p+1 \leq k \leq r${\rm;}
if $ i = 0$,
then we define $z_k \in W_S$ by 
$s_\theta x_k = \lfloor s_\theta x_k \rfloor z_k$
 for all $p+1 \leq k \leq r$.

\item
If the directed path \eqref{eq_4.1.1} from $v$ to $u$ is shortest,
i.e., $\ell(v, u)= r$,
then the directed path \eqref{eq_4.1.3} from $\lfloor r_i v \rfloor$ to $u$ is also shortest,
i.e., $\ell(\lfloor r_i v \rfloor, u) = r -1$.

\item
Let
$b \in \mathbb{Q}\cap [0,1]$,
and suppose that the directed path \eqref{eq_4.1.1} is a path in $\QBG_{b \lambda}^S$.
Then the directed path \eqref{eq_4.1.3} is also a path in $\QBG_{b \lambda}^S$.

\item
Let
$b \in \mathbb{Q}\cap [0,1]$, and suppose that
 the directed path \eqref{eq_4.1.1} is a path in $(\QBG_{b \lambda}^\C)^S$.
Then the directed path \eqref{eq_4.1.3} is also a path in $(\QBG_{b \lambda}^\C)^S$,
and $v \rightarrow \lfloor r_i v \rfloor$ for $i \in I$ {\rm(}resp., $i=0${\rm)}
is a Bruhat {\rm(}resp., quantum{\rm)} edge of $(\QBG_{b \lambda}^\C)^S$.
\end{enu}
\end{lem}
%\footnote{括弧類の修正はここまで}
\begin{lem}[\normalfont{see \cite[Lemma 4.1.10]{LNSSS4}} for untwisted types]\label{4.1.10}
\ \\
\vspace{0mm}
Let $\eta = (w_1, \ldots , w_s; \sigma_0, \ldots , \sigma_s) \in \QLS(\lambda)$ {\rm(}resp., $\in  \QLS^\C(\lambda)${\rm)}.
Let $i \in I_\aff$ and $1 \leq u \leq s-1$ be such that $\langle w_{u+1} \lambda , \overline{\alpha}_i^\lor \rangle>0$.
Let
\begin{equation*}
w_u = x_0 \xleftarrow{\gamma_1^\lor} x_1 \xleftarrow{\gamma_2^\lor} \cdots \xleftarrow{\gamma_r^\lor} x_r = w_{u+1}
\end{equation*}
be a directed path from $w_{u+1}$ to $w_u$ in $\QBG_{\sigma_u \lambda}^S$ {\rm(}resp., $(\QBG_{\sigma_u \lambda}^\C)^S${\rm)}.
If there exists $0 \leq k < r$ such that $\langle x_k \lambda , \overline{\alpha}_i^\lor \rangle \leq 0$,
then $H_i^\eta (\sigma_u) \in \mathbb{Z}$ {\rm(}resp., $\in \chi_i \mathbb{Z}${\rm)}.
In particular, if $\langle w_u \lambda, \alpha_i^\lor \rangle \leq 0$, then $H_i^\eta (\sigma_u) \in \mathbb{Z}$  {\rm(}resp., $\in \chi_i \mathbb{Z}${\rm)}.
\end{lem}

\begin{proof} 
%The proof for $\QLS(\lambda)$ is similar to that of \cite[Lemma 4.1.10]{LNSSS4}.
Since $\eta = (w_1, \ldots , w_s; \sigma_0, \ldots, \sigma_s) \in \QLS^\C (\lambda)$,
we see from the definition of $\C$-type $\QLS$ path that 
$\eta' = (w_1, \ldots , w_u, w_{u+1}; \sigma_0, \ldots , \sigma_{u}, \sigma_s) \in \QLS(\lambda)^\C$.
Observe that $\eta'(t)=\eta(t)$ for $0 \leq t \leq \sigma_{u+1}$,
and hence $H_i^{\eta'}(t)=H_i^\eta (t)$ for $0 \leq t \leq \sigma_{u+1}$.
Also, it follows that
$H_i^\eta (\sigma_u) = H_i^{\eta'}(\sigma_u)= H_i^{\eta'}(1)-(1 -\sigma_{u}) \langle w_{u+1}\lambda, \overline{\alpha}_i^\lor \rangle$.
Since $\eta'(1) \in Q $ by Lemma \ref{4.1.7}, and since $w_{u+1}\lambda \in Q$,
we have $H_i^{\eta'}(1), \pair{w_{u+1}\lambda}{\overline{\alpha}^\lor_i} \in \chi_i \mathbb{Z}$.
Hence it suffices to show that $\sigma_{u} \langle w_{u+1}\lambda, \overline{\alpha}_i^\lor \rangle \in \chi_i \mathbb{Z}$.

We deduce from Lemma \ref{4.1.9} (4) that 
for $i \in I$ (resp., $i=0$),
$w_{u+1} \rightarrow \lfloor r_i w_{u+1} \rfloor$
is a Bruhat (resp., quantum) edge of $(\QBG_{\sigma_u \lambda}^\C)^S$.
Hence if $i \in I$,
then $\sigma_u \langle \lambda, w_{u+1}^{-1}\alpha_i^\lor \rangle \in \chi_i \mathbb{Z}$;
if $i=0$,
then $\sigma_u \langle \lambda, w_{u+1}^{-1}(-\theta)^\lor \rangle \in \mathbb{Z} = \chi_0 \mathbb{Z}$.
Therefore, in both case, we have $(1 -\sigma_{u}) \langle w_{u+1} \lambda, \overline{\alpha}_i^\lor \rangle \in \chi_i \mathbb{Z}$.
\end{proof}

The following Lemma can be shown in the same way as Lemma \ref{4.1.10}.
\begin{lem}[\normalfont{see \cite[Lemma 4.1.11]{LNSSS4}} for untwisted types]\label{4.1.11}
\ \\
\vspace{0mm}
Let $\eta = (w_1 , \ldots, w_s ; \sigma_0, \ldots, \sigma_s) \in \QLS(\lambda)$ {\rm(}resp., $\in  \QLS(\lambda)^\C${\rm)}.
Let $i \in I_\aff$ and $1 \leq u \leq s-1$ be such that $\langle w_u\lambda, \overline{\alpha}_i^\lor \rangle<0$.
Let
\begin{equation*}
w_u = x_0 \xleftarrow{\gamma_1^\lor} x_1 \xleftarrow{\gamma_2^\lor} \cdots \xleftarrow{\gamma_r^\lor} x_r = w_{u+1}
\end{equation*}
be a directed path from $w_{u+1}$ to $w_{u}$ in $\QBG_{\sigma_u \lambda}^S$ {\rm(}resp., $(\QBG_{\sigma_u \lambda}^\C)^S${\rm)}.
If there exists $0 < k \leq r$ such that $\langle x_k \lambda, \overline{\alpha}_i^\lor \rangle \geq 0$,
then $H_i^\eta (\sigma_u) \in \mathbb{Z}$ {\rm(}resp., $\in \chi_i \mathbb{Z}${\rm)}.
In particular, if $\langle w_{u+1}\lambda, \overline{\alpha}^\lor_i \rangle \geq 0$, then $H_i^\eta (\sigma_u) \in \mathbb{Z}$ {\rm(}resp., $\in \chi_i \mathbb{Z}${\rm)}.
\end{lem}

\begin{prop}[\normalfont{see \cite[Proposition 4.1.12]{LNSSS4}} for untwisted types]\label{4.1.12}
Let $\eta \in \QLS(\lambda)$ {\rm(}resp., $\in \QLS(\lambda)^\C${\rm)}, and $i \in I_\aff$.
Then all local minima of $H_i^\eta (t)$, $t \in [0,1]$, are elements in $\mathbb{Z}$ {\rm(}resp., $\chi_i \mathbb{Z}${\rm)}.
\end{prop}

\begin{proof}
Assume that the function $H_i^\eta (t)$ attains a local minimum at $t' \in [0,1]$;
we may assume $t' = \sigma_u$ for some $0 \leq u \leq s$.
If $u =0$, then $H_i^\eta (0) = 0 \in 2\mathbb{Z}\subset \mathbb{Z}$;
if $u =s$, then $H_i^\eta (1) \in \chi_i \mathbb{Z}$
since $\eta(1) \in Q $.
Otherwise, we have either $\langle w_u \lambda , \overline{\alpha}_i^\lor \rangle \leq 0$ and $\langle w_{u+1} \lambda , \overline{\alpha}_i^\lor \rangle > 0$,
or $\langle w_u \lambda , \overline{\alpha}_i^\lor \rangle < 0$ and $\langle w_{u+1} \lambda , \overline{\alpha}_i^\lor \rangle \geq 0$.
Therefore, it follows from Lemma \ref{4.1.10} or \ref{4.1.11} that $H_i^\eta (\sigma_u) \in \chi_i \mathbb{Z}$.
\end{proof}

\begin{lem}[\normalfont{see \cite[Lemma 4.1.13]{LNSSS4}} for untwisted types]\label{4.1.13}
\ \\
\vspace{0mm}
Let
$\eta = (w_1 , \ldots , w_s ; \sigma_0, \ldots, \sigma_s) \in \QLS(\lambda)$ {\rm(}resp., $\in \QLS(\lambda)^\C${\rm)}.
Let
$i \in I_\aff$ and $1 \leq u \leq s-1$ be such that $\langle w_{u+1}\lambda, \overline{\alpha}_i^\lor \rangle >0$ and $H_i^\eta (\sigma_u) \notin \mathbb{Z}$
{\rm(}resp., $\notin \chi_i\mathbb{Z}${\rm)}.
Let
\begin{equation}\label{eq_4.1.4}
w_u = x_0 \xleftarrow{\gamma_1^\lor}  \cdots \xleftarrow{\gamma_r^\lor} x_r = w_{u+1}
\end{equation}
be a directed path from $w_{u+1}$ to $w_u$ in $\QBG_{\sigma_u \lambda}^S$ {\rm(}resp., $(\QBG_{\sigma_u \lambda}^\C)^S${\rm)}.
Then, 
$\langle x_k \lambda , \overline{\alpha}_i^\lor \rangle >0$ for all $0 \leq k \leq r$,
and there exists a directed path from $\lfloor r_i w_{u+1} \rfloor$ to  $\lfloor r_i w_u \rfloor$ in $\QBG_{\sigma_u \lambda}^S$ {\rm(}resp., $(\QBG_{\sigma_u \lambda}^\C)^S${\rm)} of the form:
\begin{equation}\label{eq_4.1.5}
\lfloor w_u \rfloor= \lfloor x_0 \rfloor \xleftarrow{z_1 \gamma_1^\lor}  \cdots \xleftarrow{z_r \gamma_r^\lor}\lfloor x_r \rfloor =\lfloor w_{u+1} \rfloor.
\end{equation}
Here,
if $ i \in I$,
then we define $z_k = e$ for all $1 \leq k \leq r$;
if $ i = 0$,
then we define $z_k \in W_S$ by 
$s_\theta x_k = \lfloor s_\theta x_k \rfloor z_k$
 for all $1 \leq k \leq r$.
Moreover, if the directed path \eqref{eq_4.1.4} is a shortest one from $w_{u+1}$ to $w_u$, i.e., $\ell (w_{u+1}, w_u) =r$,
then the directed path \eqref{eq_4.1.5} is a shortest one from $\lfloor r_i w_{u+1} \rfloor$ to  $\lfloor r_i w_u \rfloor$,
i.e., $\ell(\lfloor r_i w_{u+1} \rfloor, \lfloor r_i w_u \rfloor) = r$.
\end{lem}

\begin{proof}
It follows from Lemma \ref{4.1.10} that if $H_i^\eta (\sigma_u) \notin \chi_i \mathbb{Z}$,
then $\langle x_k \lambda, \overline{\alpha}^\lor_i \rangle >0$ for all $0 \leq k \leq r$
(in particular, $\langle w_u \lambda, \overline{\alpha}^\lor_i \rangle >0$).
Assume that $i \in I$ (resp., $i=0$), and suppose, for a contradiction,
that $x_k \gamma_k = \alpha_i$ (resp., $= \pm \theta$) for some $1 \leq k \leq r$.
Then, $x_{k-1}\lambda = x_k s_{\gamma_k}\lambda = s_{x_k \gamma_k} x_k \lambda = r_i x_k \lambda$,
and hence $\langle x_{k-1}\lambda, \overline{\alpha}_i^\lor \rangle = \langle r_i x_k \lambda, \overline{\alpha}_i^\lor \rangle = - \langle x_k \lambda, \overline{\alpha}_i^\lor \rangle$,
which contradicts the fact that $\langle x_{k-1}\lambda, \overline{\alpha}_i^\lor \rangle >0 $ and $\langle x_k \lambda, \overline{\alpha}_i^\lor \rangle >0$.
Thus, we conclude that $x_k \gamma_k \neq \pm \alpha_i$ (resp., $\neq \pm \theta$)
for all $1 \leq k \leq r$.
Therefore, we deduce from Lemma \ref{diamond} (1), (2) (resp., (3), (4))
that there exists a directed path of the form (\ref{eq_4.1.5}) from $\lfloor r_i w_{u+1}\rfloor$ to $\lfloor r_i w_u\rfloor$.
Because the directed path (\ref{eq_4.1.4}) lies in $\QBG^\C_{\sigma_u \lambda}$,
and $\sigma_u \langle \lambda, z\gamma_k^\lor \rangle=\sigma_u \langle \lambda, \gamma_k^\lor \rangle$,
the directed path (\ref{eq_4.1.5}) is a path from $\lfloor r_i w_{u+1} \rfloor$ to $\lfloor r_i w_u \rfloor$ in  $(\QBG_{\sigma_u \lambda}^C)^S$.
The proof of being the shortest is similar to that of \cite[Lemma 4.1.13]{LNSSS4}.
\end{proof}

\subsection{Proof of Theorems \ref{stable} and \ref{realization_theorem}}
We set $\lambda = \iota^*(\lambda^\dag)$. We recall that $\lambda$ is a dominant weight in $Q$.
\begin{proof}[Proof of Theorem $\ref{stable}$]
The outline of the proof is similar to that of \cite[Proposition 4.2.1]{LNSSS4}.
It suffices to show that the set $\QLS^\C(\lambda)$ is stable under the action of the root operators $\widetilde{f}\pdag_i$, $i \in I_\aff$.
Let $\eta = (w_1, \ldots , w_s; \sigma_0,\ldots ,  \sigma_s)\in \QLS^\C(\lambda)$,
and assume that $\widetilde{f}\pdag_i \eta \neq \zero$.
Then there exists $0 \leq u < s$ such that $\sigma_u = t_2$.
Let $u \leq m < s$ be such that $\sigma_m < t_3^\chi \leq \sigma_{m+1}$;
notice that $H_i^\eta(t)$ is strictly increasing on $[t_2, t_3^\chi]$,
which implies that $\langle w_p \lambda, \alpha_i^\lor \rangle>0$ for all $u+1 \leq p \leq m+1$.
We need to consider the following four cases:

Case 1: $w_u \neq \lfloor s_j w_{u+1}\rfloor$ or $u = 0$, and $\sigma_m < t_3^\chi < \sigma_{m+1}$,

Case 2: $w_u \neq \lfloor s_j w_{u+1}\rfloor$ or $u = 0$, and $ t_3^\chi = \sigma_{m+1}$,

Case 3: $w_u = \lfloor s_j w_{u+1}\rfloor$, and $\sigma_m < t_3^\chi < \sigma_{m+1}$,

Case 4: $w_u = \lfloor s_j w_{u+1}\rfloor$, and $ t_3^\chi = \sigma_{m+1}$.

\noindent
In what follows, we assume Case 2; the proofs of other cases are easier.
Then it follows from the definition of the root operator $\widetilde{f}\pdag_i$ that
\begin{align*}
&\widetilde{f}\pdag_i \eta = (w_1 , \ldots , w_u , \lfloor s_i w_{u+1}\rfloor,\ldots , \lfloor s_i w_{m+1}\rfloor, w_{m+2}, \ldots, w_s; \\
& \hspace{40mm}\sigma_0, \ldots , \sigma_u, \ldots , \sigma_m, t_3^\chi = \sigma_{m+1}, \sigma_{m+2}, \ldots, \sigma_s )
\end{align*}

It suffices to show that

(i)
there exists a directed path from $\lfloor s_i w_{u+1}\rfloor$ to $w_u$ in $(\QBG_{\sigma_u \lambda}^\C)^S$ (when $u>0$);

(ii)
there exists a directed path from $\lfloor s_i w_{p+1}\rfloor$ to $\lfloor s_i w_{p}\rfloor$ in $(\QBG_{\sigma_{p} \lambda}^\C)^S$ for each $u+1 \leq p \leq m$;

(iii)
there exists a directed path from $ w_{m+2}$ to $\lfloor s_i w_{m+1}\rfloor$ in $(\QBG_{\sigma_{m+1} \lambda}^\C)^S$.

\noindent
Also,
we will show that if $\eta \in \widetilde{\QLS}^\C(\lambda)$,
then the directed paths in (i)-(iii) above
can be chosen from the shortest ones, which implies that $f_i \eta \in \widetilde{\QLS}^\C(\lambda)$.

(i)
We deduce that from the definition of 
$t_2 = \sigma_u$ that
$\langle w_u \lambda , \overline{\alpha}_i^\lor \rangle \leq 0$ and $\langle w_{u+1} \lambda , \overline{\alpha}_i^\lor \rangle > 0$.
Since $\eta \in \QLS^\C(\lambda)$,
there exists a directed path from $w_{u+1}$ to $w_u$ in
$(\QBG^\C_{\sigma_u \lambda})^S$.
Hence it follows from Lemma \ref{4.1.9} (1), (4)
that there exists a directed path from $\lfloor r_i w_{u+1} \rfloor$ to $w_u$ in $(\QBG^\C_{\sigma_u \lambda})^S$.
Moreover, we see from the definition of $\widetilde{\QLS}^\C (\lambda)$ and
Lemma \ref{4.1.9} (2) that
if $\eta \in \widetilde{\QLS}^\C (\lambda)$,
then there exists a directed path from $\lfloor r_i w_{u+1} \rfloor$ to $w_u$ in $(\QBG^\C_{\sigma_u \lambda})^S$ whose length is equal to $\ell(\lfloor r_i w_{u+1} \rfloor, w_u)$.

(ii)
Recall that $H_i^\eta (t)$ is strictly increasing on $[t_2, t_3^\chi]$,
and that $H_i^\eta (t_0) = m_i^\eta \in \chi_i \mathbb{Z}$ and $H_i^\eta (t_1) = m_i^\eta +\chi_i$.
Hence it follows that $H^\eta_i (\sigma_p)  \notin
\chi_i \mathbb{Z}$
for any $u+1 \leq p \leq m$.
Therefore, we deduce from Lemma \ref{4.1.13} that
there exists a directed path from $\lfloor r_i w_{p+1}\rfloor$ to $\lfloor r_i w_{p} \rfloor$ in $(\QBG_{\sigma_p}^\C)^S$ for each $u+1 \leq p \leq m$.
Moreover, we see from the definition of $\widetilde{\QLS}^\C (\lambda)$ and Lemma \ref{4.1.13} that 
if $\eta \in \widetilde{\QLS}^\C(\lambda)$,
then for each $u+1 \leq p \leq m$, there exists a 
directed path $\lfloor r_i w_{p+1}\rfloor$ to $\lfloor r_i w_{p} \rfloor$ in $(\QBG_{\sigma_p}^\C)^S$ for each $u+1 \leq p \leq m$ whose length is equal to $\ell(\lfloor r_i w_{p+1} \rfloor , \lfloor r_i w_{p} \rfloor) $.

(iii)
Since $\langle w_{m+1} \lambda , \overline{\alpha}_i^\lor \rangle >0$,
it follows from  Lemma \ref{leftaction} (1), (2) that
if $i \in I$ (resp., $i=0$), then
$\lfloor r_i w_{m+1} \rfloor \xleftarrow{\gamma^\lor}
w_{m+1}$ is a Bruhat (resp., quantum) edge;
here
if $i \in I$ (resp., $i =0$), then $\gamma = w_{m+1}\inv \alpha_i$ (resp., $\gamma= - w_{m+1}\inv \theta$).
Note that $\langle \lambda, \gamma^\lor \rangle =
\langle x_{m+1}\lambda, \overline{\alpha}_i^\lor \rangle$.

\begin{claim}
$\lfloor r_i w_{m+1} \rfloor \xleftarrow{\gamma^\lor}
w_{m+1}$ is an edge of $(\QBG_{t_3^\chi \lambda }^\C)^S$;
that is,
$t_3^\chi \langle w_{m+1}\lambda , \overline{\alpha}_i^\lor \rangle \in \chi_i \mathbb{Z}$.
\end{claim}

$Proof \ of \ Claim.$
Since
\begin{equation*}
\eta (t_3^\chi)
=
\sum_{k=1}^{m+1} (\sigma_k-\sigma_{k-1})w_k \lambda
=
t_3^\chi w_{m+1}\lambda +
 \sum_{k=1}^{m} \sigma_k(w_k \lambda - w_{k+1} \lambda) ,
\end{equation*}
we have
\begin{equation*}
H_i^\eta (t_3^\chi)
=
t_3^\chi \langle w_{m+1}\lambda, \overline{\alpha}_i^\lor
\rangle 
+
 \sum_{k=1}^{m} \langle \sigma_k(w_k \lambda - w_{k+1} \lambda) , \overline{\alpha}_i^\lor
\rangle.
\end{equation*}
Here, $\sigma_k(w_k \lambda - w_{k+1} \lambda) \in Q$ by Lemma \ref{4.1.6},
and hence $\langle \sigma_k(w_k \lambda - w_{k+1} \lambda) , \overline{\alpha}_i^\lor
\rangle \in \chi_i \mathbb{Z}$.
Also, it follows from Lemma \ref{4.1.12} that $H^\eta_i (t_3^\chi) = m_i^\eta + \chi_i \in \chi_i \mathbb{Z}$.
Therefore, we deduce that 
$t_3^\chi \langle w_{m+1}\lambda, \overline{\alpha}_i^\lor
\rangle \in \chi_i \mathbb{Z}$.
\bqed

By the definition of $\QLS^\C (\lambda)$,
there exists  a directed path from $w_{m+2}$ to $w_{m+1}$ in $(\QBG^\C_{\sigma_{m+1}\lambda})^S$.
Concatenating this directed path and the edge $\lfloor r_i w_{m+1} \rfloor \xleftarrow{\gamma^\lor}
w_{m+1}$,
we obtain a directed path from $w_{m+2}$ to $\lfloor r_i w_{m+1} \rfloor$ in $(\QBG^\C_{\sigma_{m+1}\lambda})^S$.
Thus, we have proved that $f_i \eta \in \QLS^\C(\lambda)$.

Assume that $\eta \in \widetilde{\QLS}^\C (\lambda)$,
and set $r \eqdef \ell(w_{m+2}, w_{m+1})$.
We see from the argument above that there exists a directed path from $w_{m+2}$
to $\lfloor r_i w_{m+1} \rfloor$
in $(\QBG_{\sigma_{u+1}\lambda}^\C)^S$
whose length is equal to $r+1$.
Suppose, for a contradiction, that there exists a directed path to $\lfloor r_i w_{m+1} \rfloor$
in $\QBG^S$
whose length $\l$ is less than $r+1$.
Since
$\langle r_i w_{m+1} \lambda ,  \overline{\alpha}_i^\lor \rangle< 0$
and
$\langle w_{m+2} \lambda, \overline{\alpha}_i^\lor \rangle \geq 0$,
we deduce from Lemma \ref{4.1.8} that
there exists a directed path from $w_{m+2}$
to
$\lfloor r_i \lfloor r_i w_{m+1}\rfloor \rfloor =\lfloor w_{m+1} \rfloor =w_{m+1}$
in $\QBG^S$
whose length is equal to $l-1<r= \ell(w_{m+2}, w_{m+1})$, a contradiction.
Thus, we have proved that if $\eta \in \widetilde{\QLS}^\C (\lambda)$,
then $f_i \eta  \in \widetilde{\QLS}^\C (\lambda)$.
\end{proof}

\begin{proof}[Proof of Theorem $\ref{realization_theorem}$]
Since $\eta_\lambda \in \widetilde{\QLS}^\C(\lambda) \subset \QLS^\C(\lambda)$,
it follows from the definition of $\widetilde{\mathbb{B}}(\lambda)_\cl$ and Theorem \ref{stable} that $\widetilde{\mathbb{B}}(\lambda)_\cl  \subset \widetilde{\QLS}^\C(\lambda) \subset \QLS^\C(\lambda)$.
%Recall that the set $\widetilde{\mathbb{B}}(\lambda)_\cl$ is a simple crystal generated by $\eta_\lambda$.
Hence it suffices to show that $\QLS^\C(\lambda) \subset \widetilde{\mathbb{B}}(\lambda)_\cl$.

\begin{claim}
$(\widetilde{f}\pdag_i)^{\max} \pi = f_i^{\max} \pi$ for each $\pi \in  \QLS^\C(\lambda)$ and $i \in I_\aff$.
\end{claim}
$Proof \ of \ Claim.$
It follows from Proposition \ref{4.1.12} and (\ref{epsilon_phi}) that 
$\varphi_i (\pi) \in \chi_i \mathbb{Z}$.
Therefore, by (\ref{epsilon_phi_C}),  $\widetilde{\varphi}\pdag_i (\pi) = \frac{\varphi_i(\pi)}{\chi_i} \in \mathbb{Z}$.
Hence by Remark \ref{root_operator_identification},
$(\widetilde{f}\pdag_i)^{\widetilde{\varphi}\pdag_i (\pi)} \pi = f_i^{\varphi_i (\pi)} \pi$, which implies
$(\widetilde{f}\pdag_i)^{\max} \pi = f_i^{\max} \pi$.
\bqed

Since $\QLS^\C(\lambda) \subset \QLS(\lambda)$, 
for each element $\eta \in \QLS^\C(\lambda)$, 
there exist  $i_1, \ldots, i_k\in I_\aff$ such that 
\begin{equation*}
f_{i_1}^{\max} \dots f_{i_k}^{\max}\eta =  \eta_\lambda.
\end{equation*}
by Remark \ref{simple_crystal_dual_untwisted}.
It follows from the claim above that
\begin{equation*}
(\widetilde{f}\pdag_{i_1})^{\max} \dots (\widetilde{f}\pdag_{i_k})^{\max}\eta =  \eta_\lambda,
\end{equation*}
which implies that $\eta \in \widetilde{\mathbb{B}}(\lambda)_\cl$.

\end{proof}

\section{The graded character of $\QLS(\lambda)$}
In this section, we fix $(A, A_\aff)= (B_n, \D)$.
For $x \in \mathfrak{h}^*_\mathbb{R} \oplus \mathbb{R}\delta$, 
we define $\overline{x} \in \mathfrak{h}^*_\mathbb{R}$ and $\dg(x) \in \mathbb{R}$ by
%%%%%%%%%%%%
%%%%%%%%%%%%
%equ:4-1
%%%%%%%%%%%%
%%%%%%%%%%%%
\begin{equation}\label{eq:4-1}
x= \overline{x} + \dg(x)\delta.
\end{equation}

For $x \in Q$,
let $t(x)$ denote the linear transformation on $\mathfrak{h}^*_\mathbb{R} \oplus \mathbb{R}\delta$: 
$ t(x) (y+ r \delta ) \eqdef y +(r- 2(x,y))\delta$ for $y \in \mathfrak{h}^*_\mathbb{R}$, $r \in \mathbb{R}$,
where 
$\delta$ denotes the null root of $\mathfrak{g}(\D)$ and
$(\cdot , \cdot) : \mathfrak{h}_\mathbb{R}^*   \times \mathfrak{h}_\mathbb{R}^* \rightarrow \mathbb{R}$ the bilinear form defined in Remark \ref{BC'}.
The affine Weyl group of $\mathfrak{g}(\D)$
are defined by
$W_\aff \eqdef t(Q) \rtimes W $.
Also, we define $s_0 : \mathfrak{h}^* \rightarrow \mathfrak{h}^*$ by $y + r \delta \mapsto s_\theta x - (r - 2(x, \theta))\delta
=
s_\theta x - (r- \pair{x}{\theta^\lor})\delta
$.
Then $W_\aff = \langle s_i \ | \ i \in I_\aff \rangle$;
note that $s_0 = t(\theta)s_\theta$.

%In this section, we fix $(A, A_\aff)= (B_n, \D)$.
%For $x \in Q$,
%let $t(x)$ denote the linear transformation on $P_\aff^0$: 
%$ t(x) (y+ z \delta ) = y +(z- 2(x,y))\delta$ for $y \in \oplus_{i \in I}\mathbb{Z}(\Lambda_i - \langle \Lambda_i,c \rangle_\aff \Lambda_0) $, $z\in \mathbb{Z}$,
%where $(\cdot , \cdot) : \mathfrak{h}_\mathbb{R}^* \times \mathfrak{h}_\mathbb{R}^* \rightarrow \mathbb{R}$ denotes the bilinear form defined in Remark \ref{BC'}.
%The affine Weyl group of $\mathfrak{g}(\D)$
%are defined by
%$W_\aff \eqdef t(Q) \rtimes W $.
%Also, we define \textcolor{red}{$s_0 : \mathfrak{h}^* \rightarrow \mathfrak{h}^*$ by $x \mapsto s_\theta x - 2(x, \theta)\delta
%=
%s_\theta x - \pair{x}{\theta}\delta
%$}.
%Then $W_\aff = \langle s_i \ | \ i \in I_\aff \rangle$;
%note that $s_0 = t(\theta)s_\theta$.
%The affine Weyl group $W_\aff$ acts on
%$P \oplus \mathbb{Z}\delta$ induced by $\xi : P \rightarrow P_\aff^0$\footnote{$\overline{P_\aff^0}$ が良いかもしれない (extremal weight module について書くかどうかによる)}
%(as affine transformations):
%for $v\in W$, $t(x) \in t(Q)$,
%\begin{equation}
%v t(x)( \overline{\beta}+r\delta )=v\overline{\beta}+(r- 2 ( x, \overline{\beta} ) ) \delta,
%\ \ 
%\overline{\beta} \in  Q , r \in \mathbb{Z}.
%\end{equation}

An element $u \in W_\aff $ can be written as 
%%%%%%%%%%%%
%%%%%%%%%%%%
%equ:decomposition
%%%%%%%%%%%%
%%%%%%%%%%%%
\begin{equation}\label{equ:decomposition}
u=t( \wt (u) ) \dr (u),
\end{equation}
where $\wt (u) \in Q$ and $ \dr(u) \in W$,
according to the decomposition
$W_\aff = t(Q) \rtimes W $.
For $w \in W_\aff$,
we denote the length of $w$ by 
$\ell (w) $,
which equals
$\# \left( \Delta^+_\aff
\cap w^{-1}\Delta^-_\aff \right)$.

\subsection{Nonsymmetric and symmetric Macdonald polynomials of type $\C$}
In this subsection, we introduce the notion of nonsymmetric and symmetric Macdonald polynomials of type $\C$;
see \cite[\S 3.6]{OS} for more detail.
We fix $(A, A_\aff) = (B_n , \D)$.
Set $\widetilde{\Delta}_\aff \eqdef W_\aff (\{ \alpha_i \ | \ i=0,\ldots, n \} \cup \{ 2\alpha_0, 2\alpha_n \})\subset \mathfrak{h}^*_\mathbb{R} \oplus \mathbb{R}\delta$.
Then, $\widetilde{\Delta}_\aff$ is decomposed into five $W_\aff$-orbits: 
\begin{equation*}
\widetilde{\Delta}_\aff = 
\underbrace{W_\aff \{ \alpha_0 \}}_{\eqdef O_1} \sqcup  \underbrace{W_\aff \{ 2\alpha_0 \}}_{\eqdef O_2=2O_1} \sqcup 
\underbrace{W_\aff \{ \alpha_n \}}_{\eqdef O_3} \underbrace{\sqcup W_\aff \{ 2\alpha_n \}}_{\eqdef O_4=2O_2} \sqcup 
\underbrace{W_\aff \{ \alpha_i \ | \ i = 1, \ldots , n-1 \}}_{\eqdef O_5}.
\end{equation*}

Let $t_i$ be a indeterminate with respect to $O_i$, $i=1,\ldots , 5$.
For $\mu \in Q$,
let $E_\mu (q, t_1 , t_2, t_3, t_4, t_5)$ denote the nonsymmetric Macdonald-Koornwinder polynomial.
% which is of the from
%$E_\mu (q, t_1 , t_2, t_3, t_4, t_5) = e^\mu + \sum_{\nu < \mu} f_\nu e^\nu$, $f_\nu \in \mathbb{Q}(q, t_1 , t_2, t_3, t_4, t_5)$;
%here 
%the partial order on $Q$ is the one in \cite[(2.7.5)]{M}.
Also, for a dominant weight $\lambda \in Q$,
let $P_\lambda (q, t_1 , t_2, t_3, t_4, t_5)$ denote the symmetric Macdonald-Koornwinder polynomial.
% which is of the from
%$P_\lambda (q, t_1 , t_2, t_3, t_4, t_5) = m^\lambda + \sum_{\nu < \lambda} f_\nu m^\nu$, $f_\nu \in \mathbb{Q}(q, t_1 , t_2, t_3, t_4, t_5)$;
%here we set
%$m^\nu \eqdef \sum_{\mu \in W\nu} e^\mu$ for a dominant weight $\nu \in Q$, and  
%the partial order on the set of all dominant weights is defined by 
%\begin{align*}
%\nu < \lambda \ \ardef  \ \lambda -\nu \in \sum_{i \in I} \mathbb{Z}_{\geq 0}\alpha_i.
%\end{align*}
%
Then we define the nonsymmetric and symmetric Macdonald polynomials of type $\C$ by
$E_\mu^\C (q, t) \eqdef E_\mu (q, t, t, t, 1, t) $ and
$P_\lambda^\C (q, t) \eqdef P_\lambda (q, t, t, t, 1, t) $, respectively.
We set  
$E_\mu^\C (q, 0) \eqdef \lim_{t \rightarrow 0} E_\mu^\C (q, t)$ and 
$P_\lambda^\C (q, 0) \eqdef \lim_{t \rightarrow 0} P_\lambda^\C (q, t)$, which are well-defined.

\begin{rem}\label{Lemma7.7}
As in \cite[Lemma 7.7]{LNSSS2},
for a dominant weight $\lambda \in Q$, we have
\begin{equation*}
P_\lambda^\C (q, 0) =
E_{\lon \lambda}^\C (q, 0),
\end{equation*}
where $\lon$ denotes the longest element of $W$.
\end{rem}

Let $\lambda \in Q$ be a dominant weight.
For $\eta = (w_1, \ldots, w_s; \sigma_0, \ldots, \sigma_s) \in \QLS^\C (\lambda)$, we set
\begin{align*}
\Dg (\eta) = \sum_{i=0}^{s-1} (1 - \sigma_i) \wt_\lambda (w_i \Rightarrow w_{i+1}),
\end{align*}
where $\wt_\lambda (w_i \Rightarrow w_{i+1})$ is the $\lambda$-weight of the path from $w_{i+1}$ to $w_i$ in $\QBG$ (or $\QBG^S$) as defined in \S 2.3.
We define
\begin{align*}
\gch (\QLS^\C(\lambda)) \eqdef \sum_{\eta \in \QLS^\C (\lambda)} q^\Dg (\eta) e^{\wt(\eta)} .
\end{align*}
We will show the following theorem in \S 4.4.
\begin{thm}\label{theorem_graded_character}
Let $\lambda \in Q$ be a dominant weight.
Then,
$P^\C_\lambda (q,0) = \gch  (\QLS^\C(\lambda)) $.
\end{thm}

\subsection{Orr-Shimozono formula}

For $\mu \in Q$, 
we denote the shortest element in the coset	$t(\mu)W$ by $m_{\mu} \in W_\aff$.
In the following, we fix $\mu \in Q$,
and take a reduced expression $m_{\mu} =  s_{\ell_{1}}\cdots s_{\ell_{L}}$,
where $ \ell_1 , \ldots , \ell_L \in  I_{\aff}$.

For each $J = \{ j_{1} < j_{2} < j_{3} < \cdots < j_{r} \} \subset \{1,\ldots,L\}$,
we define an alcove path 
\begin{equation*}
p^{\OS}_{J} =
			\left( m_{\mu} = z^{\OS}_0, z^{\OS}_{1} , \ldots , z^{\OS}_{r} ; \beta^{\OS}_{j_1} , \ldots , \beta^{\OS}_{j_r} \right)
\end{equation*}
 as follows: 
we set
$\beta^{\OS}_{k} \eqdef s_{\ell_{L}}\cdots s_{\ell_{k+1}} \alpha_{\ell_{k}} \in \Delta^+_\aff$ 
for $1 \leq k \leq L$, and set $z^{\OS}_{0} \eqdef m_{\mu}$ and $z^\OS_k \eqdef z^\OS_{k-1} s_{\beta^{\OS}_{j_{k}}}$, $1 \leq k \leq s$.
%	\begin{align*}
%		z^{\OS}_{0} &=m_{\mu} ,\\
%		z^{\OS}_{1} &=m_{\mu}s_{\beta^{\OS}_{j_{1}}},\\
%		z^{\OS}_{2} &=m_{\mu}s_{\beta^{\OS}_{j_{1}}}s_{\beta^{\OS}_{j_2}},\\
%				&\vdots& \\
%		z^{\OS}_{r} &=m_{\mu}s_{\beta^{\OS}_{j_{1}}}\cdots s_{\beta^{\OS}_{j_r}}.
%	\end{align*}
Also, following \cite[\S 3.3]{OS}, 
we set
	$\B ({\id};m_{\mu})
	\eqdef
	\left\{ p^{\OS}_{J} \ \left| \ J \subset \{ 1,\ldots ,L \}  \right. \right\}$
and
	$\ed (p^{\OS}_{J}) \eqdef z^{\OS}_{r}\in W_\aff$.
Then we define
	$\QB^\C (\id; m_{\mu})$
to be the following subset of $\B ({\id};m_{\mu})$:
	\begin{equation*}
		\left\{
			 p^{\OS}_{J} \in \B ({\id} ; m_{\mu} ) \ 
			\left|  \ 
			\begin{array}{l}
				\dr(z^{\OS}_{i}) 
				\xrightarrow{-\left(\overline{ {\beta}^{\OS}_{j_{i+1}}  } \right)^{\lor}}
				\dr (z^{\OS}_{i+1}) \ 
				\text{ is an edge of \QBG}, \ 0\leq  i \leq r-1, \\
				\mbox{if this edge is a Bruhat edge, then }\beta_{j_i}^\OS \in \Delta \oplus 2 \mathbb{Z}\delta
			\end{array}	
			\right. 
		\right\}.
	\end{equation*}

	\begin{rem}[\normalfont{\cite[(2.4.7)]{M}}]
		If $j \in J$,
		then $-\overline{ {\beta}^{\OS}_{j}  } \in {\Delta}^{+}$.
	\end{rem}

	For $p^{\OS}_{J} \in \QB^\C({\id} ; m_{\mu})$, we define ${\qwt}(p^{\OS}_{J})$ as follows.
Let 
$J^0$ denote the set of all indices $j_i$ such that $\ell_{j_i}=0$, and
$J^- \subset J \setminus J^0$ the set of all indices $j_i \in J\setminus J^0$  such that 
	$\dr(z_{i-1}^{\OS}) 
	\xrightarrow{-\left(\overline{ {\beta}^{\OS}_{j_i}  } \right)^{\lor}} 
	\dr (z_{i}^{\OS})$ is a quantum edge;
we remark that  for $j_i \in J^0$,  $\dr(z_{i-1}^{\OS}) 
	\xrightarrow{-\left(\overline{ {\beta}^{\OS}_{j_i}  } \right)^{\lor}} 
	\dr (z_{i}^{\OS})$ is a quantum edge.
Then we set
	\begin{equation*}
		\qwt(p^{\OS}_{J}) 
		\eqdef
		\sum_{j \in J^0 \sqcup J^-} \beta^{\OS}_{j}.
	\end{equation*}

We know the following formula for the specialization of
nonsymmetric Macdonald polynomials at $t=0$ of type $\C$.

\begin{prop}[{\cite[Proposition 5.5]{OS}}]\label{os}
Let $\mu \in Q= P\pdag \subset P$. Then,
	\begin{equation*}
		E_{\mu}^\C(q,0)=
		\sum_{p^{\OS}_{J} \in \QB^\C ({\id} ; m_{\mu}) } 
		q^{{\dg}({\qwt}(p^{\OS}_{J}))}e^{\wt(\ed (p^{\OS}_{J}))}.
	\end{equation*}
\end{prop}
By Remark \ref{Lemma7.7}, we obtain the Orr-Shimozono formula for 
the specialization of symmetric Macdonald polynomials at $t=0$ of type $\C$.
\begin{cor}\label{oss}
Let $\lambda \in Q$ be a dominant weight. Then,
	\begin{equation*}
		P_\lambda^\C(q,0)=
		\sum_{p^{\OS}_{J} \in \QB^\C ({\id} ; t(\lon \lambda)) } 
		q^{{\dg}({\qwt}(p^{\OS}_{J}))}e^{\wt(\ed (p^{\OS}_{J}))}.
	\end{equation*}
\end{cor}
Here we remark that 
\begin{equation}
m_{\lon \lambda} = t(\lon \lambda),
\end{equation}
 and 
\begin{equation}\label{equ:direction}
\dr(t(\lon \lambda)) = e
\end{equation}
(see \cite[\S 2.4]{M}).

\subsection{Reduced expression for $t(\lon \lambda)$ and total order on $\Delta^+_{\aff} \cap t(\lon \lambda)\inv \Delta^-_{\aff} $}

Let $\lambda \in Q \subset P $ be a dominant weight,
and set $\lambda_- = \lon \lambda$ and $S =\{ i \in I \ | \ \langle \lambda , \alpha^\lor_i \rangle =0 \}$.
For $\mu \in W\lambda$, we denote by $v(\mu)\in W^S$ the minimal-coset representative for the coset $\{ w \in W  \ | \ w \lambda = \mu \}$.
Since
$\lon$ is the longest element of $W$,
we have
$\lon = v(\lambda_-) \lons$
and
$\ell(\lon) = \ell(v(\lambda_-)) + \ell(\lons)$.

We fix reduced expressions
\begin{align*}
v(\lambda_-)&=  s_{i_1}\cdots s_{i_M} , \\ 
\lons &= s_{i_{M+1}}\cdots s_{i_N}
\end{align*}
for 
$v(\lambda_-)$, and $\lons$, respectively.
Then 
$ \lon =  s_{i_1}\cdots s_{i_N} $
is a reduced expression for $\lon$.
We set $\beta_j \eqdef s_{i_N} \cdots s_{i_{j+1}}\alpha_{i_j}$, $1 \leq j \leq N$.
Then we have
$\Delta^+ \setminus \Delta^+_S = \{ \beta_1, \ldots, \beta_M \}$, and
$\Delta^+_S  = \{ \beta_{M+1}, \ldots, \beta_N \}$.
We fix a total order on $(\Delta^+)^\lor$ by
\begin{align}\label{weak_reflection_order}
\underbrace{ \beta_1^\lor \succ \beta_2^\lor \succ \cdots \succ \beta_M^\lor }_{\in  (\Delta^+ \setminus \Delta_S^+)^\lor} \succ 
\underbrace{ \beta_{M+1}^\lor \succ \beta_{N}^\lor}_{\in (\Delta_S^+)^\lor} .
\end{align}

\begin{rem}
This total order is a (weak) reflection order on $(\Delta^+)^\lor$; that is,
if $\alpha ,\beta ,\gamma \in \Delta^+$ with $\gamma = \alpha + \beta$,
then $\alpha^\lor \prec \gamma^\lor \prec \beta^\lor$ or $\beta^\lor \prec \gamma^\lor \prec \alpha^\lor$.
%(see \cite{papi} for more detail)
\end{rem}

We define an injective map $\Phi$ by:
\begin{align*}
\Phi : \Delta^+_{\aff} \cap \tra\inv \Delta^-_{\aff} 
&\rightarrow
\mathbb{Q}_{\geq 0} \times (\Delta^+\setminus \Delta^+_S )^\lor ,\\
\beta = \overline{\beta} + \dg(\beta) \delta
&\mapsto 
\left(\frac{c_{\overline{\beta}} \langle {\lambda_-} ,  \overline{\beta}^\lor \rangle -  \dg(\beta)}{c_{\overline{\beta}} \langle {\lambda_-} ,  \overline{\beta}^\lor \rangle }  
,  \lon \overline{\beta}^\lor
\right);
\end{align*}
note that $\langle {\lambda_-} ,  \overline{\beta}^\lor \rangle >0$, $c_{\overline{\beta}} \langle {\lambda_-} ,  \overline{\beta}^\lor \rangle -  \dg(\beta) \geq 0$,
and 
$\lon \overline{\beta} \in \Delta^+\setminus \Delta^+_S $
since $\langle {\lambda_-} ,  \overline{\beta}^\lor \rangle = \langle {\lambda} ,  \lon \overline{\beta}^\lor \rangle >0$,
since  we know from \cite[(2.4.7) (i)]{M}
that
\begin{equation}\label{B}
\Delta^+_{\aff} \cap \tra\inv \Delta^-_{\aff} =
\{ \alpha + a c_\alpha \delta 
\ | \
\alpha \in \Delta^-,
a \in \mathbb{Z}, \mbox{ and }
0 < a \leq  \langle \lambda_- , \alpha^{\lor}
\rangle 
\}.
\end{equation}
We now consider the lexicographic order $<$ on $\mathbb{Q}_{\geq 0} \times (\Delta^+\setminus \Delta^+_S)^\lor$
induced by the usual total order on $\mathbb{Q}_{\geq 0}$
and
the inverse of
the restriction to $(\Delta^+\setminus \Delta^+_S)^\lor$
of the total order $\prec$ on $(\Delta^+)^\lor$ defined above;
that is, for $(a, \alpha^\lor), (b, \beta^\lor) \in \mathbb{Q}_{\geq 0} \times (\Delta^+\setminus \Delta^+_S)^\lor$,
\begin{equation*}
(a, \alpha^\lor)<(b, \beta^\lor) \mbox{ if and only if } 
a<b, \mbox{ or }
a=b \mbox{ and } \alpha^\lor \succ \beta^\lor.
\end{equation*}
Then we denote by $\prec'$ 
the total order on $\Delta^+_{\aff} \cap \tra\inv \Delta^-_{\aff}$
induced by  the lexicographic order on $\mathbb{Q}_{\geq 0} \times (\Delta^+\setminus \Delta^+_S)^\lor$
through the map $\Phi$,
and write $\Delta^+_{\aff} \cap \tra\inv \Delta^-_{\aff} $
as 
$\left\{\gamma_1 \prec' \cdots \prec' \gamma_L \right\}$;
we call this total order a weak reflection order on $\Delta^+_{\aff} \cap \tra\inv \Delta^-_{\aff} $.

The proof of the following proposition is similar to that of \cite[Proposition 3.1.8]{NNS}.
\begin{prop}[\normalfont{see \cite[Proposition 3.1.8]{NNS} for untwisted types}]\label{goodreducedexpression}
Keep the notation and setting above.
Then, there exists a unique reduced expression $\tra= s_{\ell_1}\cdots s_{\ell_L}$ for $\tra$, $\{\ell_1, \ldots , \ell_L \} \subset I_{\aff}$,
such that
$\beta^\OS_j 
%\left( =
%s_{\ell_L}\cdots s_{\ell_{j+1}}\alpha_{\ell_j} \right) 
=\gamma_j$
for $1 \leq j \leq L$.
\end{prop}

In what following, we fix the reduced expression $\tra= s_{\ell_1}\cdots s_{\ell_L}$ for $\tra$ as above.
Recall that $\beta^\OS_j =
s_{\ell_L}\cdots s_{\ell_{j+1}}\alpha_{\ell_j}$,
$1 \leq j \leq L$.

The proof of the following lemma is similar to that of \cite[Lemma 3.1.9]{NNS}.
\begin{lem}[\normalfont{\cite[Lemma 3.1.10]{NNS}}]\label{lengthadditive}
Keep the notation and setting above.
%Since $us_{\ell_k} =s_{i'_k}u$ for some $i'_k \in I_{\aff}$,
%$1 \leq k \leq M$,
%we can rewrite the reduced expression  $u s_{\ell_1}\cdots s_{\ell_L}$ for $\tra$
%as $s_{i'_1} \cdots s_{i'_M} u s_{\ell_{M+1}} \cdots s_{\ell_L}$.
Then, $s_{\ell_1} \cdots s_{\ell_M}$ is a reduced expression for $v({\lambda_-})$,
and $s_{\ell_{M+1}} \cdots s_{\ell_L}$ is a reduced expression for $m_\lambda$.
Moreover, $i_k = \ell_k$ for $ 1 \leq k \leq M$.
\end{lem}

We set $a_k \eqdef \frac{\dg(\beta^{\OS}_k)}{c_{\overline{\beta^\OS_k}}} \in \mathbb{Z}_{> 0} $;
since 
$\Delta^+_\aff \cap \tra\inv \Delta^-_\aff
= \{ \beta^\OS_1, \ldots , \beta^\OS_L \}$,
we see by (\ref{B}) that
$0<a_k \leq \langle {\lambda_-}, \overline{\beta^\OS_k}^\lor \rangle$.

\begin{cor}\label{akbk}%\footnote{\cite{NNS}では上の Proposition の証明で書いたこと (わざわざ書かなくても良い?)}
For $1 \leq k \leq M$, $\lon \overline{\beta^\OS_k} = \beta_k$,
where $\beta_k \eqdef s_{i_N} \cdots s_{i_{k+1}} \alpha_{i_k}$.
\end{cor}

\begin{proof}
We set 
$\beta^{\rm L}_k \eqdef s_{\ell_1}\cdots s_{\ell_{k-1}}\alpha_{\ell_k}$, $1 \leq k \leq M$.
Then we have
\begin{align*}
-t({\lambda_-})\beta^{\OS}_k
&=
-( s_{\ell_1}\cdots s_{\ell_L}) ( s_{\ell_L} \cdots s_{\ell_{k+1}} \alpha_{\ell_k})
=
- s_{\ell_1}\cdots s_{\ell_{k-1}}s_{\ell_k} \alpha_{\ell_k} \\
&=
- s_{\ell_1}\cdots s_{\ell_{k-1}}( - \alpha_{\ell_k})
=
u s_{\ell_1}\cdots s_{\ell_{k-1}} \alpha_{\ell_k}
=
\beta^{\rm L}_k.
\end{align*}
From this, together with
$
-\tra \beta^{\OS}_k =
 -\overline{ \beta^{\OS}_k} 
 -c_{\overline{ \beta^{\OS}_k }}(a_k - \langle {\lambda_-}, \overline{\beta^{\OS}_k}^\lor \rangle  )\delta
$,
we obtain
$\overline{\beta^{\rm L}_k }=-\overline{\beta^{\OS}_k}$.
%and
%$\langle {\lambda_-} ,\overline{\beta^{\OS}_k} \rangle - a_k = b_k$.
Thus we have
\begin{align*}
\lon \overline{ \beta^\OS_k }=\lon (-\overline{\beta^{\rm L}_k})
&=
\lon(-\overline{ s_{\ell_1}\cdots s_{\ell_{k-1}} \alpha_{\ell_k}})
\overset{{\rm Lemma \ \ref{lengthadditive}}}{=}
\lon(-s_{i_1}\cdots s_{i_{k-1}} \alpha_{i_k})\\
&=
(s_{i_N}\cdots s_{i_1})(-s_{i_1}\cdots s_{i_{k-1}}\alpha_{i_k})=
(s_{i_N}\cdots s_{i_{k+1}} \alpha_{i_k})
=
\beta_k.
\end{align*}
\end{proof}

\begin{rem}[\normalfont{\cite[Theorem 8.3]{LNSSS2}}, \normalfont{\cite[proof of Lemma 3.1.10]{NNS}}]\label{2.15}
For $1 \leq k \leq L$, we set
	\begin{equation*}
		d_k \eqdef
			 \frac{ \langle \lambda_- ,  {\overline{\beta^{\OS}_k }}^\lor \rangle - a_k}{\langle \lambda_- ,  {\overline{\beta^{\OS}_k }}^\lor \rangle} 
			= \frac{c_{ \overline{\beta^{\OS}_k }} \langle \lambda_- ,   {\overline{\beta^{\OS}_k }}^\lor \rangle - \deg( \beta^{\OS}_k) }{c_{ \overline{\beta^{\OS}_k }}\langle \lambda_- ,   {\overline{\beta^{\OS}_k }}^\lor \rangle} 
%			= \frac{b_k}{\langle - \lambda_-  , \overline{ \beta^{\rm L}_k } \rangle};
	\end{equation*}
%the second equality follows from Remark \ref{akbk}.
Here $d_k$ is just the first component of $\Phi(\beta^{\OS}_k) \in \mathbb{Q}_{\geq 0} \times (\Delta^+ \setminus \Delta^+_S)^\lor$.
For  $1 \leq k,j \leq L$,
$\Phi(\beta^\OS_k)<\Phi(\beta^\OS_j)$ if and only if $k < j$,
and hence we have
	\begin{equation}\label{C}
		0\leq d_1 \leq \cdots \leq d_L \lneqq 1.
	\end{equation}
Moreover,
\begin{equation}\label{3.1.9}
d_i = 0  \  \Leftrightarrow \ i \leq M.
\end{equation}
\end{rem}

\begin{lem}[\normalfont{\cite[Lemma 3.1.112]{NNS}}]\label{remark2.11}
If $1\leq k<j \leq L$ and $d_k =d_j$,
then	$\lon  \overline{ \beta^{\OS}_k }^\lor  \succ \lon  \overline{ \beta^{\OS}_j }^\lor $.
\end{lem}

\subsection{Proof of Theorem~\ref{theorem_graded_character}}
We continue to use all notations in \S 4.3. 
%Let $\lambda \in Q$ be a dominant weight, and $\mu \in W\lambda$;
%recall that 
%$S = \{ i \in I \ | \ \langle \lambda , \alpha^\lor_i \rangle =0 \}$, and
%$\lambda_- = \lon \lambda$.
In this subsection, we give a bijection
	\begin{equation*}
		\Xi : \QB^\C({\id}; t({\lambda_-})) \rightarrow \QLS^\C(\lambda)
	\end{equation*}
that preserves  weights and degrees.

\begin{rem}\label{affinereflectionorder}
Let
$\gamma_1, \gamma_2 ,\ldots, \gamma_r \in \Delta^+_{\aff} \cap \tra\inv \Delta^-_{\aff}$,
and 
consider the sequence 
$\left( y_0 , y_1 , \ldots , y_r ; \gamma_1 , \gamma_2, \ldots , \gamma_r \right)$
defined by
$y_0 = \tra $, and $y_i =y_{i-1}s_{\gamma_i}$ for $1 \leq i \leq r$.
Then, the sequence 
$\left( y_0 , y_1 , \ldots , y_r ; \gamma_1 , \gamma_2, \ldots , \gamma_r \right)$
is an element of 
$\QB^\C({\id}; \tra )$
if and only if the following conditions hold{\rm :}

\begin{enu}
\item
$\gamma_1\prec' \gamma_2 \prec' \cdots \prec' \gamma_r$,
where the order $\prec'$ is 
the weak reflection order on $ \Delta^+_{\aff} \cap \tra\inv \Delta^-_{\aff}$
introduced in \S 4.3;

\item
For $1\leq i \leq r$,
$\dr(y_{i-1}) \xleftarrow{-\left( \overline{\gamma_i} \right)^\lor} \dr(y_{i})$
is
an edge of $\QBG$, 
and if this edge is a Bruhat one, then $\dg(\gamma_i) \in 2 \mathbb{Z}$.
\end{enu}
\end{rem}

In the following,
we define a map
$\Xi : \QB^\C (\id; \tra) \rightarrow \QLS^\C (\lambda)$.
Let $p^{\OS}_{J}$ be an arbitrary element of  $\QB^\C (\id; \tra)$ of the form
	\begin{equation*}
		p^{\OS}_{J} = \left( \tra = z^{\OS}_0 ,  z^{\OS}_{1} , \ldots , z^{\OS}_{r}
; \beta^{\OS}_{j_1} , \beta^{\OS}_{j_2}, \ldots , \beta^{\OS}_{j_r} \right) \in \QB^\C (\id; \tra ),
	\end{equation*}
with $J = \{ j_1 < \cdots < j_r\} \subset \{ 1 , \ldots , L \}$.
	We set $x_k \eqdef {\dr}(z^{\OS}_k)$,
	$0 \leq k \leq r$. Then, by the definition of $\QB^\C (\id ; \tra)$, 
	\begin{equation}\label{2.15.5}
		e =  x_0 \xrightarrow{- \left( \overline{ \beta^{\OS}_{j_1} } \right)^{\lor}  }   x_1
		\xrightarrow{- \left( \overline {\beta^{\OS}_{j_2} } \right)^{\lor} } \cdots  \xrightarrow{ - \left( \overline{ \beta^{\OS}_{j_r} } \right)^{\lor} }  x_r
	\end{equation}
	is a directed path in $\QBG$; the first equality follows from \eqref{equ:direction}.
	We take $0 = u_0 \leq u_1 < \cdots < u_{s-1} < u_s=r$ and $0 \leq \sigma_1 < \cdots <\sigma_{s-1} < 1 = \sigma_{s}$ in such a way that
(see (\ref{C}))
\vspace{3mm}
	\begin{equation}\label{2.16}
		 \underbrace{0 = d_{j_1} = \cdots = d_{j_{u_1}} }_{=\sigma_0}
		< \underbrace{d_{j_{u_1 +1}} = \cdots =d_{j_{u_2}}}_{=\sigma_1} < \cdots <
		\underbrace{ d_{j_{u_{s-1}+1}} = \cdots =d_{j_r} }_{=\sigma_{s-1}} <1 = \sigma_{s};
	\end{equation}
	note that 
	$d_{j_1}>0$ if and only if $u_1=0$,
	and $\sigma_0=0$ even if $d_{j_1}=0$.
	We set
	$w'_p \eqdef x_{u_p}$ for $0 \leq p \leq s-1$, and
	$w'_s \eqdef x_r$. 
	Then, by taking a subsequence of (\ref{2.15.5}), we obtain the following  directed path in $\QBG$
	for each $0 \leq p \leq s-1$:
	\begin{equation*}
		w'_p = x_{{u_p}} \xrightarrow{- \left( \overline{ \beta^{\OS}_{j_{u_p +1}} } \right)^{\lor} }   x_{{u_p +1}} 
		\xrightarrow{- \left( \overline {  \beta^{\OS}_{j_{u_p +2}} } \right)^{\lor} } \cdots  
		\xrightarrow{ -\left( \overline{  \beta^{\OS}_{j_{u_{p+1}}} } \right)^{\lor} }  x_{{u_{p+1}}} = w'_{p+1}.
	\end{equation*}
	Multiplying this directed path on the right by $\lon$, 
	we obtain the following directed path in $\QBG$
	for each $0 \leq p \leq s-1$
	(see Lemma \ref{involution}):
	\begin{equation}\label{w'_p}
		w_p \eqdef
		w'_p \lon =
		x_{{u_p}} \lon \xleftarrow{ \lon \left( \overline{ \beta^{\OS}_{j_{u_p +1}} } \right)^{\lor} }  \cdots  
		\xleftarrow{ \lon \left( \overline{ \beta^{\OS}_{j_{u_{p+1}}} } \right)^{\lor} }  x_{{u_{p+1}}} \lon
		= w'_{p+1} \lon \eqdef w_{p+1}.
	\end{equation}
	Note that $w_0 = w'_0 \lon = x_0 \lon = \lon$, and
the edge labels of this directed path are increasing
in the weak reflection order $\prec$ on $(\Delta^+)^\lor$ introduced at the beginning of \S 4.3
(see Lemma \ref{remark2.11})
	  and lie in $(\Delta^+ \setminus \Delta^+_S)^\lor$;
	this property will be used to give the inverse to $\Xi$.
	Because
	\begin{equation}\label{sigma_with_a}
		\sigma_p \langle \lambda , \lon \overline{ {\beta^{\OS}_{j_u}} }^\lor   \rangle
		=
		d_{j_u} \langle \lambda , \lon \overline{ {\beta^{\OS}_{j_u}} }^\lor   \rangle
		=\frac{\langle \lambda_- ,  \overline{ {\beta^{\OS}_{j_u}} }^\lor   \rangle - a_{j_u}}{\langle \lambda_- ,  \overline{ {\beta^{\OS}_{j_u}} }^\lor   \rangle} \langle \lambda , \lon \overline{ {\beta^{\OS}_{j_u}} }^\lor   \rangle
		=\langle \lambda_- ,  \overline{ {\beta^{\OS}_{j_u}} }^\lor   \rangle - a_{j_u} \in \mathbb{Z}
	\end{equation}
	for $u_p +1 \leq u \leq u_{p+1}$, $0 \leq p \leq s-1$,
	we find that (\ref{w'_p}) is a  directed  path in $\QBG_{\sigma_p \lambda}^\C$ for $0 \leq p \leq s-1$;
	indeed, if $ x_{u-1} \lon \xleftarrow{\lon \left( \overline{ \beta^{\OS}_{j_{u} }} \right)^{\lor} }x_{u} \lon$
	is a Bruhat edge with $\lon  \overline{ \beta^{\OS}_{j_{u} }}  $ a short root,
	that is, $\dr (z_{u-1}^\OS)=x_{u-1}  \xrightarrow{- \left( \overline{ \beta^{\OS}_{j_{u} }} \right)^{\lor} }x_{u}=\dr(z_u^\OS)$
	is  a Bruhat edge with $\overline{ \beta^{\OS}_{j_{u} }}  $ a short root, then
	$\dg(\beta^{\OS}_{j_{u}}) = c_{\overline{\beta^{\OS}_{j_{u}}}}a_{j_u}\in 2\mathbb{Z}$ by the definition of $\QB^\C(\id; \tra)$, and hence
	$a_{j_u} \in 2\mathbb{Z}$ since $c_{\overline{ \beta^{\OS}_{j_{u} }} } = 1$.
	Also, $\langle \lambda_- ,  \overline{ {\beta^{\OS}_{j_u}} }^\lor \rangle \in 2\mathbb{Z}$ since $\overline{ \beta^{\OS}_{j_{u} }}  $ is a short root.
	Hence we have $\sigma_p \langle \lambda , \lon \overline{ {\beta^{\OS}_{j_u}} }^\lor   \rangle \in 2\mathbb{Z}$ by (\ref{sigma_with_a}).
	Therefore, by Lemma \ref{lem:6.1}, there exists a  directed path in $(\QBG_{\sigma_p \lambda}^\C)^S$ from $\lfloor w_{p+1} \rfloor$ to $\lfloor w_p \rfloor$,
where
$S = \{ i \in I \ | \ \langle \lambda , \alpha^\lor_i \rangle =0 \}$.
Also, we claim that $\lfloor w_p \rfloor \neq \lfloor w_{p+1} \rfloor$ for $1 \leq p \leq s-1$. Suppose,
for a contradiction, that $\lfloor w_p \rfloor = \lfloor w_{p+1} \rfloor$
for some $p$.
Then, $w_p W_S = w_{p+1} W_S$,
and hence $\min(w_{p}W_S, w_{p+1}) = \min(w_{p+1}W_S, w_{p+1}) =w_{p+1}$.
Recall that 
the directed path (\ref{w'_p}) is a path in QBG from $w_{p+1}$ to $w_{p}$ whose labels are increasing and lie in $\Delta^+ \setminus \Delta^+_S$.
%Therefore, the directed path (\ref{w'_p}) is a shortest path from $w_{p}$ to $w_{p+1}$ by Proposition \ref{shellability} (1).
By Lemma \ref{8.5} (1), (2), the directed path (\ref{w'_p}) is a shortest path in QBG from $\min(w_{p+1}W_S, w_p) = \min(w_{p}W_S, w_p) =w_p$ to $w_p$,
which implies that the length of the directed path (\ref{w'_p}) is equal to $0$.
Therefore, $\{ j_{u_p +1}, \ldots , j_{u_{p+1}} \} = \emptyset$, and hence $u_p = u_{p+1}$,
which contradicts the fact that $u_p < u_{p+1}$.

	Thus we obtain
	\begin{equation}\label{D}
		\eta \eqdef 
		(\lfloor w_{1} \rfloor , \ldots \ , \lfloor w_{s} \rfloor ; \sigma_0, \ldots, \sigma_s)  \in \QLS^\C(\lambda).
	\end{equation}
We now define $\Xi (p^{\OS}_{J}) \eqdef \eta$.

\begin{lem}\label{injective}
		The map $\Xi:\QB^\C(\id; \tra) \rightarrow \QLS^\C(\lambda)$
is injective.
\end{lem}

\begin{proof}
We first show that the map $\Xi$ is injective.
Let $J= \{ j_1 , \ldots , j_r \}$ and $K = \{ k_1 , \ldots , k_{r'} \}$ be subsets of $\{ 1, \ldots , L\}$ such that
$\Xi (p^\OS_J)= \Xi (p^\OS_K) = (v_1, \ldots, v_s ; \sigma_0 , \ldots , \sigma_s) \in \QLS^\C(\lambda)$.
As in (\ref{2.16}), we set $0 = u_0 \leq u_1 < \cdots < u_{s-1} = r$ and $0 = u'_0 \leq u'_1 < \cdots < u'_{s-1} = r'$
in such a way that
	\begin{align*}
		 \underbrace{0 = d_{j_1} = \cdots = d_{j_{u_1}} }_{=\sigma_0}
		< \underbrace{d_{j_{u_1 +1}} = \cdots =d_{j_{u_2}}}_{=\sigma_1} < \cdots <
		\underbrace{ d_{j_{u_{s-1}+1}} = \cdots =d_{j_r} }_{=\sigma_{s-1}} <1 = \sigma_{s},\\
		 \underbrace{0 = d_{k_1} = \cdots = d_{k_{u'_1}} }_{=\sigma_0}
		< \underbrace{d_{k_{u'_1 +1}} = \cdots =d_{k_{u'_2}}}_{=\sigma_1} < \cdots <
		\underbrace{ d_{k_{u'_{s-1}+1}} = \cdots =d_{k_{r'}} }_{=\sigma_{s-1}} <1 = \sigma_{s}.
	\end{align*}
As in (\ref{w'_p}), we consider the directed paths in $\QBG$
	\begin{align}\label{surj_induction}
\begin{array}{l}
		w_p \xleftarrow{ \lon \left( \overline{ \beta^{\OS}_{j_{u_p +1}} } \right)^{\lor} }  \cdots  
		\xleftarrow{ \lon \left( \overline{ \beta^{\OS}_{j_{u_{p+1}}} } \right)^{\lor} }  w_{p+1}, \\
  		y_p  \xleftarrow{ \lon \left( \overline{ \beta^{\OS}_{k_{u'_p +1}} } \right)^{\lor} }  \cdots  
		\xleftarrow{ \lon \left( \overline{ \beta^{\OS}_{k_{u'_p}} } \right)^{\lor} } y_{p+1},
\end{array}
\ \
\mbox{for }0\leq p\leq s-1;
	\end{align}
here we note that $w_0 = y_0 =e$, and  $\lfloor w_{p} \rfloor = \lfloor y_{p} \rfloor = v_p$, $0\leq p\leq s-1$.

Suppose that $w_p = y_p$ and $u_p = u'_p$.
Then the paths (\ref{surj_induction}) are both directed paths from some element in $v_{p+1} W_S$ to $w_p$ in $\QBG$ whose labels are increasing and lie in $(\Delta^+ \setminus \Delta_S^+)^\lor$.
Hence $w_{p+1} = y_{p+1} \in v_{p+1} W_S$ and $\lon \overline{\beta_{j_i}^\OS} = \lon \overline{\beta_{k_i}^\OS}$ for $u_p+1 \leq i \leq u_{p+1}$, and $u_{p+1} = u_{p+1}'$
by Lemma \ref{8.5} (2).
Also,
$d_{j_i} = d_{k_i} = \sigma_p$ for $u_p+1 \leq i \leq u_{p+1}$.
It follows from 
$\beta^\OS_{j} = \overline{\beta_j^\OS} +c_{\overline{\beta_j^\OS}} (1 -d_j)\langle \lambda_- , \overline{\beta_j^\OS}^\lor \rangle \delta$, 
$1 \leq j \leq L$,
that
$\beta_{j_i}^\OS = \beta_{k_i}^\OS$, $u_p+1 \leq i \leq u_{p+1}$.
Thus, by induction on $p$, we deduce that $u_p = u_p'$ for $0 \leq p \leq s-1$, and $\beta_{j_i}^\OS = \beta_{k_i}^\OS$, $u_0+1 \leq i \leq u_{s-1}$.
Therefore, $r = u_{s-1}= u_{s-1}' =r'$, and hence 
$J= \{ j_1 , \ldots , j_r \} = \{ k_1 , \ldots , k_{r'} \}=K$. 
\end{proof}

\begin{lem}\label{surjective}
		The map $\Xi:\QB^\C(\id; \tra) \rightarrow \QLS^\C(\lambda)$
is surjective.
\end{lem}

\begin{proof}
Take an arbitrary
$
		\eta =(y_{1} , \ldots \ , y_{s} ; {\tau}_{0},\ldots, {\tau}_{s})  \in \QLS^\C(\lambda)
$.
By convention, we set $y_0=v(\lambda_-) \in W^S$.
We define the elements $v_{p}$, $0 \leq p \leq s$, by:
$v_0 =\lon $,
and
$v_{p} = \min(y_{p}W_S , \leq_{v_{p-1}})$  for $1\leq p \leq s$.

Because there exists a  directed path in $(\QBG^\C_{\tau_p \lambda})^S$ from $y_{p+1}$ to $y_{p}$ for $1\leq p \leq s-1$,
we see from Lemma \ref{8.5} (2), (3) that
there exists a unique  directed path 
\begin{equation}\label{path3}
v_p \xleftarrow{-\lon \gamma_{p,1}^\lor} \cdots \xleftarrow{-\lon \gamma_{p,t_{p}}^\lor} v_{p+1}
\end{equation}
in $\QBG_{\tau_p \lambda}^\C$ from $v_{p+1}$ to $v_p$ whose edge labels $-\lon\gamma_{p, t_p}^\lor, \ldots, -\lon \gamma_{p, 1}^\lor$ are increasing 
in the weak reflection order $\prec$
and lie in $(\Delta^+ \setminus \Delta^+_S)^\lor$
for $1\leq p \leq s-1$;
we remark that this is also true for $p=0$, since
$\tau_0=0$.
Multiplying this directed path on the right by $\lon$, 
we obtain by Lemma \ref{involution} the following directed  paths
	\begin{equation*}
		v_{p,  0} \eqdef v_{p}\lon  
		\xrightarrow{ \gamma_{p, 1}^\lor }  v_{p,  1} 
		\xrightarrow{ \gamma_{p, 2}^\lor } 
		\cdots \xrightarrow{ \gamma_{p, t_p}^\lor } v_{p+1}\lon 
		\eqdef v_{p,  t_p},
\ \
0 \leq p \leq s-1.
	\end{equation*}
Concatenating these paths for $0 \leq p\leq s-1$, we obtain the following directed  path
	\begin{align}\label{recover}
		\id &=v_{0,0} \xrightarrow{ \gamma_{0, 1}^\lor } \cdots  
		\xrightarrow{ \gamma_{0, t_0}^\lor }v_{0,  t_0}=v_{1,  0}
		\xrightarrow{ \gamma_{1, 1}^\lor } 
		\cdots 
 \xrightarrow{ \gamma_{1, t_1}^\lor }v_{1, t_1} = v_{2,  0}  \xrightarrow{ \gamma_{2, 1}^\lor }  \cdots
 \nonumber
\\
&\cdots \xrightarrow{ \gamma_{s-1, t_{s-1}}^\lor }  v_{s-1,  t_{s-1}}
	\end{align}
	in $\QBG$. 
	Now, for $0\leq p \leq s-1$ and $1 \leq m \leq t_p$,
	we set $d_{p,m}  \eqdef \tau_{p} \in \mathbb{Q}\cap [0,1)$, 
	$a_{p,m} \eqdef (1-d_{p,m})\langle \lambda_- , -{\gamma}^{\lor}_{p,m}   \rangle$, 
	and 
	$\widetilde{\gamma}_{p,m}  \eqdef c_{\gamma_{p,m}}a_{p,m} \delta - \gamma_{p,m}$.
	It follows from (\ref{B}) that $\widetilde{\gamma}_{p,m}   \in 
\Delta^+_{\aff} \cap \tra\inv \Delta^-_{\aff}$.

%%%%%%%%%%%%%%%%%%%%%%%%%%%%%%%%%%%%%%%%%%%%%%

%By Claim 1, we can choose $J'= \{ j'_1 < \ldots < j'_{r'} \} \subset \{ M'+1 ,\ldots , L \}$ such that
%	\begin{equation*}
%	\left\{ \beta^{\OS}_j \ | \ j \in J'  \right\} 
%= \left\{ \widetilde {\gamma_{p,m} } \ | \ 1\leq p \leq s, 1 \leq m \leq t_p \right\} 
%\subset \Delta^+_{\aff} \cap m^{-1}_\mu \Delta^-_{\aff} = \left\{ \beta^{\OS}_j \ | \ M'+1 \leq j \leq L   \right\} .
%	\end{equation*}
%Hence we can define  
%$
%p^{\OS}_{J'} 
%= \left( m_{\mu} = z^{\OS}_0 ,  z^{\OS}_{1} , \ldots , z^{\OS}_{r'}
%; \beta^{\OS}_{j'_1} , \beta^{\OS}_{j'_2}, \ldots , \beta^{\OS}_{j'_{r'}} \right) 
%$
%by 
%$z^{\OS}_{0}=m_{\mu}$, $z^{\OS}_{k}=z^\OS_{k-1} s_{\beta^{\OS}_{j'_{k}}}$ for $1\leq k \leq r'$.
%
%If Claim 2 below holds,
%then
%we obtain $p^{\OS}_{J'}\in \overleftarrow{\QB}({\id}; m_{\mu})$ 
%by Remark \ref{affinereflectionorder},
%and hence we can define $\Theta :\QLS^{\mu, \infty}(\lambda) \rightarrow \overleftarrow{\QB}({\id}; m_{\mu})$ by
%$\Theta (\eta) \overset{\text{def}}{=} p^{\OS}_{J'}$.

\begin{nclaim}
\begin{enu}
\item
%The set $J'$ above satisfies the following: 
We have
\begin{equation*}
%\left( \beta^{\OS}_{j'_1} \prec' \cdots \prec' \beta^{\OS}_{j'_{r'}} \right)
%=
%\left(  
\widetilde{\gamma}_{0,1} \prec'  \cdots \prec' \widetilde{\gamma}_{0,t_0} 
\prec'  
\widetilde{\gamma}_{1,1} \prec' \cdots \prec' \widetilde{\gamma}_{s-1,t_{s-1}},
%\right) ,
\end{equation*}
where $\prec'$ denotes the weak reflection order on $\Delta^+_{\aff} \cap m^{-1}_{\lambda_-} \Delta^-_{\aff}$
introduced in \S $4.3$;
%(see the proof of Proposition \ref{goodreducedexpression})
we choose 
$J' = \{ j_1 , \ldots , j'_{r'} \} \subset \{ 1 , \ldots , L \}$ in such way that
\begin{equation*}
\left( \beta^{\OS}_{j'_1} , \cdots , \beta^{\OS}_{j'_{r'}} \right)
=
\left(  
\widetilde{\gamma}_{0,1} ,  \cdots , \widetilde{\gamma}_{0, t_0} ,
\widetilde{\gamma}_{1,1} , \cdots , \widetilde{\gamma}_{s-1,t_{s-1}}
\right) .
\end{equation*}

\item
%For $1 \leq k \leq r'$, we set
Let $1 \leq k \leq r'$, and take $1 \leq p \leq s$ such that
\begin{equation*}
\left( \beta^{\OS}_{j'_1} \prec' \cdots \prec' \beta^{\OS}_{j'_{k}} \right)
=
\left(  \widetilde{\gamma}_{0,1} \prec'  \cdots \prec'  \widetilde{\gamma}_{p,m} 
 \right).
\end{equation*}
%$1 \leq p \leq s, 1 \leq m \leq t_p$.
Then, we have
$\dr(z^{\OS}_{k}) = v_{p,m-1}$.
Moreover, 
$\dr(z^{\OS}_{k-1})\xrightarrow{-\left(\overline{ \beta^\OS_{j'_k} } \right)^\lor} \dr(z^{\OS}_k)$ is an edge of $\QBG$.
\end{enu}
\end{nclaim}

%\begin{proof}[Proof of Claim $2$]
\noindent $Proof \ of \ Claim \ 1.$
(1)
It suffices to show the following:

(i)
for $0 \leq p \leq s-1$ and $1\leq m < t_p$,
we have
$  \widetilde{\gamma}_{p,m}\prec' \widetilde{\gamma}_{p,m+1}$;

(ii)
for $1 \leq p \leq s-1$,
we have
$\widetilde{\gamma}_{p,t_p} \prec' \widetilde{\gamma}_{p+1,1}$.

(i)
Because
$\frac{\langle \lambda_- , -\gamma_{p,m}^{\lor}\rangle - a_{p,m}}{\langle \lambda_- , -\gamma_{p,m}^{\lor}\rangle}=d_{p,m}$ and
$\frac{\langle \lambda_- , -\gamma_{p,m+1}^{\lor}\rangle - a_{p,m+1}}{\langle \lambda_- , -\gamma_{p,m+1}^{\lor}\rangle}=d_{p,m+1}$,
we have
\begin{align*}
\Phi (\widetilde{\gamma}_{p,m}) &=
 (d_{p,m}, - \lon \gamma_{p,m} ), \\
\Phi (\widetilde{\gamma}_{p,m+1}) &=
 (d_{p,m+1}, - \lon \gamma_{p,m+1} ).
\end{align*}
Therefore, the first component of $ \Phi (\widetilde{\gamma}_{p,m})$ is equal to that of $\Phi (\widetilde{\gamma}_{p,m+1})  $
since $d_{p,m} =1 - \tau_p =d_{p,m+1}$.
Moreover, since $- \lon \gamma_{p,m} \prec -\lon \gamma_{p,m+1}$,
we have $ \Phi (\widetilde{\gamma}_{p,m}) < \Phi (\widetilde{\gamma}_{p,m+1})  $.
This implies that $  \widetilde{\gamma}_{p,m}\prec' \widetilde{\gamma}_{p,m+1}$.

(ii)
The proof of (ii) is similar to that of (i).
The first components of $\Phi (\widetilde{\gamma}_{p,t_p}) $ and $\Phi (\widetilde{\gamma}_{p+1,1}) $ are $d_{p,t_p}$ and $d_{p+1,1}$, respectively.
Since $d_{p,t_p}= \tau_p < \tau_{p+1}= d_{p+1,1}$,
we have $\Phi (\widetilde{\gamma}_{p,t_p}) < \Phi (\widetilde{\gamma}_{p+1,1}) $.
This implies that
$\widetilde{\gamma}_{p,t_p} \prec' \widetilde{\gamma}_{p+1,1}$.

(2)
We proceed 
by induction on $k$.
If $\beta_{j'_1}^\OS = \gamma_{0,1}$, i.e., $y_1 \neq v(\lambda_-)$,
then, 
we have
$\dr (z^\OS_1) = \dr (z^\OS_0) s_{-\overline{\beta^{\OS}_{j'_1}}} =v_{0,0}  s_{\gamma_{0,1}} = v_{0,1}$
since $\dr (z^\OS_0) = \dr(t(\lambda_-))= \id = v_{0,0}$.
If $\beta_{j'_1}^\OS = \gamma_{1,1}$, i.e., $y_1 = v(\lambda_-)$ and $t_0 =0$,
then, 
we have
$\dr (z^\OS_1) = \dr (z^\OS_0) s_{-\overline{\beta^{\OS}_{j'_1}}} = s_{\gamma_{1,1}} =
v_{1,0} s_{\gamma_{1,1}}= v_{1,1}$
since $\dr (z^\OS_0) = \dr(\tra)=\id=v_{1,0}$.
Hence the assertion holds in the case $k=1$.

Assume that $\dr(z^{\OS}_{k-1}) = v_{p,m-1}$
for $0 \lneqq m \leq t_p$;
here we remark that $v_{p,m}$ is the successor of $v_{p,m-1}$ in the  directed path (\ref{recover}).
Hence we have
$\dr(z^{\OS}_{k}) 
= \dr(z^{\OS}_{k-1}) s_{-\overline{ \beta^\OS_{j'_k} } }
= v_{p,m-1}s_{\gamma_{p,m}}
\overset{(\ref{recover})}=v_{p,m}$.
Also, since (\ref{recover}) is a directed path in $\QBG$, 
$ v_{p,m} = \dr(z^{\OS}_{k-1})\xrightarrow{-\left(\overline{ \beta^\OS_{j'_k} } \right)^\lor} \dr(z^{\OS}_k) = v_{p,m-1} $ is an edge of $\QBG$.
%\end{proof}
\bqed

Since $J' = \{ j_1 , \ldots , j'_{r'} \} \subset \{ K+1 , \ldots , L \}$,
we can define an element
$
p^{\OS}_{J'} $
to be \\
$ \left( m_{\mu} = z^{\OS}_0 ,  z^{\OS}_{1} , \ldots , z^{\OS}_{r'}
; \beta^{\OS}_{j'_1} , \beta^{\OS}_{j'_2}, \ldots , \beta^{\OS}_{j'_{r'}} \right) 
$
where
$z^{\OS}_{0}=m_{\mu}$, $z^{\OS}_{k}=z^\OS_{k-1} s_{\beta^{\OS}_{j'_{k}}}$ for $1\leq k \leq r'$;
it follows from Remark \ref{affinereflectionorder} and Claim 1 that
 $p^{\OS}_{J'} \in \QB^\C(\id; \tra)$.
%Hence we can define a map $\Theta : \QLS^{\mu, \infty}(\lambda) \rightarrow \overleftarrow{\QB}({\id}; m_{\mu})$ by
%$\Theta (\eta) \eqdef p^{\OS}_{J'}$.

%Hence we can define  
%$
%p^{\OS}_{J'} 
%= \left( m_{\mu} = z^{\OS}_0 ,  z^{\OS}_{1} , \ldots , z^{\OS}_{r'}
%; \beta^{\OS}_{j'_1} , \beta^{\OS}_{j'_2}, \ldots , \beta^{\OS}_{j'_{r'}} \right) 
%$
%by 
%$z^{\OS}_{0}=m_{\mu}$, $z^{\OS}_{k}=z^\OS_{k-1} s_{\beta^{\OS}_{j'_{k}}}$ for $1\leq k \leq r'$.
%
%If Claim 2 below holds,
%then
%we obtain $p^{\OS}_{J'}\in \overleftarrow{\QB}({\id}; m_{\mu})$ 
%by Remark \ref{affinereflectionorder},
%and hence we can define $\Theta :\QLS^{\mu, \infty}(\lambda) \rightarrow \overleftarrow{\QB}({\id}; m_{\mu})$ by
%$\Theta (\eta) \overset{\text{def}}{=} p^{\OS}_{J'}$.

\begin{nclaim}
$\Xi (p^\OS_{J'})=\eta$.
\end{nclaim}

%\begin{nclaim}
%For $p^{\OS}_{J} 
%= \left( m_{\mu} = z^{\OS}_0 ,  z^{\OS}_{1} , \ldots , z^{\OS}_{r}
%; \beta^{\OS}_{j_1} , \beta^{\OS}_{j_2}, \ldots , \beta^{\OS}_{j_r} \right) \in \overleftarrow{\QB}(\id; m_{\mu})$,
%we have $\Theta \circ   \Xi (p^{\OS}_{J}) = p^{\OS}_{J}$.
%\end{nclaim}

\noindent $Proof \ of \ Claim \ 2.$
%We set $\Theta(\eta)= p_{J'}^\OS$, with $J'=\{ j'_1 , \ldots , j'_r \}$;
In the following description of $ p^\OS_{J'}$,
we employ the notation  $u_p$, $\sigma_p$, $w'_p$, and $w_p$ 
used  in the definition of $\Xi (p^\OS_{J})$.

For $1 \leq k \leq r'$, if we set $\beta^{\OS}_{j'_k}=\widetilde{\gamma}_{p,m}$,
then we have
\begin{equation*}
d_{j'_k}=
1 + \frac{\dg(\beta^{\OS}_{j'_k})}{c_{\overline{ \beta^{\OS}_{j'_k} }}\langle \lambda_- , -\overline{ \beta^{\OS}_{j'_k} }^\lor \rangle}
=
1 + 
\frac{\dg(\widetilde{\gamma}_{p,m})}{c_{{\gamma}_{p,m}}\langle \lambda_- , -\overline{\widetilde{\gamma}_{p,m}}^\lor \rangle}
=
1 + \frac{c_{{\gamma}_{p,m}}a_{p,m}}{c_{{\gamma}_{p,m}}\langle \lambda_- , \gamma_{p,m}^\lor \rangle}
=
d_{p,m}
.
\end{equation*}
Therefore,
the sequence (\ref{2.16}) determined by $p^{\OS}_{J'}$ is 
\begin{equation}\label{2.17}
		\underbrace{ 0 = d_{0,1} = \cdots d_{0,t_0} }_{=\tau_0}
		< \underbrace{d_{1,1} = \cdots =d_{1, t_1}}_{=\tau_1}
		 < \cdots <
		\underbrace{ d_{s-1,1} = \cdots =d_{s-1,t_{s-1}} }_{=\tau_{s-1}} <1 = \tau_{s} = \sigma_s.
	\end{equation}
Because the sequences (\ref{2.17}) of rational numbers is just the sequence (\ref{2.16}) for 
$\Theta ( \eta ) = p_{J'}^\OS$,
we deduce that 
$u_{p+1} - u_p = t_p$ for $0 \leq p \leq s-1$,
$\beta^\OS_{j'_{u_p + k}} = \widetilde{\gamma}_{p,k}$ for $0 \leq p \leq s-1$, $1 \leq k \leq u_{p+1} - u_p$,
and
$\sigma_p = \tau_{p}$ for $0 \leq p \leq s$.
Therefore, we have $w'_p = \dr(z^\OS_{u_p})=v_{p,0}$ and $w_p = v_{p,0}\lon =v_{p}$.
Since $\lfloor w_p \rfloor = \lfloor v_{p} \rfloor = y_{p}$,
we conclude that $\Xi (p^\OS_{J'}) = (\lfloor w_1 \rfloor ,\ldots , \lfloor w_s \rfloor ; \sigma_0, \ldots , \sigma_s)=(y_1 ,\ldots, y_s ; \tau_0 ,\ldots , \tau_s)= \eta$.
%\end{proof}
\bqed

%%%%%%%%%%%%%%%%%%%%%%%%%%%%%%%%%%%%%%%%
This completes the proof of Lemma \ref{surjective}.
%%%%%%%%%%%%%%%%%%%%%%%%%%%%%%%%%%%%%%%%
\end{proof}

By Lemmas \ref{injective} and \ref{surjective}, we have the following Proposition.
\begin{prop}\label{bijective}
The map $\Xi$ is bijective.
\end{prop}

We recall from \eqref{eq:4-1} and \eqref{equ:decomposition} that $\dg(x)$ is defined by: $x= \overline{x} + \dg(x)\delta$ for $x \in \mathfrak{h}_\mathbb{R}^* \oplus \mathbb{R}\delta$, and $\wt(u) \in Q$ and $\dr(u) \in W$ are defined by: $u=t(\wt (u))\dr(u)$ for $u\in W_\aff = t(Q)\rtimes W$.
\begin{prop}\label{qls_qb}
The bijection $\Xi:\QB^\C(\id; \tra) \rightarrow \QLS^\C(\lambda)$ satisfies the following{\rm:}
\begin{enu}
\item
		$\wt (\ed (p^{\OS}_{J})) = \wt(\Xi(p^\OS_{J}) )${\rm;}
\item
		$\dg (\qwt(p^{\OS}_{J})) = \Dg (\Xi(p^{\OS}_{J}))$.
\end{enu}
\end{prop}
\setcounter{nclaim}{0} %%Claimの番号をリセット

\begin{proof}
We proceed by induction on $ \#J$.

If $J=\emptyset$, it is obvious that $\deg (\qwt(p^{\OS}_{J})) = \Dg (\Xi (p^{\OS}_{J}))=0$ and 
${\wt}(\ed (p^{\OS}_{J})) = \wt(\Xi(p^{\OS}_{J}) )=\lambda_-$,
since
$\Xi(p^{\OS}_{J})=(v(\lambda_-); 0,1)$.
	
Let	$J=\{ j_1 <j_2 <\cdots <j_r\}$, and set
	$K \eqdef J\setminus \{ j_r \}$;
	assume that $\Xi (p^{\OS}_{K})$ is of the form: 
$\Xi (p^{\OS}_{K}) =
  (\lfloor w_1 \rfloor  ,\ldots , \lfloor w_s \rfloor ; \sigma_0 ,\ldots, \sigma_s)  $.
In the following, we employ the notation $w_p$, $0\leq p \leq s$, used in the definition of the map $\Xi$. 
Note that
$\dr(p^\OS_K)=w_s \lon$
by the definition of $\Xi$.
Also, observe that
if $d_{j_{r}}=d_{j_{r-1}}={\sigma}_{s-1}$, 
then $\{ d_{j_1} \leq \cdots \leq d_{j_{r-1}}\leq d_{j_r} \}
=
\{ d_{j_1} \leq \cdots \leq d_{j_{r-1}} \} $,
and
if $d_{j_{r}}>d_{j_{r-1}}={\sigma}_{s-1}$,
then
$\{ d_{j_1} \leq \cdots \leq d_{j_{r-1}}\leq d_{j_r} \}
= 
\{ d_{j_1} \leq \cdots \leq d_{j_{r-1}}< d_{j_r} \}
$.
From these, we deduce that
	\begin{align*}
	\Xi (p^{\OS}_{J})=
\begin{cases}
(\lfloor w_1 \rfloor ,\ldots,  \lfloor w_{s-1}\rfloor , \lfloor w_{s} s_{\lon \overline {\beta^{\OS}_{j_r} }  }  \rfloor ; \sigma_0, \ldots, \sigma_{s-1}, \sigma_s)
& \mbox{if } d_{j_{r}}=d_{j_{r-1}} ={\sigma}_{s-1}, \\
(\lfloor w_1 \rfloor ,\ldots , \lfloor w_{s-1} \rfloor ,  \lfloor w_{s} \rfloor ,\lfloor  w_{s}s_{\lon \overline {\beta^{\OS}_{j_r }} }  \rfloor ; \sigma_0, \ldots, \sigma_{s-1}, d_{j_r}, \sigma_{s})
& \mbox{if } d_{j_{r}}>d_{j_{r-1}} ={\sigma}_{s-1}.
\end{cases}
	\end{align*}

For the induction step, it suffices to show the following claims.

\begin{nclaim}
\
\begin{enu}
\item
We have
	\begin{equation*}
		\wt ( \Xi (p^{\OS}_{J}) ) 
		= \wt ( \Xi (p^{\OS}_{K}) )+ 
		a_{j_{r}}w_{s}\lon \left( -\overline {\beta^{\OS}_{j_r } } \right).
	\end{equation*}

\item
We have
\begin{equation*}
		\Dg ( \Xi (p^{\OS}_{J}) ) 
		= \Dg ( \Xi (p^{\OS}_{K}) )+ \zeta \dg(\beta^{\OS}_{j_r} ),
	\end{equation*}
where
$\zeta \eqdef 0$ {\rm(}resp., $\zeta \eqdef 1${\rm)} 
if $w_{s}s_{\lon \overline {\beta^{\OS}_{j_r } }  } \leftarrow w_{s}$ is a Bruhat {\rm(}resp., quantum{\rm)} edge.
\end{enu}
\end{nclaim}

\begin{nclaim}
\
\begin{enu}
\item
We have
	\begin{equation*}
		\wt ( \ed (p^{\OS}_{J}) ) 
		= \wt ( \ed (p^{\OS}_{K}) )+ 
		a_{j_{r}}w_{s}\lon \left( -\overline {\beta^{\OS}_{j_r } } \right).
	\end{equation*}

\item
We have
	\begin{equation*}
		\dg ( \qwt (p^{\OS}_{J}) )
		= \dg ( \qwt (p^{\OS}_{K}) ) + \zeta \dg(\beta^{\OS}_{j_r } ).
	\end{equation*}
\end{enu}
\end{nclaim}

The proofs of Claims 1 and 2 are similar to those of Claims 1 and 2 in \cite[Proposition 3.2.6]{NNS}.

\end{proof}

\begin{proof}[Proof of Theorem $\ref{theorem_graded_character}$]
We know from Corollary \ref{oss} that
	\begin{equation*}
		P_\lambda^\C (q,0)=
		\sum_{p^{\OS}_{J} \in\QB^\C({\id} ; \tra ) } 
		q^{{\dg}({\qwt}(p^{\OS}_{J}))}e^{\wt({\ed}(p^{\OS}_{J}))}.
	\end{equation*}
Therefore, it follows from Propositions \ref{bijective} and \ref{qls_qb} that
\begin{equation*}
		P_\lambda^\C (q,0) =
		\sum_{\eta \in {\QLS^\C(\lambda)} }
		q^{\Dg(\eta)}e^{\wt(\eta)}.
	\end{equation*}
Hence we conclude that $P_\lambda^\C (q,0) = \gch \QLS^\C (\lambda)$, as desired.
\end{proof}

\subsection{Nonsymmetric Macdonald polynomials $E_\mu^\C (q, 0)$ at $t=0$ with respect to arbitrary weights}

Let $\lambda \in Q$ be a dominant weight.
By Remark \ref{Lemma7.7} and Theorem \ref{theorem_graded_character},
we have $E_{\lon \lambda}^\C(q,0) = \sum_{\eta \in \QLS^\C(\lambda)} q^{\Dg(\eta)}e^{\wt(\eta)}$.
In general, we have the following theorem.
The proof is similar to that of \cite[Theorem 1.1]{LNSSS4}.

For $\eta = (w_1, \ldots , w_s; \sigma_0, \ldots , \sigma_s) \in \QLS^\C(\lambda)$, we set $i(\eta)\eqdef w_1$.
Then for $\mu \in W\lambda$, we define $\QLS^\C_{\leq v(\mu)}(\lambda) \eqdef
\{
\eta \in \QLS^\C(\lambda)
\ | \
i( \eta ) \leq v(\mu) \}
$,
where $\leq$ denotes the Bruhat order on $W$.

\begin{thm}
Let $\lambda \in Q$ be a dominant weight and
$\mu \in W \lambda$.
Then
\begin{equation*}
E^\C_\mu (q, 0) =\sum_{\eta \in \QLS^\C_{\leq v(\mu)} (\lambda)} q^{\Dg(\eta)}e^{\wt(\eta)}.
\end{equation*}
\end{thm}

\vspace{5mm}

\appendix
%\noindent {\LARGE Appendix}

\section{Semi-infinite Bruhat graph of type $\C$}
We fix $(A, A_\aff) = (C_n, \C)$.
Throughout Appendix, we omit the dagger for the notation because we does not consider other cases than $(A, A_\aff)= (C_n, \C)$.
We remark that, in this section 
$\QBG$, $\QBG^S$, $\QBG_{b \lambda}^\C$, $(\QBG_{b \lambda}^\C)^S$, $\QLS(\lambda)$
mean
$\QBG\pdag$, $(\QBG\pdag)^S$, $(\QBG\pdag_{b \lambda})^\C$, $((\QBG\pdag_{b \lambda})^\C)^S$, $\QLS^\dag(\lambda)$, respectively,
defined in Definitions \ref{QBG_A}-\ref{QBG_B} and \ref{QLS_C}.

We set 
\begin{equation*}
c_\alpha \eqdef \left\{
\begin{array}{ll}
1 & \mbox{if }\alpha \mbox{ is a intermediate root of }\Delta,\\
2 & \mbox{if }\alpha \mbox{ is a long root of }\Delta,
\end{array}\right.
\end{equation*}
and recall that the set of all real root of $\mathfrak{g}(\C)$ is
\begin{align*}
\Delta_\aff =
\{ \alpha + c_\alpha a \delta \ | \ \alpha \in \Delta , a \in \mathbb{Z} \} 
\sqcup \{ \half(\alpha + (2a-1)\delta ) \ | \ \alpha \in \Delta^\ell, a \in \mathbb{Z} \},
\end{align*}
where $\Delta^\ell$ denotes the set of all long roots in $\Delta$.

For $x \in Q^\lor$,
let $t(x)$ denote the linear transformation on $P_\aff^0$: 
$ t(x) (y+ a \delta ) = y +(a -  \langle y, x \rangle_\aff )\delta$ for $y \in \oplus_{i \in I}\mathbb{Z}(\Lambda_i - \langle \Lambda_i,c \rangle_\aff \Lambda_0) $, $a \in \half \mathbb{Z}$.
The affine Weyl group of $\mathfrak{g}(\C)$
is defined by
$W_\aff \eqdef t(Q^\lor) \rtimes W $.
Also, we define $s_0 : P_\aff^0 \rightarrow P_\aff^0 $ by $x \mapsto s_\theta x - \langle x, \theta^\lor \rangle_\aff \delta$.
Then $W_\aff = \langle s_i \ | \ i \in I_\aff \rangle$.

The affine Weyl group also acts on $P_\aff^\lor$ by
\begin{align*}
s_i y = y - \langle \alpha_i, y \rangle_\aff \alpha_i^\lor
\end{align*}
for $i \in I_\aff$ and $\mu \in P_\aff^\lor$.

For $\beta \in \Delta^+_\aff$ let $w \in W_\aff$ and $i \in I_\aff$ be such that $\beta = w \alpha_i$.
We define the associated reflection $s_\beta \in W_\aff$ and the associated coroot $\beta^\lor \in P_\aff^\lor$ by 
\begin{align*}
s_{\beta} &= w s_i w\inv,
\\
\beta^\lor &= w \alpha_i^\lor.
\end{align*}
Note that for $\Lambda \in P^0_\aff$,
\begin{align}
\langle \Lambda, (\alpha + c_\alpha a \delta)^\lor \rangle_\aff&= \langle  \overline{\xi}\inv \circ \cl (\Lambda), \alpha^\lor \rangle, 
 \ \ \alpha \in \Delta, a \in \mathbb{Z}, \label{aff_paring_1} \\
\langle \Lambda, \left( \half(\alpha + (2 a-1) \delta) \right)^\lor \rangle_\aff&= 2\langle  \overline{\xi}\inv \circ \cl (\Lambda), \alpha^\lor \rangle,
 \ \ \alpha \in \Delta^\ell, a \in \mathbb{Z}. \label{aff_paring_2}
\end{align}
%where $\alpha \in \Delta$ and $a \mathbb{Z}$.

The proof of the following lemma is straightforward.
\begin{lem}\label{affref}
\begin{enu}
\item
Let
$\beta \in \Delta^+_\aff$ be a positive root of the form $\beta = \alpha + a c_\alpha \delta$ with $\alpha \in \Delta$ and $a \in \mathbb{Z}$,
and let $x \in P^0_\aff$. Then, $s_\beta x = s_\alpha x  - a c_\alpha \langle x, \alpha^\lor \rangle \delta= s_\alpha t(a c_\alpha \alpha^\lor) x$.

\item
Let
$\beta \in \Delta^+_\aff$ be a positive root of the form $\beta = \half(\alpha + (2a-1)  \delta)$ with $\alpha \in \Delta$ and $a \in \mathbb{Z}$,
and let $x \in P^0_\aff$. Then, $s_\beta x = s_\alpha x  -(2a-1) \langle x, \alpha^\lor \rangle \delta = s_\alpha t((2a-1)\alpha^\lor) x$.
\end{enu}
\end{lem}

%The affine Weyl group $W_\aff$ acts on
%$P \oplus \mathbb{Z}\delta$ induced by $\xi : P \rightarrow P_\aff^0$
%(as affine transformations):
%for $v\in W$, $t(x) \in t(Q^\lor)$,
%\begin{equation}
%v t(x)( \overline{\beta}+r\delta )=v\overline{\beta}+(r- 2 ( x, \overline{\beta} ) ) \delta,
%\ \ 
%\overline{\beta} \in  \mathfrak{h}^* , r \in \mathbb{C}.
%\end{equation}

\subsection{Peterson's coset representatives $\pcr$}
Let $S$ be a subset of $I$.
We define 
\begin{align*}
\apr &\eqdef
\{ \alpha + c_\alpha a \delta \ | \ \alpha \in \Delta_S, a \in \mathbb{Z} \}
\cup
\{ \half (\alpha + (2 a-1) \delta \ | \ \alpha \in \Delta_S \cap \Delta^\ell , a \in \mathbb{Z} \}, \\
\appr &\eqdef \apr \cap \Delta^+_\aff, \\
\apsg &\eqdef W \ltimes \{ t(\mu) \ | \ \mu \in Q^\lor_S \}, \\
\pcr &\eqdef \{x \in W_\aff \ | \ x \beta \in \Delta^+_\aff \mbox{ for all } \beta \in \appr  \}.
\end{align*}

\begin{lem}[{\cite{Pe}, {see also \cite[Lemma 10.6]{LS10}}}]\label{2.2.2}
For every $x \in W_\aff$, there exists a unique factorization $x = x_1 x_2$
with $x_1 \in \pcr$, and $x_2 \in \apsg$.
\end{lem}

We define a surjective map $\Pi^S : W_\aff \rightarrow \pcr$ by $\Pi^S (x) \eqdef x_1$
if $x = x_1 x_2$ with $x_1 \in \pcr$ and $x_2 \in \apsg$.

\begin{dfn}[{see \cite[Lemma 3.8]{LNSSS1}}]\label{2.2.5}
An element $\mu \in Q^\lor $ is said to be $S$-adjusted if $\langle \mu, \gamma \rangle \in \{ 0, -1 \}$
for all $\gamma \in \Delta^+_S$. 
Let $Q^\lor_{S \adj}$ denote the set of $S$-adjusted elements.
\end{dfn}

Since
$W_\aff$ is isomorphic to the affine Weyl group of $\mathfrak{g}(C_n^{(1)})$,
the proof of the following Lemma is the same as the proof in \cite[Lemma 2.3.5]{INS}.

\begin{lem}[{\cite[Lemma 2.3.5]{INS}, see \cite[(3.6), (3.7), and Lemma 3.7]{LNSSS1}}]\label{2.2.6}
Let $S$ be a subset of $I$. 
\begin{enu}
\item
For each $\mu \in Q^\lor$, there exists a unique $\phi_S (\mu) \in Q_S^\lor$ such that $\mu+ \phi_S(\mu) \in Q^\lor_{S \adj}$.

\item
For each $\mu \in Q^\lor $, there exists a unique $z_\mu \in W_S$ such that $\Pi^S (t (\mu) ) = z_\mu t(\mu+ \phi_S(\mu))$.

\item
For $w \in W$ and $\mu \in Q^\lor $, we have $\Pi^S (w t (\mu) ) = \lfloor w \rfloor  z_\mu t(\mu+ \phi_S(\mu))$.
In particular, 
\begin{align*}
\pcr = \{ w z_\mu t(\mu) \ | \ w \in W^S, \mu \in Q^\lor_{S \adj} \}.
\end{align*}
\end{enu}
\end{lem}

\begin{lem}[{\cite[Lemma 2.2.7]{INS}}]\label{2.2.8}
Let $x \in W_\aff$, and $\mu \in Q^\lor_{S\adj}$.
Then,  $x z_\mu t(\mu) \in  \pcr$ if and only if $x \in \pcr$.
\end{lem}

\subsection{Semi-infinite Bruhat graph of type $\C$}

\begin{dfn}[\cite{Pe}]\label{2.3.1}
Let $x = w t(\mu) \in W_\aff$ with 
$w \in W$ and $\mu \in Q^\lor$.
Then we define 
$\ell^\si(x)  \eqdef \ell (x) + 2\langle \rho, \mu  \rangle$,
where
$\rho = \half\sum_{\alpha \in \Delta^+}\alpha$.
\end{dfn}

\begin{dfn}[{see \cite[Definition 2.4.2]{INS} for untwisted types}]\label{2.3.2}
Let $S$ be a subset of $I$.
The semi-infinite Bruhat graph $\SB^S$ is the directed graph with vertex set $\pcr$
and directed edges labeled by elements in $\Delta^+_\aff$;
for $x\in \pcr$ and $\beta \in \Delta^+_\aff$,
$x \xrightarrow{\beta} s_\beta x$ is an edge of $\SB^S$
if the following  hold:

\begin{enu}
\item
$s_{\beta}x \in \pcr$,

\item
$\ell^\si (s_{\beta}x) = \ell^\si( x) +1 $.
\end{enu}
\end{dfn}

\begin{dfn}[{see \cite[Definition 3.1.1]{INS} for untwisted types}]\label{3.1.1}
Let $\Lambda \in P^0_\aff$ be a level-zero dominant weight, and set
$\lambda = \overline{\xi}\inv \circ \cl (\Lambda) \in P$ and $S = S_\lambda$.
Let $b \in \mathbb{Q}\cap [0,1]$.
We denote by $\SB^S_{b \Lambda}$ the subgraph of $\SB^S$ with the same vertex set but having only the edges:
\begin{equation*}
x \xrightarrow{\beta} s_\beta x \mbox{ with } b\langle x\Lambda, \beta^\lor \rangle_\aff \in \mathbb{Z}.
\end{equation*}
\end{dfn}

\begin{dfn}[{see \cite[Definition 3.1.2]{INS} for untwisted types}]
Let $\Lambda \in P_\aff^0$ be a level-zero dominant  weight,
and set $\lambda = \overline{\xi}\inv \circ \cl (\Lambda)$ and $S = S_\lambda = \{ i \in I \ | \ \langle \lambda , \alpha^\lor_i \rangle =0 \}$.
A pair $\pi = (x_1, x_2 ,\ldots ,x_s ; \tau_0, \tau_1 , \ldots , \tau_s ) $
of a sequence
$w_1,\ldots , w_s$ of  elements in $\pcr$  and a increasing sequence  
$0=\tau_0, \ldots , \tau_s=1$ of rational numbers
is called a semi-infinite Lakshmibai-Seshadri ($\SILS$) path of shape $\Lambda$
if 

(C)
for every $1\leq i \leq s-1$, there exists a directed  path from $w_{i+1}$ to $w_{i}$ in $\SB^S_{b \Lambda}$.

\noindent
Let $\SILS(\Lambda)$ denote the set of all $\SILS$ paths of shape $\Lambda$.
\end{dfn}

Let $\Lambda \in P_\aff^0$ be a level-zero dominant  weight, and set 
$\lambda =  \overline{\xi}\inv \circ \cl (\Lambda)$.
We take $\pi = (x_1, \ldots , x_s ; \sigma_0, \ldots, \sigma_s) \in \SILS(\Lambda)$.
It follows from Lemma \ref{2.2.6} that for each $1 \leq i \leq s$, there exists a unique decomposition $x_i = w_i z_{\mu_i} t(\mu_i)$
with $w_i \in W^S$ and $\mu_i \in Q^\lor_{S\adj}$.
Then we set $\cl (\pi ) = (w_1, \ldots , w_s ; \sigma_0, \ldots, \sigma_s)$.

We will show the following theorem in \S A.3.
\begin{thm}\label{qls_and_sils}
For each $\eta \in \SILS(\Lambda)$, $\cl (\eta ) \in \QLS(\lambda)$.
Hence we obtain the map $\cl : \SILS(\Lambda) \rightarrow \QLS(\lambda)${\rm;}
moreover, the map $\cl :   \SILS(\Lambda) \rightarrow \QLS(\lambda)$ is surjective.
\end{thm}

Let $\Lambda \in P_\aff^0$ be a level-zero dominant  weight.
Let $B (\Lambda)$ be the crystal basis of the extremal weight module $V(\Lambda)$ over $U_q(\mathfrak{g}(A_\aff))$
(see \cite[\S 3.1]{Kas}).
The proof of the following theorem is the same as \cite[Theorem 3.2.1]{INS}.

\begin{thm}[{see \cite[Theorem 3.2.1]{INS} for untwisted types}]
Let $\Lambda \in P$ be a level-zero dominant weight.
Then,
$\SILS(\Lambda)$ has a $U_q(\mathfrak{g}(\C))$-crystal structure.
Moreover,
$B (\Lambda)$ is isomorphic to the set $\SILS(\Lambda)$ as a $U_q(\mathfrak{g}(A_\aff))$-crystal, and the surjective map $\cl :   \SILS(\Lambda) \rightarrow \QLS(\overline{\xi}\inv \circ \cl (\Lambda))$ satisfies the following{\rm:}
for $\eta \in  \SILS(\Lambda)$ and $i \in I_\aff$,
\begin{align*}
\cl (e_i (\eta)) = e_i (\cl (\eta)), \\
\cl (f_i (\eta)) = f_i (\cl (\eta)),
\end{align*}
where $\cl (\zero) = \zero$.
\end{thm}

%We take $\pi = (x_1, \ldots , x_s ; \sigma_0, \ldots, \sigma_s) \in \SILS(\Lambda)$.
%It follows from Lemma \ref{2.2.6} that for each $1 \leq i \leq s$, there exists a unique decomposition $x_i = w_i z_{\mu_i} t(\mu_i)$
%with $w_i \in W^S$ and $\mu \in Q^\lor_{S\adj}$.
%Then we set $\cl (\pi ) = (w_1, \ldots , w_s ; \sigma_0, \ldots, \sigma_s)$.
%
%We will show the following Theorem in \S 6.3.
%\begin{thm}\label{qls_and_sils}
%For each $\eta \in \SILS(\Lambda)$, $\cl (\eta ) \in \QLS(\lambda)$.
%Hence we obtain the map $\cl : \SILS(\Lambda) \rightarrow \QLS(\lambda)$;
%moreover, the map $\cl$ is surjective.
%\end{thm}

\subsection{Proof of Theorem \ref{qls_and_sils}}

\begin{lem}[{\cite[Lemma 4.3]{BFP}}]\label{4.3.4}
We have 
$\ell(s_\alpha) \leq 2 \langle \rho, \alpha^\lor \rangle -1$ for all $\alpha \in \Delta^+$.
\end{lem}

\begin{lem}[{see \cite[Corollary 4.2.2]{INS} for untwisted types}]\label{4.3.5_1}
Let $S$ be a subset of $I$.
Let $x = w z_\mu t(\mu) \in \pcr$ with $w \in W^S$ and $\mu \in Q^\lor_{S\adj}$, 
and let $\beta \in \Delta_\aff^+$ be such that $x \xrightarrow{\beta} s_\beta x$ in $\SB^S$.
If $\beta$ is a long or intermediate root, then we write $\beta$ as $\beta = \gamma + a c_\gamma \delta$ with $\gamma \in \Delta$ and $a \in \mathbb{Z}_{\geq 0}${\rm;}
if $\beta$ is a short root, then we write $\beta$ as $\beta = \half( \gamma + (2a+1) \delta)$ with $\gamma \in \Delta^\ell$ and $a \in \mathbb{Z}_{\geq 0}$.
Set $\alpha \eqdef w^{-1}\gamma \in \Delta$.
Then the following hold.
\begin{enu}
\item
$\alpha \in \Delta^+ \setminus  \Delta^+_S $.

\item
If $\beta$ is a long or intermediate root, then
$\ell(w s_\alpha z_\mu ) = \ell(w z_\mu ) + 1 - 2 c_\alpha a \langle \rho, z_\mu^{-1}\alpha^\lor \rangle$;
if $\beta$ is a short root, then 
$\ell(w s_\alpha z_\mu ) = \ell(w z_\mu ) + 1 - 2 (2a + 1) \langle \rho, z_\mu^{-1}\alpha^\lor \rangle$,

\item
$a \in \{ 0, 1\}$.
Moreover, $\beta$ is a long or short root, then $a = 0$.
\end{enu}
\end{lem}

\begin{proof}
(1)
%As in \cite[Lemma 4.3.5]{Is}.
We first show that $\alpha \notin \Delta_S$. Suppose that $\alpha \in \Delta_S$.
If $\beta$ is a long or intermediate root, then we see that
\begin{equation*}
x\inv \beta = t(-\mu)z_{\mu}\inv w\inv (w \alpha + a c_\gamma  \delta) = t(-\mu)(z_{\mu}\inv \alpha + a c_\alpha \delta) = z_{\mu}\inv \alpha + \zeta c_\alpha \delta
\ \mbox{ for some }
\zeta \in \mathbb{Z};
\end{equation*}
if $\beta$ is a short root, then we see that $x\inv \beta = \half ( z_{\mu}\inv \alpha + \zeta \delta)$ for some odd integer $\zeta$.
Because $\alpha \in \Delta_S$ and $z_\mu \in W_S$, it follows that $x\inv \beta \in  \apr$, and hence $s_{x\inv \beta}\in  \apsg$.
Therefore, we deduce from Lemma \ref{2.2.2} that $s_{\beta}x = x s_{x\inv \beta}$ is not contained in $\pcr$,
which contradicts the assumption that $x \rightarrow s_{\beta}x $ is an edge of $\SB^S$. Thus we obtain $\alpha \notin \Delta_S$.

We set 
\begin{align*}
a' \eqdef \left\{
\begin{array}{ll}
a c_\alpha & \mbox{if }\beta \mbox{ is a long or intermediate root,}\\
2a+1 & \mbox{if }\beta \mbox{ is a short root}.
\end{array}\right.
\end{align*}
It follows from Lemma \ref{affref} that
$s_\beta = s_\gamma t(a' \gamma^\lor)$.
Hence we have 
\begin{align*}
s_\beta x &= s_\gamma t(a' \gamma^\lor) w z_\mu t(\mu) \\
&= w s_{w^{-1}\gamma} z_\mu t(a' z_\mu^{-1} w^{-1} \gamma^\lor +\mu )\\
&= w s_{\alpha}  z_\mu t(a'   z_\mu^{-1} \alpha^\lor +\mu ).
\end{align*}

%We next show that $\alpha \in \Delta^+$. 
Suppose, for a contradiction,  that $-\alpha \in \Delta^+$.
Then, $-z_\mu\inv \alpha \in \Delta^+$ since $\alpha \notin  \Delta_S$.
We have
\begin{align*}
\ell^\si (s_\beta x) &= \ell^\si (w s_{\alpha}  z_\mu t(a'   z_\mu^{-1} \alpha^\lor +\mu) ) & &\\
&= \ell (w s_{\alpha}  z_\mu) + 2\langle \rho,  a'   z_\mu^{-1} \alpha^\lor +\mu  \rangle & &\\
&= \ell (w z_\mu  s_{z\inv_\mu \alpha}) + 2\langle \rho,  a'   z_\mu^{-1} \alpha^\lor +\mu  \rangle  & &\\
&\leq \ell (w z_\mu ) + \ell(  s_{z\inv_\mu \alpha}) + 2\langle \rho,  a'   z_\mu^{-1} \alpha^\lor +\mu  \rangle  & &\\
&\leq \ell (w z_\mu ) + (2\langle \rho , -z\inv_\mu \alpha^\lor \rangle-1) + 2 \langle \rho,  a'   z\inv_\mu \alpha^\lor +\mu  \rangle  &(&\mbox{by Lemma \ref{4.3.4}})\\
&=\ell^\si (x)-1 +2(a'-1)\langle \rho , z_\mu\inv \alpha^\lor \rangle.& &
\end{align*}
Since $\ell^\si (s_\beta x) = \ell^\si (x)+1$, we deduce that $(a'-1)\langle \rho , z_\mu\inv \alpha^\lor \rangle \geq 1$.
Because $\langle \rho , z_\mu\inv \alpha^\lor \rangle <0$, we obtain $a' = 0$, and hence $\beta = w \alpha$.
Recall that $w \alpha \in\Delta^+$ by assumption.
Because $-\alpha \in \Delta^+$ and $w \alpha \in \Delta^+$, we see that $\ell(w s_\alpha) < \ell(w)$.
However, we have
\begin{align*}
\ell^\si (s_\beta x) &= \ell^\si (w s_{\alpha}  z_\mu t(\mu) )  &(&\mbox{since }a'=0 )\\
&= \ell (w s_{\alpha}  z_\mu) + 2\langle \rho, \mu  \rangle & &\\
&\leq \ell (w s_{\alpha} ) + \ell(z_\mu)  + 2\langle \rho,  \mu  \rangle & & \\
&< \ell (w ) + \ell(  z_\mu) + 2\langle \rho,  \mu  \rangle   = \ell^\si (x)&(&\mbox{since } \ell(w s_\alpha) < \ell(w) ),
\end{align*}
which contradicts $\ell^\si (s_\beta x )= \ell^\si(x)+1$.
Thus we obtain $\alpha \in \Delta^+$. This proves (1).

(2)
%Set 
%\begin{align*}
%a' \eqdef \left\{
%\begin{array}{ll}
%a c_\alpha & \mbox{if }\beta \mbox{ is a long or intermediate root,}\\
%2a+1 & \mbox{if }\beta \mbox{ is a short root}.
%\end{array}\right.
%\end{align*}
%It follows from Lemma \ref{affref} that
%$s_\beta = s_\gamma t_{a' \alpha^\lor}$.
%Hence we have 
%\begin{align*}
%s_\beta x &= s_\gamma t_{a' \gamma^\lor} w z_\mu t_\mu \\
%&= w s_{w^{-1}\gamma} z_\mu t_{a' z_\mu^{-1} w^{-1} \gamma^\lor +\mu }\\
%&= w s_{\alpha}  z_\mu t_{a'   z_\mu^{-1} \alpha^\lor +\mu }.
%\end{align*}
%Therefore, it follows that
As in (1), we see that
\begin{align*}
1 &= \ell^\si (s_\beta x) - \ell^\si (x) & & \\
&= \ell(w s_{\alpha}  z_\mu) + 2 \langle \rho, a'   z_\mu^{-1} \alpha^\lor +\mu  \rangle
 - \ell(w z_\mu) - 2 \langle \rho, \mu  \rangle &(&\mbox{by Definition \ref{2.3.1}}) \\
&=\ell(w s_{\alpha}  z_\mu) -\ell(w  z_\mu) + 2 a' \langle \rho,   z_\mu^{-1} \alpha^\lor \rangle, & &
\end{align*}
which proves (2).

(3)
It follows that
\begin{align*}
1 &=\ell(w s_{\alpha}  z_\mu) -\ell(w  z_\mu) + 2 a' \langle \rho,   z_\mu^{-1} \alpha^\lor \rangle &(&\mbox{by (2)}) \\
&=\ell(w   z_\mu s_{z_\mu^{-1} \alpha}) - \ell(w  z_\mu) +2 a' \langle \rho,   z_\mu^{-1} \alpha^\lor \rangle & &\\
&\geq \ell(s_{z_\mu^{-1} \alpha})+2 a' \langle \rho,   z_\mu^{-1} \alpha^\lor \rangle & & \\
&\geq 1 - 2\langle \rho, z_\mu^{-1} \alpha^\lor \rangle +2 a' \langle \rho,   z_\mu^{-1} \alpha^\lor \rangle &(&\mbox{by Lemma \ref{4.3.4}})
\end{align*}
which implies that $a' \in \{ 0,1 \}$, since $z_\mu^{-1} \alpha \in \Delta^+ $.
If $\beta$ is a intermediate root, then $c_\alpha =1$ and hence $a \in \{ 0,1 \}$.
If $\beta$ is a long root, then $c_\alpha =2$ and hence $a =0$.
If $\beta$ is a short root, then $2a+1 \in \{ 0,1 \}$ and hence $a =0$.
\end{proof}

\begin{lem}[{\cite[Lemma 3.10]{LNSSS2}}]\label{4.3.6}
For $\mu \in Q^\lor_{S\adj}$,
we have $\ell(z_\mu) = -2\langle \rho_S, \mu \rangle$.
\end{lem}

\begin{rem}\label{4.2.2}
\
\begin{enu}
\item
If $w \xrightarrow{\alpha} \lfloor w s_\alpha \rfloor$ is a Bruhat edge of $\QBG^S$, then we have $w s_\alpha = \lfloor w s_\alpha \rfloor \in W^S$
(\cite[Remark 3.1.2]{LNSSS4}).

\item
Let $w \xrightarrow{\alpha} \lfloor w s_\alpha \rfloor$ be a quantum edge of $\QBG^S$.
Then we have $w s_\alpha t(\alpha^\lor) \in \pcr$, and $\ell(w s_\alpha) = \ell(w)+ 1 -\langle \rho, \alpha^\lor \rangle$
(\cite[\S 4.3]{LNSSS1}; see also \cite[\S 10]{LS10}).
\end{enu}
\end{rem}

\begin{prop}[{see \cite[Proposition A.1.2]{INS} for untwisted types}]\label{4.3.7}
Let $0 < b \leq 1$ be a rational number.
Let $x = w z_\mu t(\mu) \in \pcr$ with $w \in W^S$ and $\mu \in Q^\lor_{S\adj}$.

\begin{enu}
\item
Assume that $x \xrightarrow{\beta} s_\beta x \in \SB^S_{b \Lambda}$ for $\beta \in \Delta^+_\aff$.
We set $\alpha$ as in Lemma $\ref{4.3.5_1}$.
Then $w \xrightarrow{\alpha} \lfloor w s_\alpha \rfloor $ is an edge of $(\QBG^\C_{b \lambda})^S$.

\item
Assume that 
$w \xrightarrow{\alpha} \lfloor w s_\alpha \rfloor$ is a Bruhat edge of $(\QBG^\C_{b \lambda})^S$.
We set $\beta \eqdef w \alpha \in \Delta^+ \subset \Delta^+_\aff$.
Then  $s_\beta x \in \pcr$.
Moreover, for each $\mu \in Q^\lor_{S\adj}$,
$w z_\mu t(\mu) \xrightarrow{\beta} s_\beta w z_\mu t(\mu)$
is an edge of $\SB^S_{b \Lambda}$.

\item
Assume that
$w \xrightarrow{\alpha} \lfloor w s_\alpha \rfloor$ is a quantum edge of $(\QBG^\C_{b \lambda})^S$ with
$\alpha \in \Delta^+$ a short root.
We set
$\beta \eqdef w \alpha + \delta \in  \Delta^+_\aff$.
Then
$s_\beta x \in \pcr$.
Moreover, for each $\mu \in Q^\lor_{S\adj}$,
$w z_\mu t(\mu) \xrightarrow{\beta} s_\beta w z_\mu t(\mu)$
is an edge of $\SB^S_{b \Lambda}$.

\item
Assume that
$w \xrightarrow{\alpha} \lfloor w s_\alpha \rfloor$ is a quantum edge of $(\QBG^\C_{b \lambda})^S$ with
$\alpha \in \Delta^+$ a long root.
We set
$\beta \eqdef \half(w \alpha + \delta) \in  \Delta^+_\aff$.
Then
$s_\beta x \in \pcr$.
Moreover, for each $\mu \in Q^\lor_{S\adj}$,
$w z_\mu t(\mu) \xrightarrow{\beta} s_\beta w z_\mu t(\mu)$
is an edge of $\SB^S_{b \Lambda}$.
\end{enu}
\end{prop}

\begin{proof}
(1)
By Lemma \ref{4.3.5_1} (1), (3),
the affine root $\beta$ is of the form:

(i)
(long or intermediate root)
$\beta = w \alpha$ with $\alpha \in \Delta^+$, or

(ii)
(intermediate root)
$\beta = w \alpha + \delta$ with $\alpha \in \Delta^+$ a short root, or

(iii)
(short root)
$\beta = \half(w \alpha + \delta )$ with $\alpha \in \Delta^+$ a long root.

Consider case (i).
Since $\pcr \ni s_\beta x = s_\beta w z_\mu t(\mu) = w s_\alpha z_\mu t(\mu)$,
we have
$w s_\alpha \in W^S$, and hence
$\ell( w s_\alpha z_\mu) = \ell( w s_\alpha ) + \ell(z_\mu)$.
Also, since $\ell(w z_\mu )= \ell(w ) + \ell(z_\mu)$, 
$\ell(w s_\alpha) = \ell(w) + 1$ because $\ell^\si (s_\beta x) = \ell^\si (x) +1$.
Hence $w \xrightarrow{\alpha} \lfloor w s_\alpha \rfloor = w s_\alpha$ is a Bruhat edge of $\QBG$.
Moreover,
$w \xrightarrow{\alpha} \lfloor w s_\alpha \rfloor$ is an edge of $(\QBG_{b \lambda}^\C)^S$
since 
$\langle \lambda , \alpha^\lor \rangle =\langle z_\mu t(\mu) \Lambda, w\inv \beta^\lor \rangle_\aff =\langle x \Lambda, \beta^\lor \rangle_\aff \in \mathbb{Z}$ by (\ref{aff_paring_1}).

Consider cases (ii) and (iii).
Since 
$\pcr \ni s_\beta x = w s_\alpha z_\mu 
t(z_\mu^{-1}\alpha^\lor + \mu)$,
we obtain $z_\mu^{-1}\alpha^\lor + \mu \in Q^\lor_{S\adj}$
by Lemma \ref{2.2.6} (3).
We have
\begin{align*}
 w s_\alpha z_\mu 
t(z_\mu^{-1}\alpha^\lor + \mu)
&= \Pi^S (w s_\alpha z_\mu 
t(z_\mu^{-1}\alpha^\lor + \mu))
\\ &=
\lfloor w s_\alpha \rfloor z_{z_\mu^{-1}\alpha^\lor + \mu}
t(z_\mu^{-1}\alpha^\lor + \mu) &(& \mbox{by Lemma \ref{2.2.6} (3)}),
 \end{align*}
 which implies that $ w s_\alpha z_\mu  = \lfloor w s_\alpha \rfloor z_{z_\mu^{-1}\alpha^\lor + \mu}$,
and hence 
$\ell(w s_\alpha z_\mu) = \ell(\lfloor w s_\alpha \rfloor) +
\ell(z_{z_\mu^{-1}\alpha^\lor + \mu})$.
Therefore,
\begin{align*}
\ell(\lfloor w s_\alpha \rfloor)
&= \ell(w s_\alpha z_\mu) - \ell(z_{z_\mu^{-1}\alpha^\lor + \mu}) & &\\
&=\ell(w  z_\mu) + 1
-2\langle \rho, \alpha^\lor \rangle - \ell(z_{z_\mu^{-1}\alpha^\lor + \mu})&(&
\mbox{by Lemma \ref{4.3.5_1} (2))}
\\
&=\ell(w) + \ell(z_\mu) + 1
-2\langle \rho, \alpha^\lor \rangle - \ell(z_{z_\mu^{-1}\alpha^\lor + \mu})& &\\
&=\ell(w) - 2\langle \rho_S, \mu \rangle + 1
-2\langle \rho, \alpha^\lor \rangle +
2\langle \rho_S, z_\mu^{-1}\alpha^\lor + \mu \rangle
&(&\mbox{by Lemma \ref{4.3.6}})\\
&=\ell(w)  + 1
-2\langle \rho, \alpha^\lor \rangle +
2\langle \rho_S, z_\mu^{-1}\alpha^\lor \rangle
& &\\
&=
\ell(w)  + 1
-2\langle \rho, \alpha^\lor \rangle +
2\langle \rho_S, \alpha^\lor \rangle
&(&\mbox{since }z_\mu \in W_S)\\
&=
\ell(w)  + 1
-2\langle \rho- \rho_S, \alpha^\lor \rangle, & & 
\end{align*}
which implies that $w \xrightarrow{\alpha} \lfloor w s_\alpha \rfloor$ is a quantum edge of $\QBG$.
Moreover, if
$\beta = \gamma+ \delta$,
then
$\langle \lambda, \alpha^\lor \rangle=  \langle x \Lambda ,  \beta^\lor \rangle_\aff  \in \mathbb{Z}$ by (\ref{aff_paring_1});
if $\beta = \half (\gamma + \delta)$, then
$\langle \lambda, \alpha^\lor \rangle =  \half \langle x \Lambda ,  \beta^\lor \rangle_\aff \in \half \mathbb{Z}$ by (\ref{aff_paring_2}).
Therefore, in both cases, we have
$w \xrightarrow{\alpha} \lfloor w s_\alpha \rfloor$ is a quantum edge of $(\QBG_{b \lambda}^\C)^S$.

(2)
%Set $x = w z_\mu t_\xi$.
%We see that $\beta = w \alpha \in \Delta^+ (C_n) \subset 
%\Delta^+ (A^2_{2n})$.
Since $s_\beta x = w s_\alpha z_\mu t(\mu)$ and $w s_\alpha = \lfloor w s_\alpha \rfloor  \in W^S$ by Remark \ref{4.2.2} (1),
it follows 
from Lemma \ref{2.2.6} (3)
that
$s_\beta x \in \pcr$.
Also,
we have
\begin{align*}
\ell^\si (s_\beta x) &=
\ell (w s_\alpha z_\mu) + 2\langle \rho, \mu \rangle & &\\
&=
\ell (w s_\alpha ) + \ell(z_\mu) +  2\langle \rho, \mu \rangle &(&\mbox{since }w s_\alpha \in W^S
\mbox{ and }z_\mu \in W_S \\
&=
\ell (w) + 1 + \ell(z_\mu)  + 2\langle \rho, \mu \rangle 
&(& \mbox{since }w\xrightarrow{\alpha}w s_\alpha
\mbox{ is a Bruhat edge}) 
\\
&=
\ell (w z_\mu) + 2\langle \rho, \mu \rangle + 1 
&(&\mbox{since }w \in W^S
\mbox{ and }z_\mu \in W_S
\\
&=
\ell^\si (x) +1
\end{align*}
Moreover,
$ 
b \langle \lambda, \alpha^\lor \rangle =b \langle x\Lambda, \beta^\lor \rangle_\aff \in \mathbb{Z}$ by (\ref{aff_paring_1}).
Therefore,
$x \xrightarrow{\beta} s_\beta x$ is an edge of $\SB^S_{b \Lambda}$.

(3), (4)
%Set $x = w z_\xi t_\xi$.
%It is clear that $\beta \in \Delta^+(A^2_{2n})$.
Since $w s_\alpha t(\alpha^\lor) \in \pcr$
by Remark \ref{4.2.2} (2),
it follows from Lemma \ref{2.2.8} that 
$s_\beta x = w s_\alpha t(\alpha^\lor) z_\mu t(\mu) \in \pcr$.
Because
\begin{align*}
s_\beta x &= \Pi^S (s_\beta x)   &(& \mbox{since }s_\beta x \in \pcr )\\
&=
\Pi^S (w s_\alpha z_\mu t(\mu + z_\mu^{-1}\alpha^\lor ))
& &
\\
&=
\lfloor w s_\alpha \rfloor z_{\mu + z_\mu^{-1}\alpha^\lor } t(\mu + z_\mu^{-1}\alpha^\lor )
&(& \mbox{by Lemma \ref{2.2.6} (3)}),
\end{align*}
we deduce that
\begin{align*}
\ell^\si (s_\beta x) &= \ell^\si (\lfloor w s_\alpha \rfloor z_{\mu + z_\mu^{-1}\alpha^\lor } t(\mu + z_\mu^{-1}\alpha^\lor ) ) \\
&=
\ell (\lfloor w s_\alpha \rfloor z_{\mu + z_\mu^{-1}\alpha^\lor} )
+
2\langle \rho , \mu + z_\mu^{-1}\alpha^\lor \rangle  \\
&=
\ell (\lfloor w s_\alpha \rfloor) + \ell( z_{\mu + z_\mu^{-1}\alpha^\lor }) + 2\langle \rho , \mu + z_\mu^{-1}\alpha^\lor \rangle  \\
&=
\ell(w)+1 -2\langle \rho- \rho_S, \alpha^\lor \rangle
+ \ell( z_{\mu + z_\mu^{-1}\alpha^\lor }) + 2 \langle \rho , \mu + z_\mu^{-1}\alpha^\lor \rangle \\
& \hspace{50mm} (\mbox{since }w \xrightarrow{\alpha} \lfloor w s_\alpha \rfloor 
\mbox{ is a quantum edge}) \\
&\overset{{\rm Lem. \ \ref{4.3.6}}}{=}
\ell(w)+1 -2\langle \rho- \rho_S, \alpha^\lor \rangle
-
2\langle \rho_S, \mu + z_\mu^{-1}\alpha^\lor \rangle + 2\langle \rho , \mu + z_\mu^{-1}\alpha^\lor \rangle \\
&=
\ell(w)+1 -2\langle \rho- \rho_S, z_\mu^{-1} \alpha^\lor \rangle
+
2\langle \rho - \rho_S, \mu + z_\mu^{-1}\alpha^\lor \rangle 
\ \ (\mbox{since }z_\mu \in W_S) \\
&=
\ell(w)+1 +
2\langle \rho - \rho_S, \mu \rangle \\
&\overset{{\rm Lem. \ \ref{4.3.6}}}{=}
\ell(w)+ 1 +
2\langle \rho, \mu \rangle + \ell(z_\mu) \\
&=
\ell^\si (x)+1.
\end{align*}
Thus, 
$x \xrightarrow{\beta} s_\beta x $ is an edge of $\SB^S$.
Moreover, in case (3), we have
$b\langle x\Lambda, \beta^\lor \rangle_\aff = 
b\langle \lambda, \alpha^\lor \rangle 
\in \mathbb{Z}$ by (\ref{aff_paring_1});
in case (4), we have
$b\langle x\Lambda, \beta^\lor \rangle_\aff = 
2b\langle \lambda, \alpha^\lor \rangle 
\in 2 \cdot \half \mathbb{Z} = \mathbb{Z}$ by (\ref{aff_paring_2}).
Therefore, in both cases,
$x \xrightarrow{\beta} s_\beta x$ is an edge of $\SB^S_{b \Lambda}$.
\end{proof}

\begin{proof}[Proof of Theorem $\ref{qls_and_sils}$]
Let $\pi = (x_1, \ldots, x_s ; \sigma_0, \ldots ,\sigma_s) \in \SILS(\Lambda)$.
Let $w_i \in W^S$ and $\mu_i \in Q^\lor_{S\adj}$ be such that $x_i = w_i z_{\mu_i} t(\mu_i)$ for $1 \leq i \leq s$.
If there exists a directed path from
$x_{i+1}$ to $x_i$ in $\SB^S_{b \Lambda}$, then 
there exists a directed path from
$w_{i+1}$ to $w_i$ in $(\QBG^\C_{b \lambda})^S$ by Lemma \ref{4.3.7}(1).
This implies $\cl(\pi) = (w_1, \ldots, w_s ; \sigma_0, \ldots ,\sigma_s) \in \QLS(\lambda)$.

Let $\eta = (w_1, \ldots, w_s ; \sigma_0, \ldots ,\sigma_s) \in \QLS(\lambda)$, and set $\mu_s = 0$.
We proceed by descending induction on $1 \leq p \leq s$.
Assume that for every $i+1 \leq p \leq s$ there exist an element $\mu_p \in Q^\lor_{S\adj}$ and 
a directed path from $w_{p+1} z_{\mu_{p+1}} t(\mu_{p+1})$ to $w_p z_{\mu_p} t(\mu_p)$ in $\SB^S_{b \Lambda}$.
By the definition of $\QLS(\lambda)$,
There exists a directed path from $w_{i+1}$ to $w_i$ in $(\QBG^\C_{b \lambda})^S$.
It follows from Lemmas \ref{2.2.6} and \ref{4.3.7} that 
there exist
an element $\mu_i \in Q^\lor_{S\adj}$ and 
a directed path from $w_{i+1} z_{\mu_{i+1}} t(\mu_{i+1})$ to $w_i z_{\mu_i} t(\mu_i)$ in $\SB^S_{b \Lambda}$.
Thus we obtain $\pi \eqdef (w_1 z_{\mu_{1}} t(\mu_{1}), \ldots , w_s z_{\mu_{s}} t(\mu_{s}); \sigma_0, \ldots , \sigma_s) \in \SILS(\Lambda)$.
It is obvious that $\cl (\pi) = (w_1, \ldots, w_s ; \sigma_0, \ldots ,\sigma_s) $, which implies that the map $\cl : \SILS(\Lambda)\rightarrow \QLS(\lambda)$ is surjective.
\end{proof}

\appendix

\end{document}